\documentclass[11pt,reqno]{article}

\usepackage[margin=0.8in]{geometry}
\usepackage{amsthm,amsmath,epsfig,amsfonts,amssymb,bm,xcolor,amsbsy,times, dsfont, amssymb,amsmath,amscd,latexsym, amsthm, amsxtra,amsfonts}
\usepackage{bm}
\usepackage{bbm}
\usepackage{enumerate}
\usepackage{mathrsfs}
\usepackage{graphicx}
\usepackage{subfig}
\usepackage{comment}
\usepackage{mathtools}
\usepackage{cases}
\usepackage{dsfont, amssymb,amsmath,amscd,latexsym, amsxtra,amsfonts,amsthm}
\numberwithin{equation}{section}
\usepackage{ifpdf}
\usepackage{color}
\usepackage[title]{appendix}
\definecolor{webgreen}{rgb}{0,.5,0}
\definecolor{webbrown}{rgb}{.8,0,0}
\definecolor{emphcolor}{rgb}{0.95,0.95,0.95}

\usepackage{hyperref}
\hypersetup{hidelinks,
	colorlinks=true,
	linkcolor=blue,
	filecolor=blue,
	citecolor=blue,
	urlcolor=blue,
	breaklinks=true}

\newtheorem{theorem}{Theorem}[section]
\newtheorem{lemma}[theorem]{Lemma}
\newtheorem{proposition}[theorem]{Proposition}

\newtheorem{assumption}{Assumption}[section]
\newtheorem{definition}[theorem]{Definition}

\theoremstyle{definition}

\newtheorem{remark}[theorem]{Remark}
\numberwithin{equation}{section}

\newcommand{\nc}{\newcommand}
\nc{\ind}{\mathds{1}}

\newcommand{\R}{\mathbb{R}}

\newcommand{\E}{\mathcal{E}}
\newcommand{\F}{\mathcal{F}}

\newcommand{\T}{\mathcal{T}}

\newcommand{\talpha}{\tilde{\alpha}}

\allowdisplaybreaks[4]

\DeclareMathOperator{\esssup}{esssup}

\def\esssup_#1{\underset{#1}{\mathrm{ess\,sup\, }}}
\def\essinf_#1{\underset{#1}{\mathrm{ess\,inf\, }}}
\def\argmax_#1{\underset{#1}{\mathrm{arg\,max\, }}}
\def\argmin_#1{\underset{#1}{\mathrm{arg\,min\, }}}

\def\b1{\bf 1}

\def \R{\mathbb{R}}

\def \E{\mathbb{E}}
\def \F{\mathbb{F}}

\def \P{\mathbb{P}}

\def \S{\mathbb{S}}
\def \T{\mathbb{T}}

\def \Ac{{\cal A}}
\def \Bc{{\cal B}}
\def \Cc{{\cal C}}
\def \Dc{{\cal D}}
\def \Ec{{\cal E}}
\def \Fc{{\cal F}}
\def \Gc{{\cal G}}

\def \Kc{{\cal K}}
\def \Lc{{\cal L}}
\def \Pc{{\cal P}}
\def \Mc{{\cal M}}

\def \Tc{{\cal T}}

\def\beqs{\begin{eqnarray*}}
	\def\enqs{\end{eqnarray*}}
\def\beq{\begin{eqnarray}}
	\def\enq{\end{eqnarray}}

\title{On time-inconsistent extended mean-field control problems with common noise}
\author{Zongxia Liang\thanks{Department of Mathematical Sciences, Tsinghua University, Beijing, China. \url{ liangzongxia@tsinghua.edu.cn}}
\and Xiang Yu\thanks{Department of Applied Mathematics,  The Hong Kong Polytechnic University, Kowloon, Hong Kong. \url{xiang.yu@polyu.edu.hk}}
	\and Keyu Zhang\thanks{
	Department of Mathematical Sciences, Tsinghua University, Beijing , China. \url{Zhangky21@mails.tsinghua.edu.cn}}}
\date{}

\begin{document}
	
	\maketitle
 \vspace{-0.2in}
	\begin{abstract}
		This paper studies a class of time-inconsistent mean field control (MFC) problems in the presence of common noise under non-exponential discount and joint law dependence of both state and control. {We investigate the closed-loop time-consistent equilibrium strategies for these extended MFC problems and characterize them through an equilibrium Hamilton-Jacobi-Bellman (HJB) equation defined on the Wasserstein space.}  We first apply the results to the linear-quadratic (LQ) time-inconsistent MFC problems and obtain the existence of time-consistent equilibria via a comprehensive study of a nonlocal Riccati system. To illustrate the theoretical findings, two financial applications are presented. We then examine a class of non-LQ time-inconsistent MFC problems, for which we contribute the existence of time-consistent equilibria by analyzing a nonlocal nonlinear partial differential equation.
		
	\vspace{0.05in}	
 
	\end{abstract}
	
	\section{Introduction}

         Mean-field control (MFC) problem, or optimal control of McKean-Vlasov dynamics, has received a surge of attention in recent years due to its wide applications and new advances in mean field theory. Comparing with mean-field game problem that models the competitive interactions of agents, the MFC problem considers the
cooperative interactions where all agents jointly optimize the social optimum that leads to the
optimal control by the social planner.
         In the existing literature, two main approaches have been used to tackle MFC problems in various context, namely, the probabilistic approach based on stochastic maximum principle and the PDE approach on the Wasserstein space on strength of dynamic programming principle. 
         By employing the stochastic maximum {principle} and the adjoint 
         backward stochastic differential equation (BSDE) coupled with a forward SDE, fruitful studies can be found among \cite{ABC17, Andersson2011,Buckdahn2011, Li12,Carmona2015, BLPR17, BLLX23} for controlled McKean-Vlasov dynamics without common noise and also in 
        \cite{Buckdahn2017,Carmona2016, BWWY} for conditional McKean–Vlasov dynamics in the presence of common noise. Along the direction of PDE approach for open-loop and closed-loop controls and the study of viscosity solution on the Wasserstein space, we refer some recent developments in \cite{Bensoussan2015,Lauriere2016, Pham2016, BCP18, Pham2018, WZ20, ZTZ24}.

        All aforementioned studies on MFC problems are time-consistent, i.e., the optimal control determined today remains optimal in the future. However, in a real-life situation, it often occurs that an agent may deviate from the optimal decision deemed previously, leading to \textit{time-inconsistent} behavior. For instance, when the cost function involves non-exponential discount, the decision making is inherently time-inconsistent. In such case, it is not meaningful to determine an initial optimal control that may not be obeyed at future dates. To address the issue of time inconsistency, the intra-personal equilibrium approach under time-consistent planning is commonly employed; see \cite{STR}. This approach mandates the agent to take into account the behavior of future selves and to seek the control strategy by game-theoretic thinking, thereby ensuring the time-consistency of the chosen strategy; see, e.g., \cite{Ekeland2008,Ekeland2012,Yong2012,Hu2012,Bjork2017,He2021,CCPW,LM23} and references therein for single agent's time-inconsistent control problems in the continuous-time framework.

        In the present paper, we are interested in the \textit{extended} MFC problems where the mean-field interactions occur via the joint conditional law of both state and control in the presence of common noise. In particular, we embrace the time-inconsistency issue caused by non-exponential discount, and we are devoted to the characterization and the verification of closed-loop time-consistent equilibrium strategies in the mean field model. Notably, previous studies on time-inconsistent MFC problems, such as \cite{Ni2017,Yong2017,Ni2019,Wang2019}, only focus on some special cases in linear-quadratic (LQ) MFC problems with marginal state law dependence and without common noise. In a more general model setup, \cite{zhuchao2020} studies some time-inconsistent MFG problems. Recently, \cite{YY23} investigates a time-inconsistent mean-field Markov decision process with stopping but in a discrete-time setting. To the best of our knowledge, our work is the first study that provides the characterization of closed-loop equilibrium strategies for time inconsistent MFC probelms in a general framework, where the McKean-Vlasov SDE involves the joint conditional law dependence and common noise.        
        
     More precisely, we first transform the original problem into a stochastic control problem where the conditional law given the common noise is regarded as the state variable (see \eqref{formulationcon}). {Based on the game-theoretic thinking with future selves, we provide a necessary and sufficient condition for a closed-loop  equilibrium strategy satisfying the time-consistent planning in the mean field setting; see Proposition \ref{Theorem:charactization}}. As a result of It\^{o}'s formula along the flow of conditional probability measures developed in \cite{Buckdahn2017}, an { equilibrium Hamilton-Jacobi-Bellman (HJB) equation} defined on the Wasserstein space of probability measures
        can be derived that gives the equivalent characterization of the closed-loop equilibrium and the resulting value function; see the { equilibrium HJB equation} \eqref{masterequation:valuefunction}-\eqref{masterequation:equilibrium} and the Verification Theorem \ref{corollary:verification}.
We first apply the equilibrium HJB characterization to time-inconsistent LQ extended MFC problems and reduce the problem to a nonlocal Riccati equation system \eqref{riccati1}-\eqref{linearode2}. {Specifically, the nonlocal feature of \eqref{riccati1}-\eqref{linearode2} appears in the diagonal dependence by \((\tau, t)\), i.e., the system depends on the value at \((t, t)\). Alternatively, one may treat \(\tau\) as a parameter and interpret \eqref{riccati1}-\eqref{linearode2} as an infinite-dimensional system of matrix-valued ordinary differential equations (ODEs), which exhibits self-interaction through the diagonal term.   A key theoretical contribution of this paper is the establishment of the well-posedness of the Riccati equation system in the mean-field framework; see Proposition \ref{Pro:riccati}. As the system is decoupled, we can solve the equations sequentially. For the first two equations, \eqref{riccati1}-\eqref{riccati2}, we employ a convergence method to establish existence. Specifically, we introduce a sequence of Riccati equations and Lyapunov equations, whose solutions are then used to construct functions that converge to the diagonal and non-diagonal values of the solution to the nonlocal system, respectively. The uniqueness of \eqref{riccati1}-\eqref{riccati2} is established via a contradiction argument. For the third equation, \eqref{linearode1}, we exploit its linearity and apply standard contraction arguments to ensure the well-posedness.  The solution to \eqref{linearode2} follows directly by integration. Thanks to Verification Theorem \ref{corollary:verification}, the characterization of a closed-loop equilibrium is established therein.} In two examples of financial applications, namely, the conditional mean-variance portfolio selection and an inter-bank systemic risk model with common noise, we illustrate our theoretical results and discuss some financial implications of the obtained equilibrium strategies. { Finally, a class of time-inconsistent non-LQ extended MFC problems is also investigated, in which we can utilize the conjectured form of the solution in
\eqref{newsect:ansatz} to reduce the equilibrium HJB on Wasserstein space to  a nonlocal nonlinear PDE in \eqref{newsect:nonlocalpde}. Although nonlocal nonlinear PDE has been studied previously in \cite{wei_time-inconsistent_2017, zhuchao2020, wang_backward_2019, lei_nonlocal_2023, lei_nonlocality_2024}, the global well-posedness remains an open problem. Existing studies, such as \cite{wei_time-inconsistent_2017, zhuchao2020, wang_backward_2019, lei_nonlocality_2024}, typically impose strong conditions on the nonlinear operator, leaving a large class of problems such as our case unaddressed. In this work, we employ a fixed-point approach and exploit the special structure of \eqref{newsect:nonlocalpde}. Under a smallness condition on model coefficients, we prove that the nonlocal PDE \eqref{newsect:nonlocalpde} has a unique classical solution in the global sense and then deduce the characterization of a closed-loop equilibrium for the time-inconsistent non-LQ MFC problem  with the aid of the Verification Theorem \ref{corollary:verification}; see Proposition \ref{newsect:lemma}.}

        The rest of the paper is organized as follows.  Section \ref{sec:pro:formulation} introduces the notations and the model setup of general time-inconsistent MFC problems and the definition of closed-loop equilibrium strategies.  {Section \ref{sec:characterization} presents the equilibrium HJB equation and a verification theorem to characterizes the closed-loop time-consistent equilibria.} Section \ref{sec:app} focuses on the LQ-type time-inconsistent MFC problems and contributes the characterization of the time-consistent equilibria via the study of a nonlocal Riccati system. {Section \ref{sect:nonlq} examines a class of non-LQ MFC problems by analyzing the associated nonlocal nonlinear PDE problem.} Finally, Section \ref{auxproofs} collects technical proofs of some main results in previous sections.

\section{Problem Formulation}\label{sec:pro:formulation}

 \subsection{Notations and preliminaries}\label{sec:notations}
\paragraph{Sets and Functions}		
		  We use $A\subseteq\R^m$ to denote the control space. 
    For any matrix or vector $M$, we denote its transpose by $M^{\top}$ and its trace by $\mathrm{tr}(M)$.  $\S^{d}$ denotes the set of symmetric matrices in $\R^{d\times d}$,  $\S^{d}_{+}$ denotes the set of nonnegative define symmetric matrices in $\R^{d\times d }$ and $\S^{d}_{>+}$ the set of positive definite symmetric matrices in $\R^{d\times d }$. For $A,B\in \S^{d}_{+}$, $A\leq B$ means that $B-A\in\S^{d}_{+}$. We also use the notation $\Delta[a,b]:=\left\{(t,s)|a\leq t\leq s\leq b\right\}$. We denote $\mathscr C_{_2}(\R^d;\R)$ as a subset of  $C(\R^d;\R)$, representing the set of continuous  functions on $\R^d$ with quadratic growth. 
 For two Euclidian spaces $E$ and $F$, we denote by $L(E;F)$ the set of measurable functions $\phi:E\rightarrow F$, and denote by $Lip(E;F) \subseteq L(E;F)$ the set of Lipschitz functions $\phi:E\rightarrow F$.  

\paragraph{Integration and probability}		
Consider a complete probability space $(\Omega,\Fc,\P)=(\Omega^0\times\Omega^1,\Fc^0\otimes\Fc^1,\P^0\otimes\P^1)$, where $(\Omega^0,\Fc^0,\P^0)$ supports a $k$-dimensional Brownian motion $W^0$, and  $(\Omega^1,\Fc^1,\P^1)$ 
supports a $n$-dimensional Brownian motion $B$. An element $\omega$ $\in$ $\Omega$ is written as 
$\omega$ $=$ $(\omega^0,\omega^1)$ $\in$ $\Omega^0\times\Omega^1$, and we extend canonically $W^0$ and $W$ on $\Omega$ by setting 
$W^0(\omega^0,\omega^1)$ $:=$ $W^0(\omega^0)$, $W(\omega^0,\omega^1)$ $:=$ $W(\omega^1)$, and extend similarly on $\Omega$ 
any random variable on  $\Omega^0$ or $\Omega^1$.  We further assume that $(\Omega^1,\Fc^1,\P^1)$ satisfies $\Omega^1$ $=$ $\tilde\Omega^1\times\Omega^{'1}$, 
$\Fc^1$ $=$ $\Gc\otimes\Fc^{'1}$, $\P^1$ $=$ $\tilde\P^1\otimes\P^{'1}$, where $\tilde\Omega^1$ is a Polish space, $\Gc$ its Borel $\sigma$-algebra,  $\tilde\P^1$ an atomless probability  measure on 
$(\tilde\Omega^1,\Gc)$, while $(\Omega^{'1},\Fc^{'1},\P^{'1})$ supports $B$.    
We denote by  $\E^0$ (resp. $\E^1$ and $\tilde\E^1$)  the expectation under $\P^0$ (resp. $\P^1$ and $\tilde\P^1$),  by $\F^0$ $=$  
$(\Fc_t^0)_{t\geq 0}$ the $\P^0$-completion of the natural filtration generated by $W^0$ 
(and w.l.o.g. we assume that $\Fc^0$ $=$ $\Fc^0_\infty$). 
Let 
$\F$ $=$ $(\Fc_t)_{t\geq 0}$ denote the natural filtration generated by $W^0,B$, augmented with the independent $\sigma$-algebra $\Gc$.  Given a normed space $(E, \Vert\cdot\Vert_{E})$, we denote by $\Pc_{2}(E)$ the set of square integrable
probability measures $\mu$ on $E$ satisfying
$M^{2}_{2}(\mu):=\int_{E}\Vert x\Vert^{2}_{E}\mu(dx)<\infty$. For notational simplicity, we sometimes use $|\cdot|$ in place of $\Vert\cdot\Vert_{E}$. 
 For any $\mu\in\Pc_{2}(E)$, $F$ Euclidian space, we denote by $L^{2}_{\mu}(F)$ the set of measurable functions $\phi: E\rightarrow F$ which are square integrable with respect to $\mu$, by the set
 $L_{\mu\otimes\mu}^2(F)$ the set of measurable functions $\psi$ $:$ 
 $E\times E$ $\rightarrow$ $F$,  which are square integrable with respect to the product measure  $\mu\otimes\mu$, and  set
 \begin{equation*}
 	\langle \phi,\mu\rangle:=\int_{E}\phi(x)\mu(dx),\quad 	\langle \psi,\mu\otimes\mu\rangle:=\int_{E\times E} \psi(x,x') \mu(dx)\mu(dx').
 \end{equation*}
We also define $L_\mu^\infty(F)$ (resp.  $L_{\mu\otimes\mu}^\infty(F)$) as the subset of elements $\varphi$ $\in$ $L_\mu^2(F)$ (resp.  $L_{\mu\otimes\mu}^2(F)$) that are bounded 
$\mu$ (resp. $\mu\otimes\mu$) a.e., and $\|\varphi\|_\infty$ is their essential supremum. 
We denote  by $L^2(\Gc;E)$ (resp. $L^2(\Fc_t;E)$)  the set of $E$-valued  square integrable random variables on $(\tilde\Omega^1,\Gc,\tilde\P^1)$ (resp. on $(\Omega,\Fc_t,\P)$). 
For any random variable $X$ on $(\Omega,\Fc,\P)$, we denote by $\P_X$  its distribution under  $\P$, and by $\delta_{X}$ the Dirac measure on $X$.  
  
Note that  $\Pc_{_2}(\R^d)$ is a metric space equipped with the $2$-Wasserstein distance
\begin{equation*}
	\begin{split}
			W_{2}(\mu,\mu'):&=\inf\left\{\left(\int_{E\times E} |x-y|^{2}\pi(dx,dy)\right)^{\frac{1}{2}}: \pi\in\Pc_{2}(E\times E) \text{with marginals}\, \mu \,\text{and}\, \mu' \right\}\\
			&=\inf\left\{(\E|\xi-\xi'|^{2})^{\frac{1}{2}}:\, \xi,\xi'\in L^{2}(\Gc;E) \, \text{with} \, \P_{\xi}=\mu,\, \P_{\xi'}=\mu' \right\}
	\end{split}
\end{equation*}
and endowed with the Borel $\sigma$-field $\Bc(\Pc_{_2}(\R^d))$. 
It is easy to see that 
	\begin{equation}\label{winequality}
		W^{2}_{2}(\P_{\xi},\P_{\xi'})\leq \E[|\xi-\xi'|^{2}].
	\end{equation}
	Let $L^{2}_{\F}(\Omega;C([t_{0},T]);\R^{d}))$ denote the set of $\F$-progressively measurable and continuous processes with $$\E[\sup_{t_{0}\leq t\leq T}|X_{t}|^{2}]<\infty.$$ 
	Finally, we recall the next measurability result given in Remark 2.1 of \cite{Pham2016}.
\begin{proposition}\label{pro:measurability}
	 Given a measurable space $(E,\Ec)$ and a map $\rho$ $:$ $E$ $\rightarrow$ $\Pc_{_2}(\R^d)$, $\rho$ is measurable if and only if  the map  $\langle\varphi,\rho\rangle$ $:$ $E$ $\rightarrow$ $\R$  is measurable, for any $\varphi\in\mathscr C_{_2}(\R^d;\R)$. 
\end{proposition}
\paragraph{Differentiability on Wassersetin space}
Let us recall the notion of derivative with respect to a probability measure, as introduced by P.L. Lions in the lecture notes \cite{Pierre2013}. Lion's notion is based on the lifting of functions $u:\Pc_{2}(\R^{d})\rightarrow\R$ into functions $U$ defined on $L^2(\Gc;\R^d)$ ($=$ $L^2(\tilde\Omega^1,\Gc,\tilde\P^1;\R^d)$) by $U(X)=u(\P_{X})$. We say that $u$ is differentiable (resp. $C^{1}$ ) on $\Pc_{2}(\R^{d})$ if the lift $U$ is Fréchet differentiable (resp. Fréchet differentiable with continuous derivatives) on $L^2(\Gc;\R^d)$, and the Fréchet derivative is defined by 
$\partial_{\mu}u(\P_{X})(X)$. Moreover, $\partial_{\mu} u(\mu)\in L^{2}_{\mu}(\R^{d})$ for $\mu$ $\in$ $\Pc_{_2}(\R^d)$ $=$ $\{\P_{\xi}:  \xi  \in L^2(\Gc;\R^d)\}$. 
 
  We say that $u$ is fully  $\Cc^2$ if it is $\Cc^1$, 
 and one can find, for any $\mu$ $\in$ $\Pc_{_2}(\R^d)$, a continuous version of the mapping $x\in\R^d$ $\mapsto$ $\partial_\mu u(\mu)(x)$ such that the mapping
 $(\mu,x)$ $\in$ $\Pc_{_2}(\R^d)\times\R^d$ $\mapsto$ $\partial_\mu u(\mu)(x)$  is continuous at any $(\mu,x)$ with $x$ $\in$ Supp$(\mu)$, and 
 \begin{itemize}
 	\item[(i)] for fixed $\mu$ $\in$ $\Pc_{_2}(\R^d)$, the mapping $x$ $\in$ $\R^d$ $\mapsto$  $\partial_\mu u(\mu)(x)$ is differentiable with a gradient denoted by  
 	$\partial_x  \partial_\mu u(\mu)(x)$  $\in$ $\R^{d\times d}$, such that the mapping  $(\mu,x)$ $\in$ $\Pc_{_2}(\R^d)\times\R^d$ 
 	$\mapsto$  $\partial_x  \partial_\mu u(\mu)(x)$ is continuous;
 	\item[(ii)] for fixed  $x$ $\in$ $\R^d$, the mapping $\mu$ $\in$ $\Pc_{_2}(\R^d)$ $\mapsto$  $\partial_\mu u(\mu)(x)$ is differentiable in the lifted sense.  Its derivative, interpreted as a mapping $x'$ $\in$ $\R^d$ $\mapsto$ $\partial_\mu \big[ \partial_\mu u(\mu)(x)\big](x')$ $\in$ $\R^{d\times d}$ in 
 	$L^2_\mu(\R^{d\times d})$ and denoted by $x'$ $\in$ $\R^d$ $\mapsto$ $\partial_\mu^2 u(\mu)(x,x')$, such that the mapping $(\mu,x,x')$ $\in$ 
 	$\Pc_{_2}(\R^d)\times\R^d\times\R^d$ $\mapsto$ $\partial_\mu^2 u(\mu)(x,x')$ is continuous. 
 \end{itemize}
 We say that $u$ $\in$ $\Cc^2_b(\Pc_{_2}(\R^d))$ if it is fully $\Cc^2$,  $\partial_x  \partial_\mu u(\mu)$ $\in$ $L_\mu^\infty(\R^{d\times d})$, $\partial_\mu^2 u(\mu)$ $\in$ $L_{\mu\otimes\mu}^\infty(\R^{d\times d})$ for any $\mu$ 
 $\in$ $\Pc_{_2}(\R^d)$,  and for any compact set $\Kc$ of $\Pc_{_2}(\R^d)$, we have
 \beq \label{Kbor}
 \sup_{ \mu \in \Kc } \Big[ \int_{\R^d} \big| \partial_\mu u(\mu)(x) |^2\mu(dx)  +
 \big \| \partial_x \partial_\mu u(\mu)\|_{_\infty} +  \big \| \partial_\mu^2 u(\mu)\|_{_\infty}
 \Big]  & < & \infty.
 \enq

\subsection{Time inconsistent extended MFC problems}

 We consider the controlled conditional McKean-Vlasov SDE of
 $X^{\alpha}:=(X^{\alpha}_{t})_{t_{0}\leq t\leq T}$ valued in $\R^{d}$ that
\begin{equation}\label{dynamics}
	\begin{split}
	&dX^{\alpha}_{t}=b(t,X^{\alpha}_{t},\alpha_{t},\P^{W^0}_{(X^{\alpha}_{t},\alpha_{t})})dt+\sigma(t,X^{\alpha}_{t},\alpha_{t},\P^{W^0}_{(X^{\alpha}_{t},\alpha_{t})})dB_{t}+\sigma_0(t,X^{\alpha}_t,\alpha_t,\P^{W^0}_{(X^{\alpha}_{t},\alpha_{t})}) dW_t^0,\\ &X^{\alpha}_{t_{0}}=\xi\in L^{2}(\Fc_{t_{0}};\R^{d}).
	\end{split}
\end{equation}
Here, $b:[0,T]\times\R^{d}\times A\times\Pc_{2}(\R^{d}\times A)\rightarrow\R^{d}$,  $\sigma:[0,T]\times\R^{d}\times A\times\Pc_{2}(\R^{d}\times A)\rightarrow\R^{d\times n}$ and $\sigma_{0}:[0,T]\times\R^{d}\times A\times\Pc_{2}(\R^{d}\times A)\rightarrow\R^{d\times k}$ are deterministic measurable functions and  the control process $\alpha=(\alpha_{t})_{t_{0}\leq t\leq T}$ is $\F$--progressive taking values in a subset $A$ of $\R^{m}$. Let $\P_{(X^{\alpha}_{t},\alpha_{t})}^{W^0}$ denote the regular conditional joint distribution of $(X^{\alpha}_{t},\alpha_{t})$ given $\Fc^0$, and its realization at some $\omega^0$ $\in$ $\Omega^0$ also reads as the law under $\P^1$ of the random vector $(X_t(\omega^0,\cdot),\alpha_t(\omega^0,\cdot))$ on 
$(\Omega^1,\Fc^1,\P^1)$, i.e., $\P^{W^0}_{(X^{\alpha}_{t},\alpha_{t})}(\omega^0)$ $=$ $\P^1_{(X_t(\omega^0,\cdot),\alpha_t(\omega^0,\cdot))}$
. Note that the process $X^{\alpha}$ depends on the initial pair $(t_{0},\xi)$, and if needed, we shall stress the dependence on $(t_{0},\xi)$ by writing $X^{t_{0},\xi,\alpha}$. We make the following assumption on coefficients $b$, $\sigma$ and $\sigma_0$:
\begin{assumption}\label{assumption:bsigma}
	\begin{itemize}
		
		\item[(i)] There exists a constant $C>0$ s.t. for all $t\in[0,T]$, $x,x'\in\R^{d}$, $a,a'\in A$, $\lambda,\lambda'\in\Pc_{2}(\R^{d}\times A)$,
		\begin{equation*}
			\begin{split}
				|b(t,x,a,\lambda)-b(t,x',a',\lambda')|&+|\sigma(t,x,a,\lambda)-\sigma(t,x',a',\lambda')|+|\sigma_0(t,x,a,\lambda)-\sigma_0(t,x',a',\lambda')|\\&\leq C\left[|x-x'|+|a-a'|+W_{2}(\lambda,\lambda')\right],
			\end{split}
		\end{equation*}
		and 
	{\begin{equation*}
		|b(t,0,a_{0},\delta(0,a_{0}))|+|\sigma(t,0,a_0,\delta(0,a_0))|+|\sigma_0(t,0,a_0,\delta(0,a_0))|<C
		\end{equation*}
	for some $C\geq0$, $a_{0}\in A$.}

		\item[(ii)] For each fixed $(x,a,\lambda)\in\R^{d}\times A\times\Pc_{2}(\R^{d}\times A)$, $b(t,x,a,\lambda)$, $\sigma(t,x,a,\lambda)$ and $\sigma_0(t,x,a,\lambda)$ are right continuous in $t\in[0,T)$.
	\end{itemize}
\end{assumption}
Condition (i) of Assumption \ref{assumption:bsigma} ensures that for any square integrable control process $\alpha$, i.e., $\E[\int_{t_{0}}^{T}|\alpha_{t}|^{2}dt]<\infty,$ there exists a unique solution $X^{\alpha}\in L^{2}_{\F}(\Omega;C([t_{0},T];\R^{d}))$ to \eqref{dynamics}, and it satisfies 
\begin{equation}\label{estimate:sde}
	\E\left[\sup_{t_{0}\leq t\leq T}|X^{\alpha}_{t}|^{2}\right]\leq C\left(1+\E[\xi^{2}]+\E\left[\int_{t_{0}}^{T}|\alpha_{t}|^{2}dt\right]\right)<\infty.
\end{equation}
Condition (ii) of Assumption \ref{assumption:bsigma} imposes a mild continuity requirement and will be used later.

Given the running cost $f:[0,T]\times[0,T]\times\R^{d}\times A\times\Pc_{2}(\R^{d}\times A)$ and the terminal cost $g:[0,T]\times\R^{d}\times\Pc_{2}(\R^{d})$, the cost functional is defined by 
\begin{equation}\label{objective}
	J(\tau;t_{0},\xi;\alpha):=\E\left[\int_{t_{0}}^{T}f(\tau;t,X^{t_{0},\xi,\alpha}_{t},\alpha_{t},\P^{W^0}_{(X^{t_{0},\xi,\alpha}_{t},\alpha_{t})})dt+g(\tau;X^{t_{0},\xi,\alpha}_{T},\P^{W^0}_{X^{t_{0},\xi,\alpha}_{T}})\right].
\end{equation}
We shall make the following assumptions on $f$ and $g$:
\begin{assumption}\label{assumption:fg}
	\begin{itemize}
		\item[(i)] There exists a constant $C(\tau)>0$ which may depend on $\tau$ s.t. for all $t\in[0,T]$, $x\in\R^{d}$, $a\in A$, $\mu\in\Pc_{2}(\R^{d})$, $\lambda\in\Pc_{2}(\R^{d}\times A )$,
		\begin{equation*}
			|f(\tau;t,x,a,\lambda)|+|g(\tau;x,\mu)|\leq C(\tau)\left[1+|x|^{2}+|a|^{2}+M^{2}_{2}(\mu)+M^{2}_{2}(\lambda)\right].
		\end{equation*}

		\item[(ii)] For each fixed $\tau\in[0,T]$, $f(\tau;t,x,a,\lambda)$ is  continuous with respect to $t$, $x$, $a$ and $\lambda$, where the continuity with respect to the probability measures is understood in the sense of the $W_{2}$- metric.
	\end{itemize}
\end{assumption}
Under Condition (i) of Assumption \ref{assumption:fg}, we obtain that $J(\tau;t_{0},\xi;\alpha)$ is finite for any square integrable control process $\alpha$. The continuity condition(ii) in Assumption \ref{assumption:fg} is technical and will be used in what follows.

In the present paper, we mainly focus on closed-loop strategies\footnote{For simplicity of classification, in the present paper, by a closed-loop strategy we mean a deterministic feedback function.}, thus we restrict ourselves to controls $\alpha$ in the closed loop form that
\begin{equation}
	\alpha_{t}=\tilde{\alpha}(t,X_{t},\P^{W^0}_{X_{t}}),\quad t_0\leq t\leq T,
\end{equation}
for some deterministic measurable function $\talpha(t,x,\mu)$ defined on $[0,T]\times\R^{d}\times\Pc_{2}(\R^{d})$.  Under this setup,  SDE \eqref{dynamics} becomes 
\begin{align}\label{closeddynamics}
dX^{\alpha}_{t}&=b(t,X^{\alpha}_{t},\tilde{\alpha}(t,X^{\alpha}_{t},\P^{W^0}_{X^{\alpha}_{t}}),\P^{W^0}_{(X^{\alpha}_{t},\tilde{\alpha}(t,X^{\alpha}_{t},\P^{W^0}_{X^{\alpha}_{t}}))})dt+\sigma(t,X^{\alpha}_{t},\tilde{\alpha}(t,X^{\alpha}_{t},\P^{W^0}_{X^{\alpha}_{t}}),\P^{W^0}_{(X^{\alpha}_{t},\tilde{\alpha}(t,X^{\alpha}_{t},\P^{W^0}_{X^{\alpha}_{t}}))})dB_{t}\nonumber\\&+ \sigma_0(t,X^{\alpha}_{t},\tilde{\alpha}(t,X^{\alpha}_{t},\P^{W^0}_{X^{\alpha}_{t}}),\P^{W^0}_{(X^{\alpha}_{t},\tilde{\alpha}(t,X^{\alpha}_{t},\P^{W^0}_{X^{\alpha}_{t}}))})dW^0_{t},
\end{align} 
with $X^{\alpha}_{t_{0}}=\xi$. 

Let $LIP([0,T]\times\R^d\times\Pc_{2}(\R^d);A)$ denote the set of deterministic measurable functions $\talpha$ defined on $[0,T]\times\R^d\times\Pc_{2}(\R^d)$, valued in $A$, which are Lipschitz in $(x,\mu)$ and satisfy a linear growth condition on  $(x,\mu)$, uniformly on $t\in[0,T]$,  i.e., there exists a positive constant $C>0$ s.t. for all $t\in[0,T]$, $x,x'\in\R^{d}$, $\mu,\mu'\in\Pc_{2}(\R^{d})$,
\begin{equation*}
	|\talpha(t,x,\mu)-\talpha(t,x',\mu')|\leq C(|x-x'|+W_{2}(\mu,\mu')),\, |\talpha(t,0,\delta_{0})|<C.
\end{equation*}
\begin{definition}\label{admissiblestrategy}
	A deterministic measurable function $\talpha:[0,T]\times\R^d\times\Pc_{2}(\R^d)\rightarrow A$ is called an admissible strategy if 
	\begin{itemize}
		\item[(i)] $\talpha\in LIP([0,T]\times\R^d\times\Pc_{2}(\R^d);A)$.
		
		\item[(ii)] For each fixed $(x,\mu)\in\R^{d}\times\Pc_{2}(\R^{d})$, $\talpha(t,x,\mu)$ is right continuous in $t\in[0,T)$.
	\end{itemize}
	We use $\Ac$ to denote the set of admissible strategies.
\end{definition}
 It is straightforward to see that for any $\talpha\in LIP([0,T]\times\R^d\times\Pc_{2}(\R^d);A)$,  under Condition (i) of Assumption \ref{assumption:bsigma}, the closed-loop system \eqref{closeddynamics} admits a unique solution $X^{\alpha}\in L^{2}_{\F}(\Omega;C([t_{0},T];\R^{d}))$  satisfying 
	\begin{equation*}
			\E\left[\sup_{t_{0}\leq t\leq T}|X^{\alpha}_{t}|^{2}\right]\leq C\left(1+\E[\xi^{2}]\right)<\infty.
	\end{equation*}
Similar to the continuity conditions in Assumptions \ref{assumption:bsigma} and \ref{assumption:fg}, Condition (ii) of Definition \ref{admissiblestrategy} will be used later to facilitate some technical proofs. Moreover, the outcome $\alpha$ \footnote{For any admissible strategy $\talpha\in\Ac$, we call the control process $\alpha$ as its outcome if, for $t\in[t_{0},T]$, $\alpha_{t}=\talpha(t,X_{t},\P^{W^0}_{X_{t}})$, where $X_{t}$ solves the closed-loop system \eqref{closeddynamics}. } corresponding to an admissible strategy $\talpha\in\Ac$ is square integrable, thus the cost functional $J(t_{0};t_{0},\xi;\alpha)$ is well defined under the admissible strategy $\talpha$.  In addition, it is easy to see that $Lip(\R^{d};A)\subset\Ac$, where  $Lip(\R^{d};A)$ is the set of Lipschitz functions.

 For any measurable function $\talpha:\R^{d}\rightarrow A$, we denote by $Id\talpha$ the function
\begin{equation*}
	\begin{split}
			Id\talpha:\R^{d}&\rightarrow\R^{d}\times A \\
			x&\rightarrow(x,\talpha(x)).
	\end{split}
\end{equation*}
We observe that the realization of the conditional  distribution $\P^{W^0}_{(X_{t},\alpha_{t})}$ at some $\omega^0\in\Omega^0$ associated to a feedback strategy $\talpha$ is equal to the image by $Id\talpha(t,\cdot,\P^{W^{0}}_{X_{t}}(\omega^0))$ of the  distribution $\P^{W^0}_{X_{t}}(\omega^0)$, i.e.
\begin{equation*}
	\P^{W^0}_{(X_{t},\alpha_{t})}(\omega^0)=Id\talpha(t,\cdot,\P^{W^{0}}_{X_{t}}(\omega^0))\star\\P^{W^{0}}_{X_{t}}(\omega^0),
\end{equation*}
where $\star$ denotes the standard pushforward of measures: for any $\talpha\in L(\R^{d};A)$ and $\mu\in\Pc_{2}(\R^{d})$:
\[(Id\talpha\star\mu)(B)=\mu\left(Id\talpha^{-1}(B)\right),\quad \forall B\in\Bc(\R^{d}\times A).\]
Hence, the dynamics of $X^{t,\xi,\talpha}$ starting with $t\in[0,T]$ and $\xi\in L^{2}(\Fc_{t};\R^{d})$ under the strategy $\talpha$ can be written as
\begin{equation}\label{closedynamics2}
	\begin{split}
		X^{t,\xi,\talpha}_{s}=\xi&+\int_{t}^{s}b(r,X^{t,\xi,\talpha}_{r},\tilde{\alpha}(r,X^{t,\xi,\talpha}_{r},\P^{W^{0}}_{X^{t,\xi,\talpha}_{r}}),Id\talpha(r,\cdot,\P^{W^{0}}_{X^{t,\xi,\talpha}_{r}})\star\P^{W^{0}}_{X^{t,\xi,\talpha}_{r}})dr\\&+\int_{t}^{s}\sigma(r,X^{t,\xi,\talpha}_{r},\tilde{\alpha}(r,X^{t,\xi,\talpha}_{r},\P^{W^{0}}_{X^{t,\xi,\talpha}_{r}}),Id\talpha(r,\cdot,\P^{W^{0}}_{X^{t,\xi,\talpha}_{r}})\star\P^{W^{0}}_{X^{t,\xi,\talpha}_{r}})dB_{r}\\
		&+\int_{t}^{s}\sigma_0(r,X^{t,\xi,\talpha}_{r},\tilde{\alpha}(r,X^{t,\xi,\talpha}_{r},\P^{W^{0}}_{X^{t,\xi,\talpha}_{r}}),Id\talpha(r,\cdot,\P^{W^{0}}_{X^{t,\xi,\talpha}_{r}})\star\P^{W^{0}}_{X^{t,\xi,\talpha}_{r}})dW^0_{r}
		,\quad t\leq s\leq T.
	\end{split}
\end{equation}
As $\{X_s^{t,\xi,\talpha},t\leq s \leq T\}$ is $\F$-adapted, and
$\P_{X_s^{t,\xi,\talpha}}^{W^0}(dx)$ $=$ $\P[X_s^{t,\xi,\talpha} \in dx | \Fc^0]$ $=$ $\P[X_s^{t,\xi,\talpha} \in dx | \Fc_s^0]$, we have for any $\varphi$ $\in$ $\mathscr C_{_2}(\R^d;\R)$:  \begin{equation*}
	 \E \Big[ \varphi(X_s^{t,\xi,\talpha}) \big| \Fc^0 \Big] \; = \; \E \Big[ \varphi(X_s^{t,\xi,\talpha}) \big| \Fc_s^0 \Big], \;\;\; t \leq s \leq T,
\end{equation*}
which shows that $\langle\varphi,\P_{X_s^{t,\xi,\talpha}}^{W^0}\rangle$ is $\Fc_s^0$-measurable. Therefore, in view of the measurability property in Proposition \ref{pro:measurability},   
$\{\P_{X_s^{t,\xi,\talpha}}^{W^0}, t\leq s \leq T\}$ is $(\Fc_s^0)_{t\leq s\leq T}$-adapted.  Moreover, as $X^{t,\xi,\talpha}$ has continuous paths and satisfies $	\E\left[\sup_{t\leq s\leq T}|X^{t,\xi,\talpha}_{s}|^{2}\right]<\infty$, we obtain $\E^{1}\left[\sup_{t\leq s\leq T}|X^{t,\xi,\talpha}_{s}(\omega^0,\cdot)|^{2}\right]<\infty$ for $\P^0$-a.s $\omega^0$ $\in$ $\Omega^0$. Then, from \eqref{winequality}, we easily get that, $\P^0$-almost surely, $\{\P_{X_s^{t,\xi,\talpha}}^{W^0}, t\leq s \leq T\}$ has continuous trajectories. 

For $\P^0$-a.s $\omega^0$ $\in$ $\Omega^0$, the law of the solution $\{X_s^{t,\xi,\talpha}(\omega^0,\cdot),t\leq s\leq T\}$ to \eqref{closedynamics2} on 
$(\Omega^1,\Fc^1,\P^1)$ is unique, which implies that, for $t\leq s \leq T$,
$ \P_{X_s^{t,\xi,\talpha}}^{W^0}(\omega^0)=\P^1_{X_s^{t,\xi,\talpha}(\omega^0,\cdot)}$ depends on $\xi$ only through $\P_\xi^{W^0}(\omega^0)$ $=$ $\P^1_{\xi(\omega^0,\cdot)}$, i.e., the law invariance property holds. Let us now check that for any $\mu$ $\in$ $\Pc_{_2}(\R^d)$,  one can find $\xi$ $\in$ $L^2(\Fc_t;\R^d)$ s.t. $\P_{\xi}^{W^0}$ $=$ $\mu$.  
Indeed, recalling that  $\Gc$ is rich enough, one can find $\xi$ $\in$ $L^2(\Gc;\R^d)$ $\subset$ $L^2(\Fc_t;\R^d)$  s.t. $\P_\xi$ $=$ $\mu$. Because $\Gc$ is independent of $W^0$, it implies that 
$\P_\xi^{W^0}$ $=$ $\mu$. Therefore, one can define the process \[\rho_s^{t,\mu,\talpha}:=\P^{W^0}_{X^{t,\xi,\talpha}_{s}}, \quad\text{for}\quad\forall 0\leq t\leq s\leq T, \,\mu=\P_{\xi}\in\Pc_{2}(\R^{d}),\] and it is a square integrable  $\F^0$-progressively measurable process in 
$\Pc_{_2}(\R^d)$.  
{ Following the arguments in \cite{Pham2016}, we have the flow property that
\begin{equation*}
\rho_s^{t,\mu,\talpha}(\omega^0) =\rho_s^{\theta(\omega^0),\rho_{\theta(\omega^0)}^{t,\mu,\talpha}(\omega^0),\talpha}(\omega^0), \;\;\;  s \in [\theta,T], \; \P^0(d\omega^0)-a.s
\end{equation*}
for all $\theta$ $\in$ $\Tc^0_{t,T}$, the set of  $\F^0$-stopping times valued in $[t,T]$. }

We then can rewrite the cost functional by
\beq 
J(\tau;t,\xi;\talpha) 
= \E \Big[ \int_t^T \hat f(\tau;s,\rho_s^{t,\mu,\talpha},\talpha(s,\cdot,\rho_s^{t,\mu,\talpha})) ds + \hat g(\tau;\rho_T^{t,\mu,\talpha}) \Big] \label{Jmu}
\enq
for $(\tau,t)$ $\in$ $[0,T]\times[0,T]$, $\xi$ $\in$ $L^2(\Gc;\R^d)$ with law $\mu$ $=$ $\P_\xi^{W^0}$  $\in$ $\Pc_{_2}(\R^d)$, $\talpha$ $\in$ $\Ac$, and with the function $\hat{f}$ defined on $[0,T]\times[0,T]\times\Pc_{2}(\R^{d})\times L(\R^{d};A)$ and $\hat{g}$ defined on $[0,T]\times\Pc_{2}(\R^{d})$ by 
 \begin{equation}\label{hat_fg}
	\hat{f}(\tau;t,\mu,\talpha):=\left\langle f(\tau;t,\cdot,\talpha(\cdot),Id\talpha\star\mu),\mu\right\rangle,\quad \hat{g}(\tau;\mu):=\left\langle g(\tau;\cdot,\mu),\mu\right\rangle.
\end{equation}
We shall denote
\begin{align}\label{formulationcon}
J(\tau;t,\mu;\talpha)  \; := \;  J(\tau;t,\xi;\talpha) 
= \E^0 \Big[ \int_t^T \hat f(\tau;s,\rho_s^{t,\mu,\talpha},\talpha(s,\cdot,\rho_s^{t,\mu,\talpha})) ds + \hat g(\tau;\rho_T^{t,\mu,\talpha}) \Big]  
\end{align}
for $(\tau,t)$ $\in$ $[0,T]\times[0,T]$, $\mu\in\Pc_{_2}(\R^d)$, $\xi \in L^2(\Gc;\R^d)$ with $\P_\xi$ $=$ $\mu$, $\talpha\in\Ac$ and the expectation is under $\P^0$. 

{Due to the dependence of the initial time $t_0$ in
the running cost and terminal cost, the problem is inherently time-inconsistent and a global optimal control is not meaningful. A rational agent, being aware of the time inconsistency, reacts in a game theoretic thinking by taking the future selves’ behavior as a constraint and determining the best current action in response. This amounts to find a Nash subgame perfect equilibrium, where the players represent different incarnations of the agent, parametrized by \( s \in [0,T] \). Using the game theoretic thinking, we aim to find the closed-loop equilibrium strategy satisfying the time-consistent planning.} Let $\Dc\subseteq\Ac$ denote the set of alternative strategies that at each time $t$ the agent can choose to implement in an infinitesimally small time period. For given $t\in[0,T)$, $\epsilon\in(0,T-t)$, $\hat{\alpha}\in\Ac$, and $v\in\Dc$, define 
 \begin{equation}
 		\alpha_{t,\epsilon,v}(s,y,\mu):=\begin{cases}
 			\hat{\alpha}(s,y,\mu),\quad s\notin[t,t+\epsilon),y\in\R^{d},\mu\in\Pc_{2}(\R^{d}),\\
 			v(s,y,\mu),\quad s\in[t,t+\epsilon),y\in\R^{d},\mu\in\Pc_{2}(\R^{d}).\\
 		\end{cases}
 \end{equation}
It is easy to see that $\alpha_{t,\epsilon,v}\in\Ac$.
\begin{definition}[Closed-loop equilibrium strategy]
	$\hat{\alpha}\in\Ac$  is a closed-loop equilibrium strategy under time-consistent planning if for any $\mu\in\Pc_{2}(\R^{d})$, $t\in[0,T)$, and $v\in\Dc$, we have 
	\begin{equation}\label{def:equilibrium}
		\Delta^{\hat{\alpha}}(t,\mu;v):=\liminf\limits_{\epsilon\downarrow0}\frac{J(t;t,\mu;\alpha_{t,\epsilon,v})-J(t;t,\mu;\hat{\alpha})}{\epsilon}\geq0.
	\end{equation}
\end{definition}
{\begin{remark}In the present paper, we are interested in characterizing the above closed-loop equilibrium strategy, as introduced in \cite{Bjork2017}, in a mean-field model by developing some PDE and verification arguments. In the existing literature, another notable concept of a time-consistent solution, known as the open-loop equilibrium, has also been studied (e.g., \cite{Hu2012,Hu2017}). Here, we can also give its definition within our framework. 
\begin{definition}[Open-loop equilibrium strategy]
    For a given initial pair $(t_0,\xi)\in[0,T)\times L^{2}(\Fc_{t_{0}};\R^{d})$, a control process \( \hat{\alpha} \) is called an open-loop equilibrium if, for any \( t \in [t_0,T) \) and any bounded \( \Fc_t \)-measurable random variable \( \nu \), the inequality holds:
    \begin{equation*}
         \liminf\limits_{\epsilon\downarrow0}\frac{\tilde{J}(t,\hat{X}_t,\alpha^{t,\epsilon,\nu})-\tilde{J}(t,\hat{X}_t,\hat{\alpha})}{\epsilon}\geq 0,\, a.s.,
    \end{equation*}
   where
   $\hat{X}:=\left(X^{t_0,\xi,\hat{\alpha}}_t\right)_{t_0\leq t\leq T}$ is the solution of \eqref{dynamics} under the control process $\hat{\alpha}$ with initial condition $(t_0,\xi)$, $\alpha^{t,\epsilon,\nu}$ is defined by
   \begin{equation*}
       \alpha^{t,\epsilon,\nu}_{s}:=\hat{\alpha}_{s}+\nu 1_{[t,t+\epsilon)}(s),\quad s\in[t_{0},T]
   \end{equation*}
   and $\tilde{J}(t,\hat{X}_t,\hat{\alpha})$ is defined by 
   \begin{equation*}
     \tilde{J}(t,\hat{X}_t,\hat{\alpha}):= \E\left[\int_{t}^{T}f(t;s,\hat{X}_s,\hat{\alpha}_{s},\P^{W^0}_{(\hat{X}_s,\hat{\alpha}_{s})})ds+g(t;\hat{X}_{T},\P^{W^0}_{\hat{X}_{T}})\bigg|\Fc_{t}\right].
   \end{equation*}
\end{definition}
We note the fundamental distinctions between two definitions, as the open-loop equilibrium is a concept in which a control process is chosen based on the initial condition, while the closed-loop equilibrium is a concept for the feedback function, which is selected independently of the initial condition. For more discussion on these two notions, we refer to \cite{Yan2019}. The characterization of the open-loop equilibrium requires different methods such as the stochastic maximum principle (or FBSDE approach) and will be left for the future research.
\end{remark}}

\section{Characterization of Closed-loop Equilibrium Strategies}\label{sec:characterization}
 For an interval $[a,b]\subset[0,T]$, let $\phi$ be a mapping from $(t,\mu)\in[a,b]\times\Pc_{2}(\R^{d})$ to $\R$. Let us denote $\overline{C}_{b}^{1,2}([a,b]\times\Pc_{2}(\R^{d}))$ as the space of $\phi$ such that $\phi$ is continuous on $[a,b]\times\Pc_{2}(\R^{d})$, $\phi(t,\cdot)\in C^{2}_{b}(\Pc_{2}(\R^{d}))$ for all $t\in[a,b]$, $\phi(\cdot,\mu)\in C^{1}([a,b])$, $\partial_{t}\phi$ is continuous on $[a,b]\times\Pc_{2}(\R^{d})$, the mapping $(t,x,\mu)\in[a,b]\times\R^{d}\times\Pc_{2}(\R^{d})\rightarrow\partial_{\mu}\phi(t,\mu)(x)$ is continuous, the mapping $(t,x,\mu)\in[a,b]\times\R^{d}\times\Pc_{2}(\R^{d})\rightarrow\partial_{x}\partial_{\mu}\phi(t,\mu)(x)$ is continuous, the mapping $(t,x,x',\mu)\in[a,b]\times\R^{d}\times\R^{d}\times\Pc_{2}(\R^{d})\rightarrow\partial^2_{\mu}\phi(t,\mu)(x,x')$ is continuous. For a given strategy $\talpha\in\Ac$, define the  generator $\Lc^{\talpha}$ by 
\begin{equation}\label{generater}
	\begin{split}		\Lc^{\talpha}\phi(t,\mu)&:=\partial_{t}\phi(t,\mu)+\big\langle\partial_{\mu}\phi(t,\mu)(\cdot)\cdot b(t,\cdot,\talpha(t,\cdot,\mu),Id\talpha(t,\cdot,\mu)\star\mu)\\&+\frac{1}{2}\mathrm{tr}(\partial_{x}\partial_{\mu}\phi(t,\mu)(\cdot)(\sigma\sigma^{\top}+\sigma_0\sigma_0^{\top})(t,\cdot,\talpha(t,\cdot,\mu),Id\talpha(t,\cdot,\mu)\star\mu)),\mu\big\rangle
		+\Mc^{\talpha}\phi(t,\mu),
	\end{split}
\end{equation}
where
\begin{equation*}
	\begin{split}
		&\Mc^{\talpha}\phi(t,\mu)\\&:=\int_{\R^{2d}}\frac{1}{2} {\rm tr}\big(  \partial_\mu^2\phi(t,\mu)(x,x')\sigma_0(t,x,\talpha(t,x,\mu),Id\talpha(t,\cdot,\mu)\star\mu))\sigma_0^{\top}(t,x',\talpha(t,x',\mu),Id\talpha(t,\cdot,\mu)\star\mu) \big)\mu(dx)\mu(dx').
	\end{split}
\end{equation*}
For a given strategy $\hat{\alpha}\in\Ac$ and $(\tau,t,\mu)\in\Delta[0,T]\times\Pc_2(\R^d)$, denote 
\begin{equation}\label{tauvaluefunction}
	V(\tau;t,\mu):=J(\tau;t,\mu;\hat{\alpha})=\E^0 \Big[ \int_t^T \hat f(\tau;s,\rho_s^{t,\mu,\hat{\alpha}},\hat{\alpha}(s,\cdot,\rho_s^{t,\mu,\hat{\alpha}})) ds + \hat g(\tau;\rho_T^{t,\mu,\hat{\alpha}}) \Big].
\end{equation}
We impose the following assumption in order to study time-consistent equilibrium strategies.
\begin{assumption}\label{assumption:V}
	
	\begin{itemize}
		
		\item[(i)] For each fixed $\tau\in[0,T]$, $V(\tau;\cdot,\cdot)\in\overline{C}_{b}^{1,2}([\tau,T]\times\Pc_{2}(\R^{d}))$. 
		
		\item[(ii)] $\partial_{t}V(\tau;t,\mu)$ satisfies the quadratic growth condition:
		\begin{equation}\label{qgrowthcondition}
			|\partial_{t}V(\tau;t,\mu)|\leq K(\tau)(1+M^{2}_{2}(\mu)),\quad \forall (\tau,t,\mu)\in\Delta[0,T]\times\Pc_{2}(\R^d),
		\end{equation}
			where $K(\tau)$ is some constant which may depend on $\tau$.

		\item[(iii)] $\partial_{\mu}V(\tau;t,\mu)(x)$ satisfies the linear growth condition:
		\begin{equation}\label{growthcondition}
			|\partial_{\mu}V(\tau;t,\mu)(x)|\leq K(\tau)(1+|x|+M_{2}(\mu)),\quad \forall (\tau,t,x,\mu)\in\Delta[0,T]\times\R^d\times\Pc_{2}(\R^d),
		\end{equation}
		where $K(\tau)$ is some constant which may depend on $\tau$.
		
	\end{itemize}
\end{assumption}

{ The following proposition serves as a technical step towards the verification theorem of the closed-loop equilibrium strategies in the mean-field model.}

 \begin{proposition}
     \label{Theorem:charactization}
 	Let Assumptions \ref{assumption:bsigma}, \ref{assumption:fg} and \ref{assumption:V} hold. 
  For any $(t,\mu)\in[0,T)\times\Pc_{2}(\R^{d})$ and $v\in\Ac$, we have 
 	\begin{equation}\label{expansionthem}
 		J(t;t,\mu;\alpha_{t,\epsilon,v})-J(t;t,\mu;\hat{\alpha})=\epsilon\Delta^{\hat{\alpha}}(t,\mu;v)+o(\epsilon),
 	\end{equation}
 	with $\Delta^{\hat{\alpha}}(t,\mu;v)=\Gamma^{t,\hat{\alpha}}(t,\mu;v)$, where for any $\tau\in[0,t]$,
 	\begin{equation}\label{Gamma}
 		\Gamma^{\tau,\hat{\alpha}}(t,\mu;v):=\hat{f}(\tau;t,\mu,v(t,\cdot,\mu))+\Lc^{v}V(\tau;t,\mu).
 	\end{equation}
 	Moreover, $\Gamma^{\tau,\hat{\alpha}}(t,\mu;v)=\Gamma^{\tau,\hat{\alpha}}(t,\mu;\tilde{v})$ for any $v,\tilde{v}\in\Ac$ with $v(t,\cdot,\mu)=\tilde{v}(t,\cdot,\mu)$, and
 	 $\Gamma^{\tau,\hat{\alpha}}(t,\mu;v)=0$ if $v(t,\cdot,\mu)=\hat{\alpha}(t,\cdot,\mu)$.
  Consequently, suppose $Lip(\R^{d};A)\subseteq\Dc\subseteq\Ac$, $\hat{\alpha}$ is an equilibrium strategy if and only if 
 	\begin{equation}\label{charactization1}
 		\Delta^{\hat{\alpha}}(t,\mu;v)\geq0, \ \  \forall v\in Lip(\R^{d};A),\  \mu\in\Pc_{2}(\R^{d}),\  t\in[0,T).
 	\end{equation}
 \end{proposition}
 \begin{proof}
 	 Note that for $0\leq\tau\leq t< T$ and $\mu\in\Pc_2(\R^d)$, we have
 	  \begin{equation}\label{flowexpression}
 	  	\begin{split}
 	  			&J(\tau;t,\mu;\alpha_{t,\epsilon,v})-J(\tau;t,\mu;\hat{\alpha})\\
    &=\E^0\big[\int_{t}^{t+\epsilon}\hat{f}(\tau;s,\rho^{t,\mu,\alpha_{t,\epsilon,v}}_{s},\alpha_{t,\epsilon,v}(s,\cdot,\rho^{t,\mu,\alpha_{t,\epsilon,v}}_{s}))ds+\hat{g}(\tau;\rho^{t,\mu,\alpha_{t,\epsilon,v}}_{T})\big]-V(\tau;t,\mu)\\&+\E^0[\int_{t+\epsilon}^{T}\hat{f}(\tau;s,\rho^{t,\mu,\alpha_{t,\epsilon,v}}_{s},\alpha_{t,\epsilon,v}(s,\cdot,\rho^{t,\mu,\alpha_{t,\epsilon,v}}_{s}))ds]\\&=\E^0[\int_{t}^{t+\epsilon}\hat{f}(\tau;s,\rho^{t,\mu,v}_{s},v(s,\cdot,\rho^{t,\mu,v}_{s}))ds+\hat{g}(\tau;\rho^{t+\epsilon,\rho^{t,\mu,v}_{t+\epsilon},\alpha_{t,\epsilon,v}}_{T})]-V(\tau;t,\mu)\\&+\E^0[\int_{t+\epsilon}^{T}\hat{f}(\tau;s,\rho^{t+\epsilon,\rho^{t,\mu,v}_{t+\epsilon},\alpha_{t,\epsilon,v}}_{s},\alpha_{t,\epsilon,v}(s,\cdot,\rho^{t+\epsilon,\rho^{t,\mu,v}_{t+\epsilon},\alpha_{t,\epsilon,v}}_{s}))ds]\\&=\E^0[\int_{t}^{t+\epsilon}\hat{f}(\tau;s,\rho^{t,\mu,v}_{s},v(s,\cdot,\rho^{t,\mu,v}_{s}))ds+\hat{g}(\tau;\rho^{t+\epsilon,\rho^{t,\mu,v}_{t+\epsilon},\hat{\alpha}}_{T})]-V(\tau;t,\mu)\\&+\E^0[\int_{t+\epsilon}^{T}\hat{f}(\tau;s,\rho^{t+\epsilon,\rho^{t,\mu,v}_{t+\epsilon},\hat{\alpha}}_{s},\hat{\alpha}(s,\cdot,\rho^{t+\epsilon,\rho^{t,\mu,v}_{t+\epsilon},\hat{\alpha}}_{s}))ds]
 	  			\\&=\underbrace{\E^0[\int_{t}^{t+\epsilon}\hat{f}(\tau;s,\rho^{t,\mu,v}_{s},v(s,\cdot,\rho^{t,\mu,v}_{s}))ds]}_{(1)}+\underbrace{\E^0[V(\tau;t+\epsilon,\rho^{t,\mu,v}_{t+\epsilon})-V(\tau;t,\mu)]}_{(2)},
 	  	\end{split}
 	  \end{equation}
 	  where the second equality is due to $\alpha_{t,\epsilon,v}(s,x,\mu)=v(s,x,\mu)$ for $s\in[t,t+\epsilon)$ and  the flow property  $\rho^{t,\mu,\alpha_{t,\epsilon,v}}_{s}(\omega^0)=\rho^{t+\epsilon,\rho^{t,\mu,v}_{t+\epsilon},\alpha_{t,\epsilon,v}}_{s}(\omega^0)$ for $\P^0$-a.s. $\omega^0\in\Omega^0$ and $t+\epsilon\leq s\leq T$, and the third equality results from the fact that $\alpha_{t,\epsilon,v}(s,x,\mu)=\hat{\alpha}(s,x,\mu)$ for $s\in[t+\epsilon,T)$.

 	  We first handle the term (2) of \eqref{flowexpression}. The standard estimate
 	  	$\E\left[\sup_{t\leq s\leq t+\epsilon}|X^{t,\xi,v}_{s}|^{2}\right]\leq C\left(1+\E|\xi|^{2}\right)<\infty$,
 	combined with the linear growth of $b$, $\sigma$, $\sigma_0$ and $v$, implies 
 	  \begin{equation*}
 	  	\begin{split}
 	  		\E\bigg[&\int_{t}^{t+\epsilon}\bigg(\bigg|b(r,X^{t,\xi,v}_{r},v(r,X^{t,\xi,v}_{r},\rho^{t,\mu,v}_{r}),Idv(r,\cdot,\rho^{t,\mu,v}_{r})\star\rho^{t,\mu,v}_{r})\bigg|^{2}\\&+\bigg|\sigma(r,X^{t,\xi,v}_{r},v(r,X^{t,\xi,v}_{r},\rho^{t,\mu,v}_{r}),Idv(r,\cdot,\rho^{t,\mu,v}_{r})\star\rho^{t,\mu,v}_{r})\bigg|^{2}\\&+\bigg|\sigma_0(r,X^{t,\xi,v}_{r},v(r,X^{t,\xi,v}_{r},\rho^{t,\mu,v}_{r}),Idv(r,\cdot,\rho^{t,\mu,v}_{r})\star\rho^{t,\mu,v}_{r})\bigg|^{2}\bigg)dr\bigg]<\infty.
 	  	\end{split}
 	  \end{equation*}
 	  Applying the Itô's formula along  a flow of conditional probability measures (see \cite{Delarue2015}), we obtain 
 	   \begin{equation}\label{ito}
 	   	\begin{split}
 	   			V(\tau;t+\epsilon,\rho^{t,\mu,v}_{s})-V(\tau;t,\mu)&=\int_{t}^{t+\epsilon}\partial_{t}V(\tau;s,\rho^{t,\mu,v}_{s})ds+\int_{t}^{t+\epsilon}\E_{W^0}\bigg[\partial_{\mu}V(\tau;s,\rho^{t,\mu,v}_{s})(X^{t,\xi,v}_{s})\cdot \tilde{b}(s)\\&+\frac{1}{2}\mathrm{tr}(\partial_{x}\partial_{\mu}V(\tau;s,\rho^{t,\mu,v}_{s})(X^{t,\xi,v}_{s})(\tilde{\sigma}(s)\tilde{\sigma}^{\top}(s)+\tilde{\sigma}_0(s)\tilde{\sigma}_0^{\top}(s)))\bigg]\\&+
 	   			\E_{W^0}\bigg[\E'_{W^0}\bigg[\frac{1}{2}\mathrm{tr}(\partial^2_{\mu}V(\tau;s,\rho^{t,\mu,v}_{s})(X^{t,\xi,v}_{s},X^{'t,\xi,v}_{s})\tilde{\sigma}_0(s)\tilde{\sigma}_0^{'\top}(s))\bigg]\bigg]ds\\&+
 	   			\int_{t}^{t+\epsilon}\E_{W^0}\bigg[\partial_{\mu}V(\tau;s,\rho^{t,\mu,v}_{s})(X^{t,\xi,v}_{s})^{\top}\tilde{\sigma}_0(s)\bigg]dW^0_{s}
 	   	\end{split}
 	   \end{equation}
 	   with 
 	   \begin{equation*}
 	   	\begin{split}
 	   		&\tilde{b}(s):=b(s,X^{t,\xi,v}_{s},v(s,X^{t,\xi,v}_{s},\rho^{t,\mu,v}_{s}),Idv(s,\cdot,\rho^{t,\mu,v}_{s})\star\rho^{t,\mu,v}_{s}),\\
 	   		&\tilde{\sigma}(s):=\sigma(s,X^{t,\xi,v}_{s},v(s,X^{t,\xi,v}_{s},\rho^{t,\mu,v}_{s}),Idv(s,\cdot,\rho^{t,\mu,v}_{s})\star\rho^{t,\mu,v}_{s}),\\
            &\tilde{\sigma}_0(s):=\sigma_0(s,X^{t,\xi,v}_{s},v(s,X^{t,\xi,v}_{s},\rho^{t,\mu,v}_{s}),Idv(s,\cdot,\rho^{t,\mu,v}_{s})\star\rho^{t,\mu,v}_{s}),\\
            &\tilde{\sigma}_0^{'}(s):=\sigma_0(s,X^{'t,\xi,v}_{s},v(s,X^{'t,\xi,v}_{s},\rho^{t,\mu,v}_{s}),Idv(s,\cdot,\rho^{t,\mu,v}_{s})\star\rho^{t,\mu,v}_{s}),
 	   	\end{split}
 	   \end{equation*}
 	   where $\E_{_{W^0}}=\E^1$ represents  the conditional expectation w.r.t. $\Fc^0$ and $X_s^{'t,\xi,\talpha}$ is a copy of $X_s^{t,\xi,\talpha}$ on another probability space 
 	   $(\Omega'=\Omega^0\times\Omega^{'1},\Fc^0\otimes\Fc^{'1},\P^0\times\P^{'1})$, with  $(\Omega^{'1},\Fc^{'1},\P^{'1})$ supporting  $B'$ as a copy of $B$. 
 	   
 	   Based on \eqref{growthcondition} and the linear growth of $\sigma_{0}$ and $v$, we can deduce that
 	   \begin{equation}
 	   	\begin{split}\label{dominated0}
 	   		|\partial_{\mu}V(\tau;s,\rho^{t,\mu,v}_{s})(X^{t,\xi,v}_{s})^{\top} \tilde{\sigma}_{0}(s)|\leq K(\tau)(1+\sup_{t\leq s\leq t+\epsilon}|X^{t,\xi,v}_{s}|^{2}+\E_{W^0}[\sup_{t\leq s\leq t+\epsilon}|X^{t,\xi,v}_{s}|^{2}])
 	   	\end{split}
 	   \end{equation} 
 	   for some constant $K(\tau)$, which may depend on $\tau$.  
 	   Then, the integrand in \eqref{ito} satisfies
 	\begin{equation*}
 		\E_{W^0}\bigg[\partial_{\mu}V(\tau;s,\rho^{t,\mu,v}_{s})(X^{t,\xi,v}_{s})^{\top}\tilde{\sigma}_0(s)\bigg]^{2}\leq K(\tau)(1+\E_{W^0}[\sup_{t\leq s\leq t+\epsilon}|X^{t,\xi,v}_{s}|^{2}]^{2}).
 	\end{equation*}
 	Therefore, 
 	\begin{equation*}
 		\E^0\left[\left(\int_{0}^{T}	\E_{W^0}\bigg[\partial_{\mu}V(\tau;s,\rho^{t,\mu,v}_{s})(X^{t,\xi,v}_{s})^{\top}\tilde{\sigma}_0(s)\bigg]^{2}ds\right)^{\frac{1}{2}}\right]\leq K(\tau)(1+\E[\sup_{t\leq s\leq t+\epsilon}|X^{t,\xi,v}_{s}|^{2}])<\infty,
 	\end{equation*}   
 	  which implies the stochastic integral in \eqref{ito} vanishes under $\E^0$-expectation.
 	  
 	  {Because $X^{t,\xi,v}_{s}$ is continuous in $s$ with $X^{t,\xi,v}_{t}=\xi$ and $\rho^{t,\mu,v}_{s}$ is continuous in $s$ with $\rho^{t,\mu,v}_{t}=\mu$,  by using Conditions (i) and (ii) in Assumption \ref{assumption:bsigma}, Conditions (i) and (ii) in Definition \ref{admissiblestrategy}, and Condition (i) in Assumption \ref{assumption:V}, $\partial_{\mu}V(\tau;s,\rho^{t,\mu,v}_{s})(X^{t,\xi,v}_{s})\cdot\tilde{b}(s)$ converges to $\partial_{\mu}V(\tau;t,\mu)(\xi)\cdot\tilde{b}(t)$ as $s\downarrow t$. 
    
Similar to \eqref{dominated0}, we  also have 
 	  	 \begin{equation}\label{dominated}
 	  		\begin{split}
 	  			|\partial_{\mu}V(\tau;s,\rho^{t,\mu,v}_{s})(X^{t,\xi,v}_{s})^{\top} \tilde{b}(s)|&\leq K(\tau)(1+\sup_{t\leq s\leq t+\epsilon}|X^{t,\xi,v}_{s}|^{2}+\E_{W^0}[\sup_{t\leq s\leq t+\epsilon}|X^{t,\xi,v}_{s}|^{2}]).
 	  		\end{split}
 	  	\end{equation}
 	  	Thanks to \eqref{dominated}, the dominated convergence theorem yields 
 \begin{equation}\label{limit_1}
 	\begin{split}
 			&\lim\limits_{s\downarrow t}\E^0\bigg[\E_{W^0}\bigg[\partial_{\mu}V(\tau;s,\rho^{t,\mu,v}_{s})(X^{t,\xi,v}_{s})\cdot \tilde{b}(s)\bigg]\bigg]=
 		 	\lim\limits_{s\downarrow t}\E\bigg[\partial_{\mu}V(\tau;s,\rho^{t,\mu,v}_{s})(X^{t,\xi,v}_{s})\cdot \tilde{b}(s)\bigg]\\&=\big\langle\partial_{\mu}V(\tau;t,\mu)(\cdot)\cdot b(t,\cdot,v(t,\cdot,\mu),Idv(t,\cdot,\mu)\star\mu),\mu\big\rangle.
 	\end{split}
 \end{equation}
 Similarly, one can show  
  \begin{equation}\label{limit_2}
 	\begin{split}
 		\lim\limits_{s\downarrow t}\E^0\bigg[\E_{W^0}\bigg[\frac{1}{2}\mathrm{tr}\left(\partial_{x}\partial_{\mu}V(\tau;s,\rho^{t,\mu,v}_{s})(X^{t,\xi,v}_{s})(\tilde{\sigma}(s)\tilde{\sigma}^{\top}(s)+\tilde{\sigma}_0(s)\tilde{\sigma}_0^{\top}(s)))\right)\bigg]\bigg]\\=\big\langle\frac{1}{2}\mathrm{tr}(\partial_{x}\partial_{\mu}V(\tau;t,\mu)(\cdot)(\sigma\sigma^{\top}+\sigma_0\sigma_0^{\top})(t,\cdot,v(t,\cdot,\mu),Idv(t,\cdot,\mu)\star\mu)),\mu\big\rangle
 	\end{split}
 \end{equation}
 and 
 \begin{equation}\label{limit_3}
 	\begin{split}
 		\lim\limits_{s\downarrow t}\E^0\bigg[\E_{W^0}\bigg[\E'_{W^0}\bigg[\frac{1}{2}\mathrm{tr}(\partial^2_{\mu}V(\tau;s,\rho^{t,\mu,v}_{s})(X^{t,\xi,v}_{s},X^{'t,\xi,v}_{s})\tilde{\sigma}_0(s)\tilde{\sigma}_0^{'\top}(s))\bigg]\bigg]\bigg]=\Mc^{v}V(\tau;t,\mu).
 	\end{split}
 \end{equation}
 As the mapping $(t,\mu)\rightarrow\partial_{t}V(\tau;t,\mu)$ is continuous on $[\tau,T]\times\Pc_2(\R^d)$ and $\rho^{t,\mu,v}_{s}$ is continuous in $s$ with $\rho^{t,\mu,v}_{t}=\mu$, we deduce from the quadratic growth of $\partial_{t}V(\tau;t,\mu)$ and the dominated convergence theorem that
 \begin{equation}\label{limit_4}
 	\lim\limits_{s\downarrow t}\E^0[\partial_{t}V(\tau;s,\rho^{t,\mu,v}_{s})]=\partial_{t}V(\tau;t,\mu).
 \end{equation}
 Combining \eqref{ito}, \eqref{limit_1}, \eqref{limit_2}, \eqref{limit_3} and \eqref{limit_4}, we obtain
 \begin{equation}\label{expansion1}
 \E^0[V(\tau;t+\epsilon,\rho^{t,\mu,v}_{s})-V(\tau;t,\mu)]=\epsilon\Lc^{v}V(\tau;t,\mu)+o(\epsilon).
 \end{equation}
 We next cope with the term (1) of \eqref{flowexpression} and  rewrite  it as
 \begin{equation*}
 \E^0\left[\int_{t}^{t+\epsilon}\hat{f}(\tau;s,\rho^{t,\mu,v}_{s},v(s,\cdot,\rho^{t,\mu,v}_{s}))ds\right]=	\int_{t}^{t+\epsilon}\E\left[f(\tau;s,X^{t,\xi,v}_{s},\overline{v}(s),\P^{W^0}_{(X^{t,\xi,v}_{s},\overline{v}(s))})\right]ds,
 \end{equation*}
 where $\overline{v}(s):=v(s,X^{t,\xi,v}_{s},\rho^{t,\mu,v}_{s})$.
 The quadratic growth property of $f$ and the linear growth property of $v$ yield 
 \begin{equation*}
 	\begin{split}
 			|f(\tau;s,X^{t,\xi,v}_{s},\overline{v}(s),\P^{W^0}_{(X^{t,\xi,v}_{s},\overline{v}(s))})|
 			\leq K(\tau)(1+\sup_{t\leq s\leq t+\epsilon}|X^{t,\xi,v}_{s}|^{2}+\E_{W^0}[\sup_{t\leq s\leq t+\epsilon}|X^{t,\xi,v}_{s}|^{2}]).
 	\end{split}
 \end{equation*}
 Similarly,  Condition (ii) in Assumption \ref{assumption:fg} and Conditions (i)-(ii) in Definition \ref{admissiblestrategy} imply 
 \begin{equation*}
 	f(\tau;s,X^{t,\xi,v}_{s},\overline{v}(s),\P^{W^0}_{(X^{t,\xi,v}_{s},\overline{v}(s))})\rightarrow f(\tau;t,\xi,\overline{v}(t),\P^{W^0}_{(\xi,\overline{v}(t))})\,\ \text{as}\, s\downarrow t.
 \end{equation*}
 Therefore, the dominated convergence theorem yields  
 \begin{equation}\label{expansion2}
 	E^{0}\left[\int_{t}^{t+\epsilon}\hat{f}(\tau;s,\rho^{t,\mu,v}_{s},v(s,\cdot,\rho^{t,\mu,v}_{s}))ds\right]=\epsilon \hat{f}(\tau;t,\mu,v(t,\cdot,\mu))+o(\epsilon).
 \end{equation}
 Combining \eqref{flowexpression}, \eqref{expansion1} and \eqref{expansion2}, we have 
 \begin{equation}\label{expansion3}
 	J(\tau;t,\mu;\alpha_{t,\epsilon,v})-J(\tau;t,\mu;\hat{\alpha})=\epsilon\Gamma^{\tau,\hat{\alpha}}(t,\mu;v)+o(\epsilon),
 \end{equation}
 where $\Gamma^{\tau,\hat{\alpha}}$ is given by \eqref{Gamma}. Setting $\tau=t$, we readily have \eqref{expansionthem}.
 
 It is straightforward to see from \eqref{Gamma} that $\Gamma^{\tau,\hat{\alpha}}(t,\mu;v)=\Gamma^{\tau,\hat{\alpha}}(t,\mu;\tilde{v})$ for any $v,\tilde{v}\in\Ac$ with $v(t,\cdot,\mu)=\tilde{v}(t,\cdot,\mu)$.  Moreover, by \eqref{expansion3}, we have $\Gamma^{\tau,\hat{\alpha}}(t,\mu;\hat{\alpha})=0$. As a result, for any $v\in\Ac$ with $v(t,\cdot,\mu)=\hat{\alpha}(t,\cdot,\mu)$, we have $\Gamma^{\tau,\hat{\alpha}}(t,\mu;v)=\Gamma^{\tau,\hat{\alpha}}(t,\mu;\hat{\alpha})=0$.
 
 Finally, we conclude from \eqref{expansionthem} that for any $\mu\in\Pc_{2}(\R^{d})$, $t\in[0,T)$, and $v\in\Dc$, 
 \begin{equation*}
 	\liminf\limits_{\epsilon\downarrow0}\frac{J(t;t,\mu;\alpha_{t,\epsilon,v})-J(t;t,\mu;\hat{\alpha})}{\epsilon}=\Gamma^{t,\hat{\alpha}}(t,\mu;v)=\Gamma^{t,\hat{\alpha}}(t,\mu;v(t,\cdot,\mu)).
 \end{equation*}
 In view that $Lip(\R^{d};A)\subseteq\Dc\subseteq\Ac$, it holds that $\hat{\alpha}$ is an equilibrium strategy if and only if \eqref{charactization1} holds.} 	   
 \end{proof}
 	Proposition \ref{Theorem:charactization}  implies that under certain conditions, $V$ should satisfy the following  { equilibrium HJB equation on the Wasserstein space} that
 	\begin{equation}\label{masterequation:valuefunction}
 		\begin{split}
 			&\Lc^{\hat{\alpha}}V(\tau;t,\mu)+\hat{f}(\tau;t,\mu,\hat{\alpha}(t,\cdot,\mu))=0,\quad 0\leq\tau\leq t<T,\,\mu\in\Pc_2(\R^d),\nonumber\\
		&V(\tau;T,\mu)=\hat{g}(\tau;\mu),\quad(\tau,\mu)\in[0,T]\times\Pc_2(\R^d)
 		\end{split}
 	\end{equation}
 	with 
 	\begin{equation}\label{masterequation:equilibrium}
 		\begin{split}
 			\hat{\alpha}(t,\cdot,\mu)=&\argmin_{v\in Lip(\R^{d};A)}\big\langle\partial_{\mu}V(t;t,\mu)(\cdot)\cdot b(t,\cdot,v(\cdot),Idv\star\mu)\\&+\frac{1}{2}\mathrm{tr}(\partial_{x}\partial_{\mu}V(t;t,\mu)(\cdot)(\sigma\sigma^{\top}+\sigma_0\sigma_0^{\top})(t,\cdot,v(\cdot),Idv\star\mu)),\mu\big\rangle
 			+\Mc^{v}V(t;t,\mu)+\hat{f}(t;t,\mu,v(\cdot)).
 		\end{split}
 	\end{equation}
   We next present the verification theorem, 
   the main result of this paper, that under certain conditions, the system \eqref{masterequation:valuefunction}-\eqref{masterequation:equilibrium} provides an equivalent characterization of the equilibrium strategy and the value function. 
 	\begin{theorem}\label{corollary:verification}
 		 Suppose that Assumptions \ref{assumption:bsigma} and \ref{assumption:fg} hold. If $Lip(\R^{d};A)\subseteq\Dc\subseteq\Ac$ and the system \eqref{masterequation:valuefunction}-\eqref{masterequation:equilibrium} admits a solution $V$ satisfying Assumption \ref{assumption:V}  and $\hat{\alpha}\in \Ac$, then $\hat{\alpha}$ is a closed-loop equilibrium strategy.
 	\end{theorem}
 \begin{proof}
Applying  Itô's formula to $V(\tau;s,\rho^{t,\mu,\hat{\alpha}}_{s})$ between $s=t$ and $s=T$ and taking $E^0$-expectation, we obtain  
 \begin{equation*}
 	\E^0[V(\tau;T,\rho^{t,\mu,\hat{\alpha}}_{T})-V(\tau;t,\mu)]=\E^0\left[\int_{t}^{T}\Lc^{\hat{\alpha}}V(\tau;s,\rho^{t,\mu,\hat{\alpha}}_{s})ds\right].
 \end{equation*}
 Consequently, by \eqref{masterequation:valuefunction}, we arrive at 
 \begin{equation*}
 	V(\tau;t,\mu)=\E^0 \Big[ \int_t^T \hat f(\tau;s,\rho_s^{t,\mu,\hat{\alpha}},\hat{\alpha}(s,\cdot,\rho_s^{t,\mu,\hat{\alpha}})) ds + \hat g(\tau;\rho_T^{t,\mu,\hat{\alpha}}) \Big].
 \end{equation*}
 The conclusion readily follows from Proposition \ref{Theorem:charactization}.
 \end{proof}

\section{Time-inconsistent LQ Extended MFC Problems}\label{sec:app}

In this section, we focus on LQ-type extended MFC problems and show the existence of solution to the
{ equilibrium HJB equation} \eqref{masterequation:valuefunction}-\eqref{masterequation:equilibrium} satisfying Assumption \ref{assumption:V} and characterize the closed loop equilibrium. To this end, let us consider a linear McKean-Vlasov dynamics with coefficients given by
 \begin{equation*}
 	\begin{split}
 		&b(t,x,\mu,a,\lambda)=b_{0}(t)+B(t)x+\overline{B}(t)\overline{\mu}+C(t)a+\overline{C}(t)\overline{\lambda},\\
 		&\sigma(t,x,\mu,a,\lambda)=\vartheta (t)+D(t)x+\overline{D}(t)\overline{\mu}+F(t)a+\overline{F}(t)\overline{\lambda},\\
 		&\sigma_0(t,x,\mu,a,\lambda)=\vartheta_0 (t)+D_0(t)x+\overline{D}_0(t)\overline{\mu}+F_0(t)a+\overline{F}_0(t)\overline{\lambda}
 	\end{split}
 \end{equation*}
 for $(t,x,\mu,a,\lambda)\in[0,T]\times\R^{d}\times\Pc_{2}(\R^{d})\times \R^{m}\times\Pc_{2}(\R^{m})$, where  $\overline{\mu}:=\int_{\R^{d}}x\mu(dx)$ and $\overline{\lambda}:=\int_{\R^{m}}a\lambda(da)$.
Here $B$, $\overline{B}$, $D$, $\overline{D}$, $D_0$, $\overline{D}_0$ are deterministic continuous functions valued in $\R^{d\times d}$, and $C$, $\overline{C}$, $F$, $\overline{F}$, $F_0$, $\overline{F}_0$ are deterministic continuous functions valued in $\R^{d\times m}$, and $b_{0}$, $\vartheta$, $\vartheta_{0}$ are deterministic continuous functions valued in $\R^{d}$.  The
 quadratic cost functions are given by
 \begin{equation*}
 	\begin{split}
 		 	&f(\tau;t,x,\mu,a,\lambda)=x^{\top}Q(\tau;t)x+\overline{\mu}^{\top}\overline{Q}(\tau;t)\overline{\mu}+a^{\top}R(\tau;t)a+\overline{\lambda}^{\top}\overline{R}(\tau;t)\overline{\lambda}+2x^{\top}M(\tau;t)a\\
 		 	&\quad\quad\quad\quad +2\overline{\mu}^{\top}\overline{M}(\tau;t)\overline{\lambda}+q(\tau;t)\cdot x+\overline{q}(\tau;t)\cdot\overline{\mu}+r(\tau;t)\cdot a+\overline{r}(\tau;t)\cdot\overline{\lambda},\\
 		 	&g(\tau;x,\mu)=x^{\top}P(\tau)x+\overline{\mu}^{\top}\overline{P}(\tau)\overline{\mu}+p(\tau)\cdot x+\overline{p}(\tau)\cdot\overline{\mu},
 	\end{split}
 \end{equation*}
 where $Q$ and $\overline{Q}$ are deterministic continuous functions  on  $\Delta[0,T]=\left\{(\tau,t)|0\leq \tau\leq t\leq T\right\}$ valued in $\S^{d}$,  $R$ and $\overline{R}$ are deterministic continuous functions  on $\Delta[0,T]$ valued in $\S^{m}$, $M$ and $\overline{M}$ are deterministic continuous functions  on $\Delta[0,T]$ valued in $\R^{d\times m}$,  $q$ and $\overline{q}$ are deterministic continuous functions  on $\Delta[0,T]$ valued in $\R^{d}$ , $r$ and $\overline{r}$ are deterministic continuous functions  on $\Delta[0,T]$ valued in $\R^{m}$,  $P$ and $\overline{P}$ are deterministic continuous functions  on $[0,T]$ valued in $\S^{d}$ and $p$ and $\overline{p}$ are deterministic continuous functions  on $[0,T]$ valued in $\R^{d}$.
 \begin{remark}\label{Continuousextension}
Note that a continuous function $K$ defined on $\Delta[0,T]$ can be continuously extended to a function on $[0,T]\times[0,T]$  that
	$K(\tau;t)=K(\tau;\tau)$, $t\in[0,\tau)$. Therefore, in this paper, continuous functions defined on \(\Delta[0,T]\) are often treated as continuous functions defined on \([0,T] \times [0,T]\).
 \end{remark}
 The functions $\hat{f}$ and $\hat{g}$ defined in \eqref{hat_fg} are then given by
 \begin{equation*}
 	\begin{cases}
 	\hat{f}(\tau;t,\mu,\talpha)=Var(\mu)(Q(\tau;t))+\overline{\mu}^{\top}(Q(\tau;t)+\overline{Q}(\tau;t))\overline{\mu}\\\quad\quad\quad\quad+Var(\talpha\star\mu)(R(\tau;t))+\overline{\talpha\star\mu}^{\top}(R(\tau;t)+\overline{R}(\tau;t))\overline{\talpha\star\mu}\\
 	\quad\quad\quad\quad+2\overline{\mu}^{\top}(M(\tau;t)+\overline{M}(\tau;t))\overline{\talpha\star\mu}+2\int_{\R^{d}}(x-\overline{\mu})^{\top}M(\tau;t)\talpha(x)\mu(dx)\\
 	\quad\quad\quad\quad+(q(\tau;t)+\overline{q}(\tau;t))\cdot\overline{\mu}+(r(\tau;t)+\overline{r}(\tau;t))\cdot\overline{\talpha\star\mu},\\
 	\hat{g}(\tau;\mu)=Var(\mu)(P(\tau))+\overline{\mu}^{\top}(P(\tau)+\overline{P}(\tau))\overline{\mu}+(p(\tau)+\overline{p}(\tau))\cdot\overline{\mu}
 	\end{cases}
 \end{equation*}
 for any $(\tau,t,\mu)\in [0,T]\times[0,T]\times\Pc_{2}(\R^{d})$, $\talpha\in L(\R^{d};\R^{m})$, where we set for any $\Lambda$ in $\S^{d}$ (resp. in $\S^{m}$), and $\mu\in\Pc_{2}(\R^{d})$ (resp. in $\Pc_{2}(\R^{m})$) that
 
 $$\overline{\mu}_{2}(\Lambda):=\int x^{\top}\Lambda x\mu(dx), \quad Var(\mu)(\Lambda):=\overline{\mu}_{2}(\Lambda)-\overline{\mu}^{\top}\Lambda\overline{\mu}.$$ Here, we always assume $ Lip(\R^{d};\R^{m})\subseteq\Dc$ and  look for a solution to the system \eqref{masterequation:valuefunction}-\eqref{masterequation:equilibrium} in the form:
 \begin{equation}\label{ansatz}
 	V(\tau;t,\mu)=Var(\mu)(\Lambda(\tau;t))+\overline{\mu}^{\top}\beta(\tau;t)\overline{\mu}+\gamma(\tau;t)\cdot\overline{\mu}+\kappa(\tau;t)
 \end{equation}
 for some undetermined map $\Lambda,\beta: \Delta[0,T]\rightarrow \S^{d}$, $\gamma: \Delta[0,T]\rightarrow \R^{d}$, and $\kappa:\Delta[0,T]\rightarrow \R$.
 
 Then, it holds that $V$ satisfies  \eqref{masterequation:valuefunction}-\eqref{masterequation:equilibrium} if and only if 
 \begin{equation*}
 	\begin{split}
 		&Var(\mu)(\Lambda(\tau;T))+\overline{\mu}^{\top}\beta(\tau;T)\overline{\mu}+\gamma(\tau;T)\cdot\overline{\mu}+\kappa(\tau;T)\\=&Var(\mu)(P(\tau))+\overline{\mu}^{\top}(P(\tau)+\overline{P}(\tau))\overline{\mu}+(p(\tau)+\overline{p}(\tau))\cdot\overline{\mu},
 	\end{split}
 \end{equation*}
 for all $(\tau,\mu)\in[0,T]\times\Pc_{2}(\R^{d})$, and 
 \begin{equation}\label{LQmasterequation}
 	\begin{split}
 	&Var(\mu)\left(\Lambda'(\tau;t)+Q(\tau;t)+D^{\top}(t)\Lambda(\tau;t)D(t)+D_0^{\top}(t)\Lambda(\tau;t)D_0(t)+\Lambda(\tau;t)B(t)+B^{\top}(t)\Lambda(\tau;t)\right)\\&\quad+G^{\mu}(\tau;t,\hat{\alpha}(t,\cdot,\mu))
 	+\overline{\mu}^{\top}\bigg(\beta'(\tau;t)+Q(\tau;t)+\overline{Q}(\tau;t)+(D(t)+\overline{D}(t))^{\top}\Lambda(\tau;t)(D(t)+\overline{D}(t))\\&\quad+(D_0(t)+\overline{D}_0(t))^{\top}\beta(\tau;t)(D_0(t)+\overline{D}_0(t))+\beta(\tau;t)(B(t)+\overline{B}(t))+(B(t)+\overline{B}(t))^{\top}\beta(\tau;t)\bigg)\overline{\mu}\\&\quad+\bigg(\gamma'(\tau;t)+q(\tau;t)+\overline{q}(\tau;t)+(B(t)+\overline{B}(t))^{\top}\gamma(\tau;t)+2(D(t)+\overline{D}(t))^{\top}\Lambda(\tau;t)\vartheta(t)\\&\quad+2(D_0(t)+\overline{D}_0(t))^{\top}\beta(\tau;t)\vartheta_0(t)+2\beta(\tau;t)b_{0}(t)\bigg)\cdot\overline{\mu}+\kappa'(\tau;t)+b_{0}^{\top}(t)\gamma(\tau;t)+\vartheta^{\top}(t)\Lambda(\tau;t)\vartheta(t)\\&\quad+\vartheta_0^{\top}(t)\beta(\tau;t)\vartheta_0(t)=0,\quad 0\leq\tau\leq t <T,\ \mu\in\Pc_{2}(\R^{d})\end{split}
 \end{equation}
where the functional $G^{\mu}(\tau;t,\cdot):L^{2}_{\mu}(\R^{m})\supset Lip(\R^{d};\R^{m})\rightarrow\R$ is defined by 
 \begin{equation*}
 	\begin{split}
 		 	G^{\mu}(\tau;t,\talpha)&:=Var(\talpha\star\mu)(U(\tau;t))+\overline{\talpha\star\mu}^{\top}W(\tau;t)\overline{\talpha\star\mu}+2\int_{\R^{d}}(x-\overline{\mu})^{\top}S(\tau;t)\talpha(x)\mu(dx)\\&+2\overline{\mu}^{\top}Z(\tau;t)\overline{\talpha\star\mu}+Y^{\top}(\tau;t)\overline{\talpha\star\mu}
 	\end{split}
 \end{equation*}
with 
 \begin{equation*}
 	\begin{cases}
 U(\tau;t):=F^{\top}(t)\Lambda(\tau;t)F(t)+F_0^{\top}(t)\Lambda(\tau;t)F_0(t)+R(\tau;t),\\
 W(\tau;t):=(F(t)+\overline{F}(t))^{\top}\Lambda(\tau;t)(F(t)+\overline{F}(t))+(F_0(t)+\overline{F}_0(t))^{\top}\beta(\tau;t)(F_0(t)+\overline{F}_0(t))\\\qquad\qquad\qquad\qquad\qquad\qquad\qquad+R(\tau;t)+\overline{R}(\tau;t),\\
 S(\tau;t):=D^{\top}(t)\Lambda(\tau;t)F(t)+D_0^{\top}(t)\Lambda(\tau;t)F_0(t)+\Lambda(\tau;t)C(t)+M(\tau;t),\\
 Z(\tau;t):=(D(t)+\overline{D}(t))^{\top}\Lambda(\tau;t)(F(t)+\overline{F}(t))+(D_0(t)+\overline{D}_0(t))^{\top}\beta(\tau;t)(F_0(t)+\overline{F}_0(t))\\\qquad\qquad\qquad\quad\qquad\qquad\qquad\quad+\beta(\tau;t)(C(t)+\overline{C}(t))+M(\tau;t)+\overline{M}(\tau;t),\\
 Y(\tau;t):=(C(t)+\overline{C}(t))^{\top}\gamma(\tau;t)+2(F(t)+\overline{F}(t))^{\top}\Lambda(\tau;t)\vartheta(t)\\\qquad\qquad\qquad\qquad\qquad\qquad\qquad+2(F_0(t)+\overline{F}_0(t))^{\top}\beta(\tau;t)\vartheta_0(t)+r(\tau;t)+\overline{r}(\tau;t),
  	\end{cases}
 \end{equation*}
  and $\hat{\alpha}(t,\cdot,\mu)$ ($\hat{\alpha}$ for short) is the infimum of the functional $G^{\mu}_{t}(\cdot):=G^{\mu}(t;t,\cdot)$. After some direct calculations, the Gateaux derivative of $G^{\mu}_{t}$ at $\hat{\alpha}$ in the direction $\theta\in L^{2}_{\mu}(\R^{m})$ can be derived as
  \begin{equation*}
  	DG^{\mu}_{t}(\hat{\alpha},\theta):=\lim\limits_{\epsilon\rightarrow0}\frac{G^{\mu}_{t}(\hat{\alpha}+\epsilon\theta)-G^{\mu}_{t}(\hat{\alpha})}{\epsilon}=\int_{\R^{d}}\dot{g}^{\mu}_{t}(x,\hat{\alpha})\cdot\theta(x)\mu(dx)
  \end{equation*}
  with
  \begin{equation*}
  	\dot{g}^{\mu}_{t}(x,\hat{\alpha})=2U(t;t)\hat{\alpha}+2(W(t;t)-U(t;t))\overline{\hat{\alpha}\star\mu}+2S^{\top}(t;t)(x-\overline{\mu})+2Z^{\top}(t;t)\overline{\mu}+Y(t;t).
  \end{equation*}
  Suppose that the symmetric matrices $U(t;t)$ and $W(t;t)$ are positive and hence invertible (that will be discussed later). Then, the functional $G^{\mu}_{t}$ is convex and coercive on the Hilbert space $L^{2}_{\mu}(\R^{m})$, and attains its infimum at some $\hat{\alpha}$ s.t. $DG^{\mu}_{t}(\hat{\alpha},\cdot)$ vanishes, which implies 
  \begin{equation}\label{feedback}
  	\hat{\alpha}(t,x,\mu)=-U^{-1}(t;t)S^{\top}(t;t)(x-\overline{\mu})-W^{-1}(t;t)Z^{\top}(t;t)\overline{\mu}-\frac{1}{2}W^{-1}(t;t)Y(t;t).
  \end{equation}
  It is clear that $\hat{\alpha}(t,\cdot,\mu)$ lies in $Lip(\R^{d};\R^{m})$, and 
  some straightforward calculations lead to
  {\small
  \begin{equation*}
  	\begin{split}
  		G^{\mu}(\tau;t,\hat{\alpha})&=Var(\mu)\bigg(S(t;t)U^{-1}(t;t)U(\tau;t)U^{-1}(t;t)S^{\top}(t;t)-S(\tau;t)U^{-1}(t;t)S^{\top}(t;t)-S(t;t)U^{-1}(t;t)S^{\top}(\tau;t)\bigg)\\&+\overline{\mu}^{\top}\bigg(Z(t;t)W^{-1}(t;t)W(\tau;t)W^{-1}(t;t)Z^{\top}(t;t)-Z(\tau;t)W^{-1}(t;t)Z^{\top}(t;t)-Z(t;t)W^{-1}(t;t)Z^{\top}(\tau;t)\bigg)\overline{\mu}\\&+\bigg(Y^{\top}(t;t)W^{-1}(t;t)W(\tau;t)W^{-1}(t;t)Z^{\top}(t;t)-Y^{\top}(t;t)W^{-1}(t;t)Z^{\top}(\tau;t)-Y^{\top}(\tau;t)W^{-1}(t;t)Z^{\top}(t;t)\bigg)\overline{\mu}\\&+\frac{1}{4}Y^{\top}(t;t)W^{-1}(t;t)W(\tau;t)W^{-1}(t;t)Y(t;t)-\frac{1}{2}
  		Y^{\top}(\tau;t)W^{-1}(t;t)Y(t;t).
  	\end{split}
  \end{equation*}}
  Plugging the above expression back to \eqref{LQmasterequation}, we obtain that  $\Lambda$, $\beta$, $\gamma$, and $\kappa$ shall satisfy
  \begin{align}
  	&\begin{cases}\label{riccati1}
  		\Lambda'(\tau;t)+Q(\tau;t)+D^{\top}(t)\Lambda(\tau;t)D(t)+D_0^{\top}(t)\Lambda(\tau;t)D_0(t)+\Lambda(\tau;t)B(t)+B^{\top}(t)\Lambda(\tau;t)\\+S(t,t,\Lambda(t;t))U^{-1}(t,t,\Lambda(t;t))U(\tau,t,\Lambda(\tau;t))U^{-1}(t,t,\Lambda(t;t))S^{\top}(t,t,\Lambda(t;t))\\-S(\tau,t,\Lambda(\tau;t))U^{-1}(t,t,\Lambda(t;t))S^{\top}(t,t,\Lambda(t;t))\\-S(t,t,\Lambda(t;t))U^{-1}(t,t,\Lambda(t;t))S^{\top}(\tau,t,\Lambda(\tau;t))=0,
  		\\
  		\Lambda(\tau;T)=P(\tau),
  	\end{cases}\\
  	&\begin{cases}\label{riccati2}
  		\beta'(\tau;t)+Q(\tau;t)+\overline{Q}(\tau;t)+(D(t)+\overline{D}(t))^{\top}\Lambda(\tau;t)(D(t)+\overline{D}(t))\\+(D_0(t)+\overline{D}_0(t))^{\top}\beta(\tau;t)(D_0(t)+\overline{D}_0(t))+\beta(\tau;t)(B(t)+\overline{B}(t))+(B(t)+\overline{B}(t))^{\top}\beta(\tau;t)\\+Z(t,t,\Lambda(t;t),\beta(t;t))W^{-1}(t,t,\Lambda(t;t),\beta(t;t))W(\tau,t,\Lambda(\tau;t),\beta(\tau;t))W^{-1}(t,t,\Lambda(t;t),\beta(t;t))\\
    \times Z^{\top}(t,t,\Lambda(t;t),\beta(t;t))-Z(\tau,t,\Lambda(\tau;t),\beta(\tau;t))W^{-1}(t,t,\Lambda(t;t),\beta(t;t))Z^{\top}(t,t,\Lambda(t;t),\beta(t;t))
  		\\-Z(t,t,\Lambda(t;t),\beta(t;t))W^{-1}(t,t,\Lambda(t;t),\beta(t;t))Z^{\top}(\tau,t,\Lambda(\tau;t),\beta(\tau;t))=0,\\
  		\beta(\tau;T)=P(\tau)+\overline{P}(\tau),
  	\end{cases}\\
  	&\begin{cases}\label{linearode1}
  		\gamma'(\tau;t)+q(\tau;t)+\overline{q}(\tau;t)+(B(t)+\overline{B}(t))^{\top}\gamma(\tau;t)+2(D(t)+\overline{D}(t))^{\top}\Lambda(\tau;t)\vartheta(t)\\+2(D_0(t)+\overline{D}_0(t))^{\top}\beta(\tau;t)\vartheta_0(t)+2\beta(\tau;t)b_{0}(t)\\+Z(t,t,\Lambda(t;t),\beta(t;t))W^{-1}(t,t,\Lambda(t;t),\beta(t;t))W(\tau,t,\Lambda(\tau;t),\beta(\tau;t))W^{-1}(t,t,\Lambda(t;t),\beta(t;t))\\
    \times Y(t,t,\Lambda(t;t),\gamma(t;t))-Z(\tau,t,\Lambda(\tau;t),\beta(\tau;t))W^{-1}(t,t,\Lambda(t;t),\beta(t;t))Y(t,t,\Lambda(t;t),\gamma(t;t))\\-Z(t,t,\Lambda(t;t),\beta(t;t))W^{-1}(t,t,\Lambda(t;t),\beta(t;t))Y(\tau,t,\Lambda(\tau;t),\gamma(\tau;t))=0,\\
  		\gamma(\tau;T)=p(\tau)+\overline{p}(\tau),
  	\end{cases}\\
  	&\begin{cases}\label{linearode2}
  		\kappa'(\tau;t)+b_{0}^{\top}(t)\gamma(\tau;t)+\vartheta^{\top}(t)\Lambda(\tau;t)\vartheta(t)+\vartheta_0^{\top}(t)\beta(\tau;t)\vartheta_0(t)\\+\frac{1}{4}Y^{\top}(t,t,\Lambda(t;t),\gamma(t;t))W^{-1}(t,t,\Lambda(t;t),\beta(t;t))W(\tau,t,\Lambda(\tau;t),\beta(\tau;t))W^{-1}(t,t,\Lambda(t;t),\beta(t;t))\\
    \times Y(t,t,\Lambda(t;t),\gamma(t;t))-\frac{1}{2}
  		Y^{\top}(\tau,t,\Lambda(\tau;t),\gamma(\tau;t))W^{-1}(t,t,\Lambda(t;t),\beta(t;t))Y(t,t,\Lambda(t;t),\gamma(t;t))=0,\\
  		\kappa(\tau;T)=0.
  	\end{cases}
  \end{align}
  Therefore, the system \eqref{masterequation:valuefunction}-\eqref{masterequation:equilibrium} in the LQ framework is reduced to the system of nonlocal Riccati equations \eqref{riccati1} and \eqref{riccati2} for $\Lambda$ and $\beta$,  the linear nonlocal ODE \eqref{linearode1} as well as the linear ODE \eqref{linearode2} for $\gamma$ and $\kappa$.  {Suppose that there exists a solution $(\Lambda,\beta,\gamma,\kappa)$ to \eqref{riccati1}-\eqref{linearode2} s.t.  $(U(t;t),W(t;t))$ lies in $\S^{m}_{>+}\times\S^{m}_{>+}$ for all $t\in[0,T]$. Then, it is easy to see that 
  the function $V$ defined by \eqref{ansatz} with $\left(\Lambda,\beta,\gamma,\kappa\right)$ satisfies Assumption \ref{assumption:V} and $\hat{\alpha}\in\Ac$.  Therefore, $\hat{\alpha}$ defined by \eqref{feedback} is an equilibrium strategy. We note that the system \eqref{riccati1}-\eqref{linearode2} differs from the ones in\cite{Yong2013,Yong2017,Ni2019,Wang2019} due to the dependence of the control distribution and common noise. Our
system \eqref{riccati1}-\eqref{linearode2} also differs substantially from the counterparts in  \cite{Hu2012,Ni2017,CCPW,LM23} for single agent's control problems, which deserves some careful investigations.

For technical convenience, we shall focus on the case when $M=\overline{M}=0$ (i.e., no cross term between the state and the control in the cost function $f$) and analyze the well-posedness of system  \eqref{riccati1}-\eqref{linearode2}, which is one of our main results.
 \begin{proposition}\label{Pro:riccati}
Assume that for some $\delta>0$,
\begin{equation}\label{condition:pd}
	\begin{cases}
		Q(\tau;t),\, Q(\tau;t)+\overline{Q}(\tau;t)\geq0, \,\ 0\leq\tau\leq t\leq T,\\
		R(\tau;t),\, R(\tau;t)+\overline{R}(\tau;t)\geq\delta I, \,\ 0\leq\tau\leq t\leq T,\\
		P(\tau),\, P(\tau)+\overline{P}(\tau)\geq0,\,\  0\leq\tau\leq T,
	\end{cases}
\end{equation}
and the following monotonicity conditions are satisfied: for $0\leq t\leq\tau\leq s\leq T$,
\begin{equation}\label{condition:monotonicity}
	\begin{cases}
		Q(t;s)\leq Q(\tau;s),\quad Q(t;s)+\overline{Q}(t;s)\leq Q(\tau;s)+\overline{Q}(\tau;s),\\
		R(t;s)\leq R(\tau;s),\quad R(t;s)+\overline{R}(t;s)\leq R(\tau;s)+\overline{R}(\tau;s),\\
		P(t)\leq P(\tau),\quad P(t)+\overline{P}(t)\leq P(\tau)+\overline{P}(\tau).
	\end{cases}
\end{equation}
Then, the nonlocal Riccati equation system \eqref{riccati1}-\eqref{linearode2} admits a unique solution. Moreover,   $(U(t;t),W(t;t))$ lies in $\S^{m}_{>+}\times\S^{m}_{>+}$ for all $t\in[0,T]$.
\end{proposition}
  \begin{proof}
  	 The proof is reported in Subsection \ref{appendixA}.
  \end{proof}
  \begin{remark}
 We give some comments on these conditions. Condition \eqref{condition:pd} is commonly found in the literature on linear-quadratic control problems; see, e.g., \cite{Yong2017}, \cite{Yong2012}, and \cite{Yong2013}, primarily to ensure the well-posedness of the related Riccati equations.  We point out that if
  \begin{equation*}
  	\begin{cases}
  		Q(\tau;t)=\lambda(t-\tau)Q(t),\quad \overline{Q}(\tau;t)=\lambda(t-\tau)\overline{Q}(t),\\
  		R(\tau;t)=\lambda(t-\tau)R(t),\quad\overline{R}(\tau;t)=\lambda(t-\tau)\overline{R}(t),\quad 0\leq\tau\leq t\leq T,\\
  		P(\tau)=\lambda(T-\tau)P,\quad\overline{P}(\tau)=\lambda(T-\tau)\overline{P}, 
  	\end{cases}
  	  \end{equation*} 
  	 $\lambda$ is a positive, continuous and decreasing function defined on $[0,T]$  with  
  	\begin{equation*}
  		\begin{cases}
  			Q(t),\ Q(t)+\overline{Q}(t)\geq0,\quad R(t), R(t)+\overline{R}(t)\geq0,\, t\in[0,T],\\
  			P, P+\overline{P}\geq0,
  		\end{cases}
  	\end{equation*}
then the monotonicity conditions \eqref{condition:monotonicity} hold.
\end{remark}

{
 \textbf{Example-1: Conditional  mean-variance portfolio selection under non-exponential discount}
\vskip 5pt
As an example in financial application, we consider a conditional mean-variance problem under non-exponential discounting characterized by the following running objective functional:
\begin{equation*}
 	\begin{split}
 		J(t_{0};t_{0},\xi;\alpha) &:= \mathbb{E}^0 \bigg[\int_{t_0}^T \lambda_1(s - t_0) \left[\frac{\eta}{2} \operatorname{Var}(X_s | W^0) - \mathbb{E}\left[X_s | W^0\right] \right] ds \\
 		&\quad + \lambda_2(T - t_{0}) \left[\frac{\eta}{2} \operatorname{Var}(X_T | W^0) - \mathbb{E}\left[X_T | W^0\right] \right] \bigg],
 	\end{split}
 \end{equation*}
where \(\eta > 0\), and \(\lambda_1, \lambda_2\) are positive continuous discount functions defined on \([0,T]\) with \(\lambda_i(0) = 1\) for \(i = 1,2\). Let $\alpha_{t}$, valued in $\R$, denote the amount invested in the risky asset. The self-financing wealth process $X=(X_{t}^{\alpha})_{t_{0}\leq t\leq T}$ is given  by
 \begin{equation*}
 	dX_{t}=r(t)X_{t}dt+\alpha_{t}\{\rho(t)dt+\theta(t)dB_{t}+\theta_0(t)dW^0_t\},\quad X_{t_{0}}=\xi\in L^{2}(\Gc;\R),
 \end{equation*}
where $r$ is the interest rate, $\rho$, $\theta>0$ and $\theta_0>0$ are the excessive rate of return and the volatility of the idiosyncratic noise and common noise. All deterministic coefficient functions are assumed to be continuous. This model fits into the LQ
framework of the time-inconsistent MFC problem, with a linear dynamics that
 \begin{equation*}
 	\begin{split}
 		&b_{0}=0,\, B(t)=r(t),\, \overline{B}=0,\, C(t)=\rho(t), \overline{C}=0,\\
 		&\vartheta=D=\overline{D}=0,\, F(t)=\theta(t),\, \overline{F}=0,\\
 		&\vartheta_0=D_0=\overline{D}_0=0,\, F_0(t)=\theta_0(t),\, \overline{F}_0=0,\\
 		&Q(\tau;t)=\frac{\eta}{2}\lambda_1(t-\tau),\,\overline{Q}(\tau;t)=-\frac{\eta}{2}\lambda_1(t-\tau),\,M=\overline{M}=R=\overline{R}=q=r=\overline{r}=0,\\
 		&\overline{q}(\tau;t)=-\lambda_1(t-\tau),\, P(\tau)=\frac{\eta}{2}\lambda_2(T-\tau),\,\overline{P}(\tau)=-\frac{\eta}{2}\lambda_2(T-\tau),\,p(\tau)=0,\,\overline{p}(\tau)=-\lambda_2(T-\tau).
 	\end{split}
 \end{equation*}
 The  system \eqref{riccati1}-\eqref{linearode2} for $\left(\Lambda,\beta,\gamma,\kappa\right)$ in this case can be simplified as
 \begin{equation*}
 	\begin{cases}
 		\Lambda'(\tau;t)+\frac{\eta}{2}\lambda_{1}(t-\tau)-(\frac{\rho^{2}(t)}{\theta^{2}(t)+\theta^2_0(t)}-2r(t))\Lambda(\tau;t)=0,\quad \Lambda(\tau;T)=\frac{\eta}{2}\lambda_2(T-\tau),\\
 		\beta'(\tau;t)+2r(t)\beta(\tau;t)+K_1(t,\tau)\beta^{2}(t;t)-2K_2(t,\tau)\beta(t;t)\beta(\tau;t)=0,\quad \beta(\tau;T)=0,\\
 		\gamma'(\tau;t)-\lambda_1(t-\tau)+r(t)\gamma(\tau;t)+K_1(t,\tau)\beta(t;t)\gamma(t;t)-K_2(t,\tau)(\beta(\tau;t)\gamma(t;t)+\beta(t;t)\gamma(\tau;t))=0,\\
   
 		\gamma(\tau;T)=-\lambda_2(T-\tau),\\
   
 		\kappa'(\tau;t)+\frac{1}{4}K_1(t,\tau)\gamma^{2}(t;t)-\frac{1}{2}K_2(t,\tau)\gamma(t;t)\gamma(\tau;t)=0,\quad \kappa(\tau;T)=0
 	\end{cases}
 \end{equation*}
with
\begin{align*}
K_1(t,\tau):=\frac{\rho^{2}(t)(\theta^{2}(t)\Lambda(\tau;t)+\theta_0^2(t)\beta(\tau;t))}{(\theta^{2}(t)\Lambda(t;t)+\theta_0^2(t)\beta(t;t))^2},\quad\quad
K_2(t,\tau):=\frac{\rho^{2}(t)}{\theta^{2}(t)\Lambda(t;t)+\theta_0^2(t)\beta(t;t)},
\end{align*}
which can be solved explicitly by 
 \begin{equation*}
 	\begin{cases}
 		\Lambda(\tau;t)=\frac{\eta}{2}\bigg[\lambda_2(T-\tau)\exp\left(\int_{t}^{T}2r(s)-\frac{\rho^{2}(s)}{\theta^{2}(s)+\theta_0^2(s)}ds\right)+\int_{t}^T\lambda_1(s-\tau)\exp\left(\int_{t}^{s}2r(v)-\frac{\rho^{2}(v)}{\theta^{2}(v)+\theta_0^2(v)}dv\right)ds\bigg],\\
 		\beta(\tau;t)=0,\\
 		\gamma(\tau;t)=-\lambda_2(T-\tau)\exp\left(\int_{t}^{T}r(s)ds\right)-\int_t^T\lambda_1(s-\tau)\exp\left(\int_{t}^{s}r(v)dv\right)ds,\\
 		\kappa(\tau;t)=\int_{t}^{T}\frac{1}{4}K_1(s,\tau)\gamma^2(s;s)-\frac{1}{2}K_2(s,\tau)\gamma(s;s)\gamma(\tau;s)ds.
 	\end{cases}
 \end{equation*}
 In addition, $(U,W)$ are explicitly given by $U(t;t)=(\theta(t)^2+\theta_0^2(t))\Lambda(t;t)>0$ and $W(t;t)=\theta^{2}(t)\Lambda(t;t)>0$. Therefore, the feedback function of the closed-loop equilibrium strategy is 
 {\small\begin{align*}
 \hat{\alpha}(t,x,\mu)&=-\frac{\rho(t)}{\theta^{2}(t)+\theta_0^2(t)}(x-\overline{\mu})\\&+\frac{\rho(t)}{\eta\theta^{2}(t)}\frac{\lambda_2(T-t)\exp\left(\int_{t}^{T}r(s)ds\right)+\int_t^T\lambda_1(s-t)\exp\left(\int_{t}^{s}r(v)dv\right)ds}{\bigg[\lambda_2(T-t)\exp\left(\int_{t}^{T}2r(s)-\frac{\rho^{2}(s)}{\theta^{2}(s)+\theta_0^2(s)}ds\right)+\int_{t}^T\lambda_1(s-t)\exp\left(\int_{t}^{s}2r(v)-\frac{\rho^{2}(v)}{\theta^{2}(v)+\theta_0^2(v)}dv\right)ds\bigg]}.
 \end{align*}}
Notably, the obtained equilibrium strategy is linear in terms of the wealth variable $x$. Moreover, in the scenario when $x<\bar{\mu}$ that the current wealth performance is lower than the conditional mean, the larger volatility $\theta_0$ of common noise leads to lower equilibrium portfolio in the risky asset. }
\vskip 5pt
 \textbf{Example-2: Inter-bank systemic risk model with non-exponential discount}
 
 We consider an inter-bank systemic risk model where the log-monetary reserve of
 the population in the limiting model is governed by the McKean-Vlasov SDE
 \begin{equation}
 	dX_{t}=[k(\E[X_{t}|W^0]-X_{t})+\alpha_{t}]dt+\sigma\sqrt{1-\rho^2} dB_{t}+\sigma\rho dW^0_t,\quad X_{t_{0}}=\xi\in L^{2}(\Gc;\R).
 \end{equation}
Here, $k\geq0$ is the rate of mean-reversion in the interaction from borrowing and lending between the banks, and $\sigma>0$ is the constant volatility coefficient of the bank reserve, and $W^0$ is the common noise for all banks. Moreover, the representative bank at time $t_{0}$ can control
the rate of borrowing/lending to a central bank with the policy $\alpha$ in order to minimize the cost functional:
\begin{equation*}
\E\left[\int_{t_{0}}^{T}\lambda(t-t_{0})\left(\frac{1}{2}\alpha^{2}_{t}-q\alpha_{t}(\E[X_{t}|W^0]-X_{t})+\frac{\eta}{2}(\E[X_{t}|W^0]-X_{t})^{2}\right)dt+\frac{c}{2}\lambda(T-t_{0})(\E[X_{T}|W^0]-X_{T})^{2}\right],
\end{equation*}
 where $\lambda$ is  a positive, non-decreasing, continuously differentiable function on $[0,T]$ with $\lambda(0)=1$, $q>0$ is a positive parameter for the incentive in borrowing $(\alpha_{t}>0)$ or lending $(\alpha_{t}<0)$, and $\eta>0, c>0$
 are positive parameters for penalizing the departure from the average. This model fits into the LQ framework with $d=m=1$ and 
 \begin{equation*}
 	\begin{split}
 		&b_{0}=0,\, B=-k,\, \overline{B}=k,\, C=1,\, \overline{C}=0,\\
 		&\vartheta=\sigma\sqrt{1-\rho^2}, \, D=\overline{D}=F=\overline{F}=0,\\
 		&\vartheta_0=\sigma\rho, \, D_0=\overline{D}_0=F_0=\overline{F}_0=0,\\
 		&Q(\tau;t)=\frac{\eta}{2}\lambda(t-\tau),\, \overline{Q}(\tau;t)=-\frac{\eta}{2}\lambda(t-\tau),\, R(\tau;t)=\frac{1}{2}\lambda(t-\tau),\,  \overline{R}=0, \\
 		&M(\tau;t)=\frac{q}{2}\lambda(t-\tau),\, \overline{M}(\tau;t)=-\frac{q}{2}\lambda(t-\tau),\\
 		&q=\overline{q}=r=\overline{r}=0, P(\tau)=\frac{c}{2}\lambda(T-\tau),\,\overline{P}(\tau)=-\frac{c}{2}\lambda(T-\tau),\,p=\overline{p}=0.
 	\end{split}
 \end{equation*}
 Note that we cannot apply Proposition \ref{Pro:riccati} here as $M$ and $\overline{M}$ are not $0$. 
 The  system \eqref{riccati1}-\eqref{linearode2} for $\left(\Lambda,\beta,\gamma,\kappa\right)$ can be written in this case as follows:
 \begin{equation*}
 	\begin{cases}
 		\Lambda'(\tau;t)-2k\Lambda(\tau;t)+2\lambda(t-\tau)(\Lambda(t;t)+\frac{q}{2})^{2}-4(\Lambda(\tau;t)+\frac{q}{2}\lambda(t-\tau))(\Lambda(t;t)+\frac{q}{2})\\\quad\quad\quad+\frac{\eta}{2}\lambda(t-\tau)=0,\quad \Lambda(\tau;T)=\frac{c}{2}\lambda(T-\tau),\\
 		\beta'(\tau;t)+2\lambda(t-\tau)\beta^{2}(t;t)-4\beta(\tau;t)\beta(t;t)=0,\quad \beta(\tau;T)=0,\\
 		\gamma'(\tau;t)+2\lambda(t-\tau)\beta(t;t)\gamma(t;t)-2\beta(\tau;t)\gamma(t;t)-2\beta(t;t)\gamma(\tau;t)=0,\,
 		\gamma(\tau;T)=0,\\
 		\kappa'(\tau;t)+\sigma^{2}(1-\rho^2)\Lambda(\tau;t)+\sigma^2\rho^2\gamma(\tau;t)+\frac{1}{2}\lambda(t-\tau)\gamma^{2}(t;t)-\gamma(\tau;t)\gamma(t;t)=0,\, \kappa(\tau;T)=0.
 	\end{cases}
 \end{equation*}
One can see that $\beta=\gamma=0$ and if the first equation of the above system admits a solution $\Lambda$, then $\kappa(\tau;t)=\sigma^{2}(1-\rho^2)\int_{t}^{T}\Lambda(\tau;s)ds$.
 Moreover,  $(U,W)$ are explicitly give by $U(t;t)=W(t;t)=\frac{1}{2}>0$. Therefore, the equilibrium control is given in feedback form by 
 \begin{equation*}
 	\hat{\alpha}_{t}=\hat{\alpha}(t,\hat{X}_{t},\P^{W^0}_{\hat{X}_{t}})=-2(\Lambda(t;t)+\frac{q}{2})(\hat{X}_{t}-\E\left[\hat{X}_{t}|W^0\right]).
 \end{equation*}
 It remains to show the following equation admits a unique solution
 \begin{equation}\label{ricattiexample}
 	\begin{cases}
 		\Lambda'(\tau;t)-2k\Lambda(\tau;t)+2\lambda(t-\tau)(\Lambda(t;t)+\frac{q}{2})^{2}-4(\Lambda(\tau;t)+\frac{q}{2}\lambda(t-\tau))(\Lambda(t;t)+\frac{q}{2})+\frac{\eta}{2}\lambda(t-\tau)=0,\\
   \Lambda(\tau;T)=\frac{c}{2}\lambda(T-\tau),
 \end{cases}
 \end{equation}
 which is guaranteed by the next lemma.
\begin{lemma}\label{lemma:wellposedness:ricattiexample}
	If $k>\max(4C^{2},1)$, i.e., the mean-reversion rate is sufficiently large,
  Eq.\eqref{ricattiexample} admits a unique solution, where $C$ is a constant depending on the discount $\lambda$, parameters $c,q$ and $\eta$.
\end{lemma}
\begin{proof}
	The proof is delegated in Subsection \ref{appendixB}.
\end{proof}

{\section{A Class of Time-inconsistent Non-LQ Extended MFC Problems}\label{sect:nonlq}
For general MFC problems beyond the LQ-type, it is a well-known challenge to deduce the desired regularity of the value function even in a time-consistent framework. We examine in this section a class of time-inconsistent non-LQ extended MFC problem with general distribution-dependent dynamics. To ease the presentation, we restrict our analysis to the one-dimensional setting in which we consider a McKean–Vlasov dynamics governed by the coefficients  
\begin{align*}
    &b(t,x,\mu,a,\lambda) =  b\left(\overline{\mu}\right) + Ca + \overline{C} \overline{\lambda},\\ 
    &\sigma(t,x,\mu,a,\lambda) = \sigma, \quad \sigma_0(t,x,\mu,a,\lambda) = \sigma_0,
\end{align*}  
for \((t,x,\mu,a,\lambda) \in [0,T] \times \mathbb{R} \times \mathcal{P}_2(\mathbb{R}) \times \mathbb{R} \times \mathcal{P}_2(\mathbb{R})\), where \(\overline{\mu} := \int_{\mathbb{R}} x \, \mu(dx)\) and \(\overline{\lambda} := \int_\mathbb{R} a \, \lambda(da)\). Here,  \( C \), \( \overline{C} \), \( \sigma > 0 \), and \( \sigma_0 > 0 \) are constants, while \( b: \mathbb{R} \to \mathbb{R} \) is a Lipschitz continuous function.  The associated cost functions are given by  
\begin{equation*}
 	\begin{split}	
    &f(\tau;t,x,\mu,a,\lambda) =  \overline{Q}(\tau;t)q(\overline{\mu}) + R(\tau;t)a^2 + \overline{R}(\tau;t)r(\overline{\lambda}),\\
    &g(\tau;x,\mu) = \overline{P}(\tau)p(\overline{\mu}),
 	\end{split}
 \end{equation*}  
where \( \overline{Q}, R, \overline{R}: \Delta[0,T] \to \mathbb{R} \) are continuous functions, \( \overline{P}: [0,T] \to \mathbb{R} \) is also a continuous function, \( q, p: \mathbb{R} \to \mathbb{R} \) are bounded continuous functions, and \( r: \mathbb{R} \to \mathbb{R} \) is a convex and differentiable function. Some additional assumptions will be specified later. Then, the functions \(\hat{f}\) and \(\hat{g}\), as defined in \eqref{hat_fg}, are given by  
\begin{equation*}
 	\begin{cases}
 	\hat{f}(\tau;t,\mu,\talpha) = \overline{Q}(\tau;t) q(\overline{\mu}) + R(\tau;t) \int_{\mathbb{R}} \tilde{\alpha}^2(x) \, d\mu(x) + \overline{R}(\tau;t) r\left(\int_{\mathbb{R}} \talpha(x) \, d\mu(x)\right), \\  
 	\hat{g}(\tau;\mu) =  \overline{P}(\tau) p(\overline{\mu}),
 	\end{cases}
\end{equation*}  
for any \((\tau,t,\mu) \in \Delta[0,T] \times \mathcal{P}_2(\mathbb{R})\) and \(\talpha \in L(\mathbb{R};\R)\). Similarly, we assume that \( \text{Lip}(\mathbb{R};\mathbb{R}) \subseteq \mathcal{D} \). It is important to note that the problem considered in this section does not fit the LQ type MFC as the dependence of mean-field term \(\mu\) in
both the state dynamics and the cost functional can be generally nonlinear and nonconvex.

In order to apply Theorem \ref{corollary:verification} to confirm the existence of time-consistent equilibria, we aim to show that the system \eqref{masterequation:valuefunction}–\eqref{masterequation:equilibrium} admits a solution satisfying the regularity conditions in Assumption \ref{assumption:V}. In the model of this section, the system can be written as
 \begin{equation}\label{newsect:masterequation}
 \begin{split}
 &\partial_{t}V(\tau;t,\mu)+\int_{\R}\partial_{\mu}V(\tau;t,\mu)(x)\mu(dx)b(\overline{\mu})+C\int_{\R}\partial_{\mu}V(\tau;t,\mu)(x)\hat{\alpha}(t,x,\mu)\mu(dx)\\
 &\qquad\qquad\quad+\overline{C}\int_{\R}\partial_{\mu}V(\tau;t,\mu)(x)\mu(dx)\int_{\R}\hat{\alpha}(t,x,\mu)d\mu(x)\\
 &\qquad\qquad\quad+\frac{\sigma^2+\sigma_0^2}{2}\int_{\R}\partial_{x}\partial_{\mu}V(\tau;t,\mu)(x)\mu(dx)+\frac{\sigma_0^2}{2}\int_{\R^2}\partial^2_{\mu\mu}V(\tau;t,\mu)(x,x')\mu(dx)\mu(dx')\\
 &\qquad\qquad\quad + \overline{Q}(\tau;t) q(\overline{\mu})+ R(\tau;t) \int_{\mathbb{R}} \hat{\alpha}^2(t,x,\mu) \, d\mu(x)\\&\qquad\qquad\quad  + \overline{R}(\tau;t) r\left(\int_{\mathbb{R}} \hat{\alpha}(t,x,\mu) \, d\mu(x)\right)=0,\quad 0\leq\tau\leq t <T,\ \mu\in\Pc_{2}(\R),\\
    &V(\tau;T,\mu)=  \overline{P}(\tau) p(\overline{\mu}),\quad \tau\in [0,T],\ \mu\in\Pc_{2}(\R)
    \end{split}
 \end{equation}
where \(\hat{\alpha}(t,\cdot,\mu)\) (abbreviated as \(\hat{\alpha}\)) is the minimizer of the functional \( F^{\mu}_{t}(\cdot):L^{2}_{\mu}(\R)\supset Lip(\R;\R)\rightarrow\R\), given by  
\begin{align*}  
    F^{\mu}_{t}(\nu) &:=  C\int_{\R}\partial_{\mu}V(t;t,\mu)(x)\nu(x)\mu(dx)+\overline{C}\int_{\R}\partial_{\mu}V(t;t,\mu)(x)\mu(dx)\int_{\R}\nu(x)d\mu(x)\\&+R(t;t) \int_{\mathbb{R}} \nu^2(x) \, d\mu(x)  + \overline{R}(t;t) r\left(\int_{\mathbb{R}} \nu(x) \, d\mu(x)\right).
\end{align*}  
After some direct calculations, the Gâteaux derivative of \( F^{\mu}_{t} \) at \( \hat{\alpha} \) in the direction \( \theta \in L^{2}_{\mu}(\mathbb{R}) \) is given by
  \begin{equation*}	DF^{\mu}_{t}(\hat{\alpha},\theta):=\lim\limits_{\epsilon\rightarrow0}\frac{F^{\mu}_{t}(\hat{\alpha}+\epsilon\theta)-F^{\mu}_{t}(\hat{\alpha})}{\epsilon}=\int_{\R}\dot{f}^{\mu}_{t}(x,\hat{\alpha})\theta(x)\mu(dx)
  \end{equation*}
  with
  \begin{equation*}
  	\dot{f}^{\mu}_{t}(x,\hat{\alpha})=C\partial_{\mu}V(t;t,\mu)(x)+\overline{C}\int_{\R}\partial_{\mu}V(t;t,\mu)(x)\mu(dx)+2R(t;t)\hat{\alpha}(x)+\overline{R}(t;t)r'\left(\int_{\mathbb{R}} \hat{\alpha}(x) \, d\mu(x)\right).
  \end{equation*}
In the same fashion to analyze the LQ problem in Section \ref{sec:app}, we deduce that the functional \( F^{\mu}_{t} \) attains its infimum at some \( \hat{\alpha} \) such that \( DF^{\mu}_{t}(\hat{\alpha}, \cdot) \) vanishes, leading to
\begin{equation}\label{newsect:alpha}
    C\partial_{\mu}V(t;t,\mu)(x)+\overline{C}\int_{\R}\partial_{\mu}V(t;t,\mu)(x)\mu(dx)+2R(t;t)\hat{\alpha}(x)+\overline{R}(t;t)r'\left(\int_{\mathbb{R}} \hat{\alpha}(x) \, d\mu(x)\right)=0.
\end{equation}
Integrating both sides of the above equation with respect to \(\mu\) yields
\begin{equation*}
    R(t;t)\int_{\mathbb{R}} \hat{\alpha}(x) \, d\mu(x)+\frac{1}{2}\overline{R}(t;t)r'\left(\int_{\mathbb{R}} \hat{\alpha}(x) \, d\mu(x)\right)+\frac{C+\overline{C}}{2}\int_{\R}\partial_{\mu}V(t;t,\mu)(x)\mu(dx)=0.
\end{equation*}
To derive an explicit expression for \( \int_{\mathbb{R}} \hat{\alpha}(x) \, d\mu(x) \), we impose the following assumption.  {\begin{assumption} \label{newassump:implicit}  
\begin{itemize}
    \item[(i)] The functions \( R(\cdot;\cdot) \) and \( \overline{R}(\cdot;\cdot) \) are constant along the diagonal, i.e., there exist two constants \( R \) and \( \overline{R} \) such that  
    \[
    R := R(t;t), \quad \overline{R} := \overline{R}(t;t), \quad \forall t \in [0,T].
    \]
    \item[(ii)] The function \( r \) is twice continuously differentiable. Additionally, there exists a positive constant \( \epsilon_0 > 0 \) such that  
    \[
    \inf_{x\in\R} \left| R + \frac{1}{2} \overline{R} r''(x) \right| \geq \epsilon_0.
    \]
Furthermore, for any $y \in \mathbb{R}$, there exists a unique $z \in \mathbb{R}$ satisfying
$$
R z + \frac{1}{2} \overline{R} \, r'(z) + (C + \overline{C}) y = 0.
$$
\end{itemize}
\end{assumption}  
Condition (i) in Assumption \ref{newassump:implicit} is introduced for simplicity and can be generalized, though at the cost of more involved arguments.  While somewhat restrictive, it covers cases of time inconsistency due to non-exponential discounting, such as when \( R(\tau;t) = R \lambda(t - \tau) \) and \( \overline{R}(\tau;t) = \overline{R} \lambda(t - \tau) \) for some positive continuous function \( \lambda: \mathbb{R}^+ \to \mathbb{R} \).  Under condition (ii) in Assumption \ref{newassump:implicit}, the implicit function theorem guarantees the existence of a function \( \rho \in C^1(\mathbb{R}) \) satisfying for all \( y \in \mathbb{R} \):
\begin{equation}\label{newsect:rhoimplict}
    R \rho(y) + \frac{1}{2} \overline{R} r'(\rho(y)) + (C + \overline{C}) y = 0.
\end{equation}
Moreover, it follows that  
\begin{equation}\label{newsect:rho'}
    \rho'(y)=-\frac{C+\overline{C}}{R+\frac{1}{2}\overline{R}r''(\rho(y))}
\end{equation}
which implies
\begin{equation}\label{newsect:rhobound}
    |\rho'|^0 := \sup_{y \in \mathbb{R}} |\rho'(y)| \leq \frac{|C + \overline{C}|}{\epsilon_0}.
\end{equation}
} Consequently, we obtain  
{\small\begin{equation} \label{newsect:alpha2}  
\begin{split}
    \int_{\mathbb{R}} \hat{\alpha}(t,x,\mu) d\mu(x) &= \rho \left( \frac{1}{2} \int_{\mathbb{R}} \partial_{\mu} V(t;t,\mu)(x) \mu(dx) \right)\\&=-\frac{1}{2R}\left[\overline{R}r'\left(\rho \left(  \frac{1}{2} \int_{\mathbb{R}} \partial_{\mu} V(t;t,\mu)(x) \mu(dx) \right)\right)+(C+\overline{C})\int_{\R}\partial_{\mu}V(t;t,\mu)(x)\mu(dx)\right].
\end{split}
\end{equation}  }
Using \eqref{newsect:alpha}, we get the explicit form  
\begin{align} \label{newsect:alpha3}  
    \hat{\alpha}(t,x,\mu) &= -\frac{1}{2R} \bigg[ C \partial_{\mu} V(t;t,\mu)(x) + \overline{C} \int_{\mathbb{R}} \partial_{\mu} V(t;t,\mu)(x) \mu(dx) \\  
    &\quad + \overline{R} r' \left( \rho \left( \frac{1}{2} \int_{\mathbb{R}} \partial_{\mu} V(t;t,\mu)(x) \mu(dx) \right) \right) \bigg]. \nonumber  
\end{align}
\begin{remark}  \label{newsect:remark}
Here, we present some examples of \( r \) that satisfy Assumption \ref{newassump:implicit} and yield an explicit form of \( \rho \):
\begin{itemize}
\item If \( r(x) = x \), then it follows directly that  
\[
\rho(x) = -\frac{1}{R} \left( \frac{1}{2} \overline{R} + (C + \overline{C}) x \right).
\]  
\item If \( r(x) = x^2 \), then we obtain  
\[
\rho(x) = -\frac{C + \overline{C}}{R + \overline{R}} x.
\]  
\end{itemize}
\end{remark}
Utilizing the expressions in \eqref{newsect:alpha2} and \eqref{newsect:alpha3}, the equation \eqref{newsect:masterequation} can be further simplified as
 \begin{equation}\label{newsect:masterequation2}
 \begin{split}
 &\partial_{t}V(\tau;t,\mu)+\int_{\R}\partial_{\mu}V(\tau;t,\mu)(x)\mu(dx)b(\overline{\mu})\\&\qquad\qquad\quad+(C+\overline{C})\int_{\R}\partial_{\mu}V(\tau;t,\mu)(x)\mu(dx)\rho \left(  \frac{1}{2} \int_{\mathbb{R}} \partial_{\mu} V(t;t,\mu)(x) \mu(dx) \right)\\
 &\qquad\qquad\quad-\frac{C^2}{2R}\int_{\R}\partial_{\mu}V(\tau;t,\mu)(x)\left(\partial_{\mu}V(t;t,\mu)(x)-\int_{\R}\partial_{\mu}V(t;t,\mu)(x')\mu(dx')\right)\mu(dx)\\
 &\qquad\qquad\quad+\frac{\sigma^2+\sigma_0^2}{2}\int_{\R}\partial_{x}\partial_{\mu}V(\tau;t,\mu)(x)\mu(dx)+\frac{\sigma_0^2}{2}\int_{\R^2}\partial^2_{\mu\mu}V(\tau;t,\mu)(x,x')\mu(dx)\mu(dx')\\
 &\qquad\qquad\quad+ \overline{Q}(\tau;t) q(\overline{\mu})+ R(\tau;t)  \rho^2 \left(\frac{1}{2} \int_{\mathbb{R}} \partial_{\mu} V(t;t,\mu)(x) \mu(dx) \right)\\
 &\qquad\qquad\quad+\frac{R(\tau;t)}{4R^2}C^2\int_{\R}\partial_{\mu}V(t;t,\mu)(x)\left(\partial_{\mu}V(t;t,\mu)(x)-\int_{\R}\partial_{\mu}V(t;t,\mu)(x')\mu(dx')\right)\mu(dx)     \\
 &\qquad\qquad\quad  + \overline{R}(\tau;t) r\left(\rho \left( \frac{1}{2} \int_{\mathbb{R}} \partial_{\mu} V(t;t,\mu)(x) \mu(dx)\right)\right)=0,\quad 0\leq\tau\leq t <T,\ \mu\in\Pc_{2}(\R),\\
    &V(\tau;T,\mu)=  \overline{P}(\tau) p(\overline{\mu}),\quad \tau\in [0,T],\ \mu\in\Pc_{2}(\R).
    \end{split}
 \end{equation}
As a key step, we consider the ansatz for the solution \( V \) in the form of 
\begin{equation}\label{newsect:ansatz}
 	V(\tau;t,\mu) = \Phi(\tau;t,\overline{\mu}),
\end{equation}  
where  \( \Phi: \Delta[0,T] \times \mathbb{R} \to \mathbb{R} \) is the function to be determined. Substituting the ansatz \eqref{newsect:ansatz} into \eqref{newsect:masterequation2}, we deduce a nonlocal nonlinear PDE problem for the function \( \Phi \): 
\begin{equation}\label{newsect:nonlocalpde}
    \begin{split}
        &\partial_t\Phi(\tau;t,x)+\left[b(x)+(C+\overline{C})\rho\left(\frac{1}{2}\partial_x\Phi(t;t,x)\right)\right]\partial_x\Phi(\tau;t,x)+\frac{\sigma^2_0}{2}\partial_{xx}\Phi(\tau;t,x)\\
        &\quad\qquad\qquad+R(\tau;t)\rho^2\left(\frac{1}{2}\partial_x\Phi(t;t,x)\right)+\overline{R}(\tau;t)r\left(\rho\left(\frac{1}{2}\partial_x\Phi(t;t,x)\right)\right)\\&\qquad\qquad\quad+\overline{Q}(\tau;t)q(x)=0,\,0\leq\tau\leq t<T,\,\ x\in\R,\\
        &\Phi(\tau;T,x)=\overline{P}(\tau)p(x),\,(\tau,x)\in[0,T]\times\R.
    \end{split}
\end{equation}
Note that a similar type of nonlocal nonlinear PDE has been studied in \cite{lei_nonlocal_2023,lei_nonlocality_2024} in a different context. 
However, as stated in Theorem 3.5 of \cite{lei_nonlocality_2024}, under sufficiently smooth coefficients, only the existence of a short-time solution for small $T$ can be concluded therein. Although there are some other results on the global well-posedness of nonlocal nonlinear PDE problems (e.g., \cite{wei_time-inconsistent_2017,zhuchao2020,wang_backward_2019,lei_nonlocality_2024}), these studies often impose strong growth conditions on the nonlinear operator, which are not satisfied in our case. To reconcile the challenge in our model, we provide a new way to tackle the PDE problem by exploiting the special structure of \eqref{newsect:nonlocalpde}.

Let us first revisit some notations that are commonly adopted to study local parabolic equations such as \cite{ladyzenskaja_linear_nodate,wang_backward_2019}. For any suitable function \( \varphi : [S, T] \times \mathbb{R} \to \mathbb{R} \), with \( \alpha \in (0,1) \) and \( S \in [0, T) \), we define the following norms and seminorms:
\begin{align*}
&|\varphi|^{(0)} = \|\varphi\|_{L^{\infty}([S, T] \times \mathbb{R})}, \,\,
|\varphi|^{(1)} = |\varphi|^{(0)} + |\varphi_x|^{(0)}, \,\,
|\varphi|^{(2)} = |\varphi|^{(1)} + |\varphi_s|^{(0)} + |\varphi_{xx}|^{(0)},\\
&\langle \varphi \rangle_s^{\left(\frac{\alpha}{2}\right)} = \sup_{\substack{s, s' \in [S, T], x \in \mathbb{R} \\ s \neq s'}} \frac{|\varphi(s, x) - \varphi(s', x)|}{|s - s'|^{\frac{\alpha}{2}}}, \quad
\langle \varphi \rangle_x^{(\alpha)} = \sup_{\substack{s \in [S, T], x, x' \in \mathbb{R} \\ 0 < |x - x'| \leq 1}} \frac{|\varphi(s, x) - \varphi(s, x')|}{|x - x'|^\alpha}, \\
&\langle \varphi \rangle^{(\alpha)} = \langle \varphi \rangle_s^{\left(\frac{\alpha}{2}\right)} + \langle \varphi \rangle_x^{(\alpha)}, \quad
|\varphi|^{(\alpha)} = |\varphi|^{(0)} + \langle \varphi \rangle^{(\alpha)}, \\
&|\varphi|^{(1+\alpha)} = |\varphi|^{(1)} + \langle \varphi_x \rangle^{(\alpha)} + \langle \varphi \rangle_s^{\left(\frac{1 + \alpha}{2}\right)}, \\
&|\varphi|^{(2+\alpha)} = |\varphi|^{(2)} + \langle \varphi_s \rangle^{(\alpha)} + \langle \varphi_{xx} \rangle^{(\alpha)} + \langle \varphi_x \rangle_s^{\left(\frac{1 + \alpha}{2}\right)}.
\end{align*}
When \( [S, T] \times \mathbb{R} \) needs to be emphasized, we write \( |\varphi|_{[S, T] \times \mathbb{R}}^{(1)} \), etc. Additionally, whenever no confusion arises, we do not distinguish between \( |\varphi|_{[S, T] \times \mathbb{R}}^{(1)} \) and \( |\varphi|_{\mathbb{R}}^{(1)} \) (or between \( |\varphi|_{[S, T] \times \mathbb{R}}^{(1)} \) and \( |\varphi|_{[S,T]}^{(1)} \)). 
We define the space \( C^{k+\alpha}([S, T] \times \mathbb{R}) \) for \( k = 0, 1, 2 \) as follows:
\[
C^{k+\alpha}([S, T] \times \mathbb{R}) := \left\{ \varphi : [S, T] \times \mathbb{R} \to \mathbb{R} \mid |\varphi|_{[S, T] \times \mathbb{R}}^{(k+\alpha)} < \infty \right\}.
\]
Let \( S \in [0, T) \) and let \( \phi \) be a function defined on \( \T[S,T] \times \mathbb{R} \), where \( \T[S,T] \) represents the trapezoidal region \( \Delta[0,T] \cap ([0,T] \times [S,T]) \). We introduce the following norms:
\begin{align*}
&[\phi]_{[S, T]}^{(k+\alpha)} := \sup_{\tau \in [0, T]} \left\{|\phi(\tau; \cdot, \cdot)|_{[\tau\vee S,T] \times \mathbb{R}}^{(k+\alpha)}\right\}, \\
&\|\phi\|_{[S, T]}^{(k+\alpha)} := \sup_{\tau \in [0, T]} \left\{|\phi(\tau; \cdot, \cdot)|_{[\tau\vee S, T] \times \mathbb{R}}^{(k+\alpha)} + \left|\partial_{\tau}\phi(\tau; \cdot, \cdot)\right|_{[\tau\vee S,T] \times \mathbb{R}}^{(k+\alpha)}\right\},
\end{align*}
where \( k = 0, 1, 2 \). These norms induce the following normed spaces:
\[
\begin{aligned}
& \Theta_{[S, T]}^{(k+\alpha)} := \left\{\theta(\cdot; \cdot, \cdot) \in C\left(\T[S, T] \times \mathbb{R}\right) : [\theta]_{[S, T]}^{(k+\alpha)} < \infty \right\}, \\
& \Xi_{[S, T]}^{(k+\alpha)} := \left\{\omega(\cdot; \cdot, \cdot) \in C\left(\T[S, T] \times \mathbb{R}\right) : \|\omega\|_{[S, T]}^{(k+\alpha)} < \infty \right\}.
\end{aligned}
\]
It is straightforward to verify that both \( [\phi]_{[S, T]}^{(k+\alpha)} \) and \( \|\phi\|_{[S, T]}^{(k+\alpha)} \) are norms, under which \( \Theta_{[S, T]}^{(k+\alpha)} \) and \( \Xi_{[S, T]}^{(k+\alpha)} \) are Banach spaces, respectively.
\begin{assumption}\label{newsect:asspcoef}
    \begin{itemize}
        \item[(i)] The function $b$ is bounded and Lipschitz continuous, $q \in C^{\alpha}(\mathbb{R})$, $p \in C^{2+\alpha}(\mathbb{R})$, and $\overline{P} \in C^{1}([0,T])$. Additionally, \( R(\cdot;\cdot), \overline{R}(\cdot;\cdot), \overline{Q}(\cdot;\cdot) \) belong to \( \Xi_{[0, T]}^{(\alpha)} \), that is, they satisfy  
       \begin{equation*}
           \sup_{\tau\in[0,T]}\left\{|\phi(\tau; \cdot)|_{[\tau, T] }^{(\alpha)} + \left|\partial_{\tau}\phi(\tau; \cdot)\right|_{[\tau,T]}^{(\alpha)}\right\}<\infty,
       \end{equation*}
       for \( \phi(\cdot;\cdot) = R(\cdot;\cdot), \overline{R}(\cdot;\cdot), \overline{Q}(\cdot;\cdot) \).
        \item[(ii)] The function \( r \) is  three times continuously differentiable. Moreover, one of the following conditions holds:
        \begin{itemize}
            \item[(1)] The derivative \( r' \) is bounded, i.e., \( |r'|^{(0)} < \infty \).
            \item[(2)] The function \( r \) is quadratic growth, that is, there exist constants \( \tilde{K} \) and \( \beta \) such that  
            \begin{equation*}
                |r(x)|\leq \tilde{K}|x|^2+\beta, \quad \forall x\in \mathbb{R}.
            \end{equation*}
            Furthermore, \( r \) satisfies 
            \begin{equation*}
                 |r(x)-r(y)|\leq \tilde{K}|x+y||x-y|, \quad \forall x,y\in \mathbb{R}.
            \end{equation*}
        \end{itemize}
    \end{itemize} 
\end{assumption}
 Condition (i) in Assumption \ref{newsect:asspcoef} imposes some mild regularity conditions on the coefficients, which are standard in the literature on nonlocal PDEs; see, for example, \cite{lei_nonlocal_2023,lei_nonlocality_2024}. Condition (ii) is a technical assumption that facilitats our proof. It is clear that the choices of \( r \) in Remark \ref{newsect:remark} satisfy this condition.
\begin{proposition}\label{newsect:lemma}  
   Suppose that Assumption \ref{newassump:implicit} and Assumption \ref{newsect:asspcoef} hold. Then, there exists a constant \( \hat{K} \)  such that if \( |C + \overline{C}| < \hat{K} \), the equation \eqref{newsect:nonlocalpde} admits a unique solution satisfying Assumption \ref{assumption:V}. Moreover,  \( \hat{\alpha} \) defined by 
\begin{equation}\label{newsect:hatalpha}
    \hat{\alpha}(t,x,\mu) = -\frac{1}{2R} \left[ (C+\overline{C}) \partial_{x} \Phi(t;t,\overline{\mu}) + \overline{R} r' \left( \rho \left( \frac{1}{2}  \partial_{x} \Phi(t;t,\overline{\mu}) \right) \right) \right]
\end{equation}
is an equilibrium strategy.
\end{proposition}

\begin{proof}
   The proof is delegated in Subsection \ref{appendixC}.
\end{proof}}

\section{Technical Proofs}\label{auxproofs}
This section reports technical proofs of some main results in previous sections. 

\subsection{The proof of Proposition \ref{Pro:riccati}}\label{appendixA}
 
Note that system \eqref{riccati1}-\eqref{linearode2}  is decoupled, we can first solve \eqref{riccati1} and then proceed to solve \eqref{riccati2}, \eqref{linearode1} and \eqref{linearode2}. Let us first address \eqref{riccati1} by modifying the convergence method proposed in \cite{Yong2017} to fit into our framework. Denote
 \begin{equation*}
 	\begin{cases}
 		\Theta(t):=U(t;t)^{-1}S^{\top}(t;t),\,
 		\tilde{B}(t):=B(t)-C(t)\Theta(t),\\
 		\tilde{D}(t):=D(t)-F(t)\Theta(t),\,
 		\tilde{D}_0(t):=D_0(t)-F_0(t)\Theta(t),\\
 		\tilde{Q}(\tau;t):=Q(\tau;t)+\Theta^{\top}(t)R(\tau;t)\Theta(t).
 	\end{cases}
 \end{equation*}
 A straightforward computation shows that  \eqref{riccati1}  is equivalent to the Lyapunov equation with parameter $\tau$:
 {\small
 \begin{equation}\label{Lyapunov:1}
 	\begin{cases}
 		\Lambda'(\tau;t)+\Lambda(\tau;t)\tilde{B}(t)+\tilde{B}^{\top}(t)\Lambda(\tau;t)+\tilde{D}^{\top}(t)\Lambda(\tau;t)\tilde{D}(t)+\tilde{D}_0^{\top}(t)\Lambda(\tau;t)\tilde{D}_0(t)+\tilde{Q}(\tau;t)=0,\quad 0\leq\tau\leq t<T,\\
 		\Lambda(\tau;T)=P(\tau),\quad \tau\in[0,T].
 	\end{cases}
 \end{equation}}
  For a partition $\Delta:0=t_{0}<t_{1}<\cdots<t_{N-1}<t_{N}=T$ of $[0,T]$, let us introduce the finite sequences 
 \begin{equation*}
 	\hat{\Lambda}_{k}(\cdot),\,\ 0\leq k\leq N-1, \,\,\ \tilde{\Lambda}_{l}(\cdot),\,\, 0\leq l\leq N-2,
 \end{equation*}
 where $\hat{\Lambda}_{k}$ defined on $(t_{k},t_{k+1}]$ satisfies the following Riccati equation, whose well-posedness is guaranteed by the condition \eqref{condition:pd} (see \cite{book:yong1999, wonham1968}):
 \begin{equation*}
 	\begin{cases}
 		\hat{\Lambda}_{k}'(t)+\hat{\Lambda}_{k}(t)B(t)+B^{\top}(t)\hat{\Lambda}_{k}(t)+D^{\top}(t)\hat{\Lambda}_{k}(t)D(t)+D_0^{\top}(t)\hat{\Lambda}_{k}(t)D_0(t)+Q(t_{k};t)\\ -(D^{\top}(t)\hat{\Lambda}_{k}(t)F(t)+D_0^{\top}(t)\hat{\Lambda}_{k}(t)F_0(t)+\hat{\Lambda}_{k}(t)C(t))(R(t_{k};t)+F^{\top}(t)\hat{\Lambda}_{k}(t)F(t)+F_0^{\top}(t)\hat{\Lambda}_{k}(t)F_0(t))^{-1}\\(F^{\top}(t)\hat{\Lambda}_{k}(t)D(t)+F_0^{\top}(t)\hat{\Lambda}_{k}(t)D_0(t)+C^{\top}(t)\hat{\Lambda}_{k}(t))
 		=0,\quad t\in(t_{k},t_{k+1}),\\
 	\hat{\Lambda}_{k}(t_{k+1})=\tilde{\Lambda}_{k}(t_{k+1})
 	\end{cases}
 \end{equation*}
 with the convention that 
 $\tilde{\Lambda}_{N-1}(t_{N})=P(t_{N-1})$,
 and $\tilde{\Lambda}_{l}$ defined on $(t_{l+1},t_{N}]$ satisfies the Lyapunov equation:
 \begin{equation}\label{Lyapunov:approximation}
 	\begin{cases}
 		\tilde{\Lambda}_{l}'(t)+\tilde{\Lambda}_{l}(t)(B(t)-C(t)\Theta^{\Delta}(t))+(B(t)-C(t)\Theta^{\Delta}(t))^{\top}\tilde{\Lambda}_{l}(t)+Q(t_{l};t)+(\Theta^{\Delta})^{\top}(t)R(t_{l};t)\Theta^{\Delta}(t)\\+(D(t)-F(t)\Theta^{\Delta}(t))^{\top}\tilde{\Lambda}_{l}(t)(D(t)-F(t)\Theta^{\Delta}(t))+(D_0(t)-F_0(t)\Theta^{\Delta}(t))^{\top}\tilde{\Lambda}_{l}(t)(D_0(t)-F_0(t)\Theta^{\Delta}(t))=0,\\
 		\tilde{\Lambda}_{l}(t_{N})=P(t_{l}),
 	\end{cases}
 \end{equation}  	
 where, for  $0\leq l\leq N-1$ and $t\in[t_{l},t_{l+1}]$,
 \begin{equation*}
 	\Theta^{\Delta}(t):=\left[F^{\top}(t)\hat{\Lambda}_{l}(t)F(t)+F_0^{\top}(t)	\hat{\Lambda}_{l}(t)F_0(t)+R(t_l;t)\right]^{-1}\left[F^{\top}(t)\hat{\Lambda}_{l}(t)D(t)+F_0^{\top}(t)\hat{\Lambda}_{l}(t)D_0(t)+C^{\top}(t)\hat{\Lambda}_{l}(t)\right].
 \end{equation*}
 Denote 
 \begin{equation*}
 	\begin{cases}
 		Q^{\Delta}(t):=\sum_{k=0}^{N-1}Q(t_{k};t)I_{(t_{k},t_{k+1}]}(t),\, t\in[0,T],\\
 	    R^{\Delta}(t):=\sum_{k=0}^{N-1}R(t_{k};t)I_{(t_{k},t_{k+1}]}(t),\, t\in[0,T],\\
 		\hat{\Lambda}^{\Delta}(t):=\sum_{k=0}^{N-1}\hat{\Lambda}_{k}(t)I_{(t_{k},t_{k+1}]}(t),\, t\in[0,T].
 	\end{cases}
 \end{equation*}
 Then
 \begin{equation*}
 	\begin{cases}
 		(\hat{\Lambda}^{\Delta})'(t)+\hat{\Lambda}^{\Delta}(t)B(t)+B^{\top}(t)\hat{\Lambda}^{\Delta}(t)+D^{\top}(t)\hat{\Lambda}^{\Delta}(t)D(t)+D_0^{\top}(t)\hat{\Lambda}^{\Delta}(t)D_0(t)+Q^{\Delta}(t)\\ -(D^{\top}(t)\hat{\Lambda}^{\Delta}(t)F(t)+D_0^{\top}(t)\hat{\Lambda}^{\Delta}(t)F_0(t)+\hat{\Lambda}^{\Delta}(t)C(t))(R^{\Delta}(t)+F^{\top}(t)\hat{\Lambda}^{\Delta}(t)F(t)+F_0^{\top}(t)\hat{\Lambda}^{\Delta}(t)F_0(t))^{-1}\\(F^{\top}(t)\hat{\Lambda}^{\Delta}(t)D(t)+F_0^{\top}(t)\hat{\Lambda}^{\Delta}(t)D_0(t)+C^{\top}(t)\hat{\Lambda}^{\Delta}(t))
 		=0,\quad t\in\cup_{k=0}^{N-1}(t_{k},t_{k+1}),\\
 		\hat{\Lambda}^{\Delta}(t_{N})=P(t_{N-1}),\, 	\hat{\Lambda}^{\Delta}(t_{N-1})=\tilde{\Lambda}_{N-2}(t_{N-1}),\cdots,	\hat{\Lambda}^{\Delta}(t_{1})=\tilde{\Lambda}_{0}(t_{1}),
 	\end{cases}
 \end{equation*}
 which is equivalent to
{\small
 \begin{equation*}
 	\begin{cases}
 		(\hat{\Lambda}^{\Delta})'(t)+\hat{\Lambda}^{\Delta}(t)(B(t)-C(t)\Theta^{\Delta}(t))+(B(t)-C(t)\Theta^{\Delta}(t))^{\top}\hat{\Lambda}^{\Delta}(t)+Q^{\Delta}(t)+(\Theta^{\Delta})^{\top}(t)R^{\Delta}(t)\Theta^{\Delta}(t)\\+(D(t)-F(t)\Theta^{\Delta}(t))^{\top}\hat{\Lambda}^{\Delta}(t)(D(t)-F(t)\Theta^{\Delta}(t))+(D_0(t)-F_0(t)\Theta^{\Delta}(t))^{\top}\hat{\Lambda}^{\Delta}(t)(D_0(t)-F_0(t)\Theta^{\Delta}(t))=0,\\
 		t\in\cup_{k=0}^{N-1}(t_{k},t_{k+1}),\\
 	\hat{\Lambda}^{\Delta}(t_{N})=P(t_{N-1}),\, 	\hat{\Lambda}^{\Delta}(t_{N-1})=\tilde{\Lambda}_{N-2}(t_{N-1}),\cdots,	\hat{\Lambda}^{\Delta}(t_{1})=\tilde{\Lambda}_{0}(t_{1}).
 	\end{cases}
 \end{equation*}}
Define
 \begin{equation*}
 	\tilde{\Lambda}^{\Delta}(\tau;t)=\sum_{k=0}^{N-2}\tilde{\Lambda}_{k}(t)I_{(t_{k+1},t_{k+2}]}(\tau),\,\, 0\leq\tau\leq t\leq T
 \end{equation*}
 with the extension on $[0,T]\times[0,T]$ by
 	$\tilde{\Lambda}^{\Delta}(\tau;t)=\tilde{\Lambda}^{\Delta}(\tau;\tau)$, $t\in[0,\tau]$.
Note that $t\rightarrow\tilde{\Lambda}^{\Delta}(\tau;t)$ is continuous on $[0,T]$, whereas $t\rightarrow\hat{\Lambda}^{\Delta}(t)$ might have jumps at $t=t_{k}$.

 In the following, we show 
 \begin{equation}\label{convergence}
 	\lim\limits_{\Vert\Delta\Vert\rightarrow0}(|\hat{\Lambda}^{\Delta}(t)-\Lambda(t;t)|+|\tilde{\Lambda}^{\Delta}(\tau;t)-\Lambda(\tau;t)|)=0,
 \end{equation} 	
 uniformly in $(\tau,t)\in[0,T]\times[0,T]$	with $\Lambda$ satisfying \eqref{riccati1}, thereby proving the existence of the solution to \eqref{riccati1}. The approach is based on an extension of the Arzelà-Ascoli theorem, as outlined in Lemma 6.3 of \cite{Yong2017}. We will also rely on the subsequent two lemmas.
 \begin{lemma}\label{lemma:uniformbound}
 	Let conditions \eqref{condition:pd} and \eqref{condition:monotonicity} hold.  Then  $(\hat{\Lambda}^{\Delta},\tilde{\Lambda}^{\Delta})$ is uniformly bounded in $\Delta$.
 \end{lemma}
 \begin{proof}
 First,	on $(t_{N-1},t_{N}]$, we compare
 {\small
  \begin{equation*}
 		\begin{cases}
 					\hat{\Lambda}_{N-1}'(t)+\hat{\Lambda}_{N-1}(t)(B(t)-C(t)\Theta^{\Delta}(t))+(B(t)-C(t)\Theta^{\Delta}(t))^{\top}\hat{\Lambda}_{N-1}(t)+Q(t_{N-1};t)+(\Theta^{\Delta})^{\top}(t)R(t_{N-1};t)\Theta^{\Delta}(t)\\+(D(t)-F(t)\Theta^{\Delta}(t))^{\top}\hat{\Lambda}_{N-1}(t)(D(t)-F(t)\Theta^{\Delta}(t))+(D_0(t)-F_0(t)\Theta^{\Delta}(t))^{\top}\hat{\Lambda}_{N-1}(t)(D_0(t)-F_0(t)\Theta^{\Delta}(t))=0,\\
 			\hat{\Lambda}_{N-1}(t_{N})=P(t_{N-1})
 		\end{cases}
 	\end{equation*}} with
 	\begin{equation*}
 		\begin{cases}
 		\tilde{\Lambda}_{l}'(t)+\tilde{\Lambda}_{l}(t)(B(t)-C(t)\Theta^{\Delta}(t))+(B(t)-C(t)\Theta^{\Delta}(t))^{\top}\tilde{\Lambda}_{l}(t)+Q(t_{l};t)+(\Theta^{\Delta})^{\top}(t)R(t_{l};t)\Theta^{\Delta}(t)\\+(D(t)-F(t)\Theta^{\Delta}(t))^{\top}\tilde{\Lambda}_{l}(t)(D(t)-F(t)\Theta^{\Delta}(t))+(D_0(t)-F_0(t)\Theta^{\Delta}(t))^{\top}\tilde{\Lambda}_{l}(t)(D_0(t)-F_0(t)\Theta^{\Delta}(t))=0,\\
 		\tilde{\Lambda}_{l}(t_{N})=P(t_{l}).
 		\end{cases}
 	\end{equation*}  	
 	Because  $l\rightarrow Q(t_{l};\cdot)$ and  $l\rightarrow R(t_{l};\cdot)$, and  $l\rightarrow P(t_{l})$ are non-decreasing, we have
 	\begin{equation*}
 		0\leq \tilde{\Lambda}_{l}(t)\leq\tilde{\Lambda}_{l+1}(t)\leq\cdots\leq \tilde{\Lambda}_{N-2}(t)\leq \hat{\Lambda}_{N-1}(t)=\hat{\Lambda}^{\Delta}(t),\, t\in(t_{N-1},t_{N}],\, 0\leq l\leq N-2.
 	\end{equation*}
 Next,	on $(t_{N-2},t_{N-1}]$, we compare 
{\small \begin{equation*}
 		\begin{cases}
 			\hat{\Lambda}_{N-2}'(t)+\hat{\Lambda}_{N-2}(t)(B(t)-C(t)\Theta^{\Delta}(t))+(B(t)-C(t)\Theta^{\Delta}(t))^{\top}\hat{\Lambda}_{N-2}(t)+Q(t_{N-2};t)+(\Theta^{\Delta})^{\top}(t)R(t_{N-2};t)\Theta^{\Delta}(t)\\+(D(t)-F(t)\Theta^{\Delta}(t))^{\top}\hat{\Lambda}_{N-2}(t)(D(t)-F(t)\Theta^{\Delta}(t))+(D_0(t)-F_0(t)\Theta^{\Delta}(t))^{\top}\hat{\Lambda}_{N-2}(t)(D_0(t)-F_0(t)\Theta^{\Delta}(t))=0,\\
 			\hat{\Lambda}_{N-2}(t_{N-1})=\tilde{\Lambda}_{N-2}(t_{N-1})
 		\end{cases}
 	\end{equation*}}
 	with
 	\begin{equation*}
 		\begin{cases}
 			\tilde{\Lambda}_{l}'(t)+\tilde{\Lambda}_{l}(t)(B(t)-C(t)\Theta^{\Delta}(t))+(B(t)-C(t)\Theta^{\Delta}(t))^{\top}\tilde{\Lambda}_{l}(t)+Q(t_{l};t)+(\Theta^{\Delta})^{\top}(t)R(t_{l};t)\Theta^{\Delta}(t)\\+(D(t)-F(t)\Theta^{\Delta}(t))^{\top}\tilde{\Lambda}_{l}(t)(D(t)-F(t)\Theta^{\Delta}(t))+(D_0(t)-F_0(t)\Theta^{\Delta}(t))^{\top}\tilde{\Lambda}_{l}(t)(D_0(t)-F_0(t)\Theta^{\Delta}(t))=0,\\
 			\tilde{\Lambda}_{l}(t_{N-1})\leq \tilde{\Lambda}_{N-2}(t_{N-1}).
 		\end{cases}
 	\end{equation*}  	
 	Using the monotonicity of $l\rightarrow Q(t_{l};\cdot)$,  $l\rightarrow R(t_{l};\cdot)$ and  $l\rightarrow P(t_{l})$ again, we have
 	\begin{equation*}
 		0\leq \tilde{\Lambda}_{l}(t)\leq\tilde{\Lambda}_{l+1}(t)\leq\cdots\leq \tilde{\Lambda}_{N-3}(t)\leq \hat{\Lambda}_{N-2}(t)=\hat{\Lambda}^{\Delta}(t),\, t\in(t_{N-2},t_{N-1}],\, 0\leq l\leq N-3.
 	\end{equation*}
 	Inductively, we can further obtain 
 	\begin{equation*}
 		0\leq \tilde{\Lambda}_{l}(t)\leq\tilde{\Lambda}_{l+1}(t)\leq\cdots\leq \tilde{\Lambda}_{k}(t)\leq \hat{\Lambda}_{k+1}(t)=\hat{\Lambda}^{\Delta}(t),\, t\in(t_{k+1},t_{k+2}],\, 0\leq l\leq k\leq N-2.
 	\end{equation*}
 	Therefore,
 	\begin{equation*}
 		0\leq \tilde{\Lambda}_{l}(t)\leq\hat{\Lambda}^{\Delta}(t),\,\ t\in(t_{l+1},t_{N}],\, 0\leq l\leq N-2. 
 	\end{equation*}
 	We now introduce, for $t\in[0,T]$, 
 	\begin{equation*}
 		\begin{cases}
 			(\Pi^{\Delta})'(t)+\Pi^{\Delta}(t)B(t)+B^{\top}(t)\Pi^{\Delta}(t)+D^{\top}(t)\Pi^{\Delta}(t)D(t)+D_0^{\top}(t)\Pi^{\Delta}(t)D_0(t)+Q^{\Delta}(t)=0,\\
 			\Pi^{\Delta}(T)=P(T).
 		\end{cases}
 	\end{equation*}
 	Comparing the above with 
 	\begin{equation*}
 		\begin{cases}
 			(\hat{\Lambda}^{\Delta})'(t)+\hat{\Lambda}^{\Delta}(t)B(t)+B^{\top}(t)\hat{\Lambda}^{\Delta}(t)+D^{\top}(t)\hat{\Lambda}^{\Delta}(t)D(t)+D_0^{\top}(t)\hat{\Lambda}^{\Delta}(t)D_0(t)+Q^{\Delta}(t)\\ -(D^{\top}(t)\hat{\Lambda}^{\Delta}(t)F(t)+D_0^{\top}(t)\hat{\Lambda}^{\Delta}(t)F_0(t)+\hat{\Lambda}^{\Delta}(t)C(t))(R^{\Delta}(t)+F^{\top}(t)\hat{\Lambda}^{\Delta}(t)F(t)+F_0^{\top}(t)\hat{\Lambda}^{\Delta}(t)F_0(t))^{-1}\\(F^{\top}(t)\hat{\Lambda}^{\Delta}(t)D(t)+F_0^{\top}(t)\hat{\Lambda}^{\Delta}(t)D_0(t)+C^{\top}(t)\hat{\Lambda}^{\Delta}(t))
 		=0,\quad t\in\cup_{k=0}^{N-1}(t_{k},t_{k+1}),\\
 		\hat{\Lambda}^{\Delta}(t_{N})=P(t_{N-1}),\, 	\hat{\Lambda}^{\Delta}(t_{N-1})=\tilde{\Lambda}_{N-2}(t_{N-1}),\cdots,	\hat{\Lambda}^{\Delta}(t_{1})=\tilde{\Lambda}_{0}(t_{1}),
 		\end{cases}
 	\end{equation*}
 	on $(t_{N-1},t_{N}]$, using  the proposition 2.3 in \cite{Yong2017}, we have 
 	\begin{equation*}
 		\hat{\Lambda}^{\Delta}(t)\leq \Pi^{\Delta}(t).
 	\end{equation*}
 	As such 
 	\begin{equation*}
 		\hat{\Lambda}^{\Delta}(t_{N-1})=\tilde{\Lambda}_{N-2}(t_{N-1})\leq \hat{\Lambda}^{\Delta}(t_{N-1}+0)\leq \Pi^{\Delta}(t_{N-1}).
 	\end{equation*}
 	Then
 	\begin{equation*}
 		\hat{\Lambda}^{\Delta}(t)\leq \Pi^{\Delta}(t),\quad t\in(t_{N-2},t_{N-1}].
 	\end{equation*}
 	Inductively, we obtain
 	\begin{equation*}
 		\hat{\Lambda}^{\Delta}(t)\leq \Pi^{\Delta}(t),\quad t\in[0,T].
 	\end{equation*}
 	It then holds
 	\begin{equation*}
 		0\leq \tilde{\Lambda}_{k}(t)\leq \hat{\Lambda}^{\Delta}(t)\leq\Pi^{\Delta}(t),\, t\in(t_{k+1},T],\, 0\leq k \leq N-2.
 	\end{equation*}
 	Note that
 	\begin{equation*}
 		\begin{split}
 			|\Pi^{\Delta}(t)|\leq \Vert P \Vert_{C([0,T];\S^{d}_{+})}&+\int_{t}^{T}\bigg(\Vert Q\Vert_{C(\Delta[0,T];\S^{d}_{+})}+(2\Vert B\Vert_{C([0,T];\R^{d\times d})}+\Vert D\Vert^2_{C([0,T];\R^{d\times d})}\\&+\Vert D_0\Vert^2_{C([0,T];\R^{d\times d})})|\Pi^{\Delta}(s)|\bigg)ds\\\leq \Vert P \Vert_{C([0,T];\S^{d}_{+})}&+\Vert Q\Vert_{C(\Delta[0,T];\S^{d}_{+})}T+(2\Vert B\Vert_{C([0,T];\R^{d\times d})}\\&+\Vert D\Vert^2_{C([0,T];\R^{d\times d})}+\Vert D_0\Vert^2_{C([0,T];\R^{d\times d})})\int_{t}^{T}|\Pi^{\Delta}(s)|ds.
 		\end{split}
 	\end{equation*}
 	Applying Gronwall's inequality, we obtain
 	\begin{equation}\label{boundednessoflambda}
 		\Vert \Pi^{\Delta}\Vert_{C([0,T];\S^{d}_{+})}\leq K_{1}e^{K_{2}T}
 	\end{equation}
 	with 
 	\begin{equation*}
 		\begin{split}
 				&K_{1}:=\Vert P \Vert_{C([0,T];\S^{d}_{+})}+\Vert Q\Vert_{C(\Delta[0,T];\S^{d}_{+})} T,\\
 				&K_{2}:=2\Vert B\Vert_{C([0,T];\R^{d\times d})}+\Vert D\Vert^2_{C([0,T];\R^{d\times d})}+\Vert D_0\Vert^2_{C([0,T];\R^{d\times d})}.
 		\end{split}
 	\end{equation*}
 	Thus $(\hat{\Lambda}^{\Delta},\tilde{\Lambda}^{\Delta})$ is uniformly bounded in $\Delta$.
 \end{proof}
 \begin{lemma}\label{lemma:tildabeta}
 		Let conditions \eqref{condition:pd} and \eqref{condition:monotonicity} hold.  Then   
 	\begin{equation}\label{tildebeta:continuity}
 		\begin{split}
 			|\tilde{\Lambda}_{l}(t)-\tilde{\Lambda}_{l-1}(t)|\leq K\bigg[|P(t_{l})-P(t_{l-1})|&+\int_{t}^{T}(|Q(t_{l};s)-Q(t_{l-1};s)|\\&+|R(t_{l};s)-R(t_{l-1};s)|)ds\bigg],\,t\in[t_{l+1},t_{N}].
 		\end{split}
 	\end{equation}
 \end{lemma}
 \begin{proof}
 	Under the given conditions, all coefficients of the Lyapunov equation \eqref{Lyapunov:approximation} are bounded.
 	Now, comparing the equations for $\tilde{\Lambda}_{l}$ and $\tilde{\Lambda}_{l-1}$, we have
 	\begin{equation*}
 		\begin{split}
 			|\tilde{\Lambda}_{l}(t)-\tilde{\Lambda}_{l-1}(t)|&\leq |P(t_{l})-P(t_{l-1})|+K\int_{t}^{T}\bigg(|\tilde{\Lambda}_{l}(s)-\tilde{\Lambda}_{l-1}(s)|+|Q(t_{l};s)-Q(t_{l-1};s)|\\&+|R(t_{l};s)-R(t_{l-1};s)|\bigg)ds,\, t\in[t_{l+1},t_{N}].
 		\end{split}
 	\end{equation*}
 	Then \eqref{tildebeta:continuity}  follows from Gronwall's inequality.
 \end{proof}
 
  \begin{proposition}\label{pro:A1}
 	Under conditions \eqref{condition:pd} and \eqref{condition:monotonicity},  \eqref{riccati1} admits a unique solution 
 	$\Lambda\in C(\Delta[0,T];\S^{d}_{+})$.  Moreover,  for any $0\leq\tau_{1}\leq\tau_{2}\leq T$,
 	\begin{equation}\label{monotone:1}
 		\Lambda(\tau_{1};t)\leq\Lambda(\tau_{2};t),\quad 0\leq\tau_{1}\leq\tau_{2}\leq t\leq T.
 	\end{equation} 
  \end{proposition}
 \begin{proof}
 	\textit{Existence.}
 	 The proof is analogous to that of Theorem 6.4 in \cite{Yong2017}. For the reader's convenience, we provide the proof here. We will show that
 	\begin{equation*}
 		\lim\limits_{\Vert\Delta\Vert\rightarrow0}(|\hat{\Lambda}^{\Delta}(t)-\Lambda(t;t)|+|\tilde{\Lambda}^{\Delta}(\tau;t)-\Lambda(\tau;t)|)=0,
 	\end{equation*}
 	 holds uniformly in $(\tau,t)\in[0,T]\times[0,T]$ with $\Lambda$ satisfying \eqref{riccati1}. Given a partition $\Delta:0=t_{0}<\cdots<t_{N}=T$, Lemma \ref{lemma:uniformbound} implies 
 	\begin{equation*}
 		\begin{cases}
 			|\tilde{\Lambda}_{t}^{\Delta}(\tau;t)|\leq K,\quad t\neq\tau, \\
 			|(\hat{\Lambda}^{\Delta})'(t)|\leq K,\quad t\neq t_{k},\, 0\leq k\leq N
 		\end{cases}
 	\end{equation*}
 	with some uniform constant $K>0$ (independent of partition $\Delta$). It follows that 
 	\begin{equation*}
 		|\tilde{\Lambda}^{\Delta}(\tau;t)-\tilde{\Lambda}^{\Delta}(\tau;\overline{t})|\leq K|t-\overline{t}|,\,\ t,\overline{t}\in[0,T],\quad \tau\in[0,T].
 	\end{equation*}
 	On the other hand, by Lemma \ref{lemma:tildabeta}, for any $\tau,\overline{\tau}\in[0,T]$ with $|\tau-\overline{\tau}|<\min\limits_{0\leq k\leq N-1}|t_{k+1}-t_{k}|$,
 	we have 
 	\begin{equation*}
 		|\tilde{\Lambda}^{\Delta}(\tau;t)-\tilde{\Lambda}^{\Delta}(\overline{\tau};t)|\leq\tilde{\omega}(\Vert\Delta\Vert)
 	\end{equation*}
 	with $\tilde{\omega}$ being a modulus of continuity. Then, using the extension of Arzela-Ascoli Theorem, we have 
 	\begin{equation*} 		\lim\limits_{\Vert\Delta\Vert\rightarrow0}|\tilde{\Lambda}^{\Delta}(\tau;t)-\Lambda(\tau;t)|=0
 	\end{equation*}
 	for some continuous function $\Lambda$.
 	
 	Next, let $t\in(t_{k+1},t_{k+2}]$. Then 
 	\begin{equation*}
 		\begin{split}
 			|\hat{\Lambda}^{\Delta}(t)-\tilde{\Lambda}^{\Delta}(t;t)|&\leq |\hat{\Lambda}_{k+1}(t)-\hat{\Lambda}_{k+1}(t_{k+2})|+|\tilde{\Lambda}_{k+1}(t_{k+2})-\tilde{\Lambda}_{k}(t)|\\&\leq|\hat{\Lambda}_{k+1}(t)-\hat{\Lambda}_{k+1}(t_{k+2})|+|\tilde{\Lambda}_{k+1}(t_{k+2})-\tilde{\Lambda}_{k}(t_{k+2})|+|\tilde{\Lambda}_{k}(t_{k+2})-\tilde{\Lambda}_{k}(t)|\\&\leq \tilde{\omega}(\Vert\Delta\Vert),
 		\end{split}
 	\end{equation*}
 	where we have used the fact that $\hat{\Lambda}_{k+1}(t_{k+2})=\tilde{\Lambda}_{k+1}(t_{k+2})$ and  Lemma \ref{lemma:tildabeta}.
 	As such, we have 
 	\begin{equation*}
 		\lim\limits_{\Vert\Delta\Vert\rightarrow0}|\hat{\Lambda}^{\Delta}(t)-\Lambda(t;t)|=0.
 	\end{equation*}
 	Consequently, 
 	\begin{equation*}
 		\lim\limits_{\Vert\Delta\Vert\rightarrow0}|\Theta^{\Delta}(t)-\Theta(t)|=0.
 	\end{equation*}
 	Thus  we see  that $\Lambda$ satisfies \eqref{riccati1}.

 	\textit{Uniqueness.}
 	 Suppose  that there are two solutions $\Lambda^{1}$ and $\Lambda^{2}$. Then
 	$\overline{\Lambda}(\cdot;\cdot):=\Lambda^{1}(\cdot;\cdot)-\Lambda^{2}(\cdot;\cdot)$
 	satisfies
 	\begin{equation}\label{uq:ode}
 		\begin{cases}
 			\overline{\Lambda}'(\tau;t)+ \overline{\Lambda}(\tau;t)\Ac(t)+\Ac^{\top}(t)\overline{\Lambda}(\tau;t)+\Bc_1^{\top}(t)\overline{\Lambda}(\tau;t)\Bc_1(t)+\Bc_2^{\top}(t)\overline{\Lambda}(\tau;t)\Bc_2(t)\\+\sum_{i=1}^{K}\Cc_{i}(\tau;t)\overline{\Lambda}(t;t)\Cc_{-i}(\tau;t)=0,\quad 0\leq\tau\leq t\leq T, \\
 			\overline{\Lambda}(\tau;T)=0
 		\end{cases}
 	\end{equation}
 	for some continuous (matrix-valued) functions $\Ac$, $\Bc_{i}$, and $\Cc_{i}$ and some constant $K$.
 	We will use a standard contraction mapping argument to show that the above linear nonlocal ODE  admits a unique solution. Hence, it is necessary that 
 	\begin{equation*}
 		\overline{\Lambda}(\tau;t)=0,\quad 0\leq\tau\leq s\leq T,
 	\end{equation*}
 	which leads to the uniqueness of the solution of \eqref{riccati1}.
 	
 	Given $v_{1}(\cdot)$ and $v_{2}(\cdot)\in C([0,T];\R^{d\times d})$, let $V_{i}$ be the solution of 
 		\begin{equation*}
 		\begin{cases}
 		V_{i}'(\tau;t)+ V_{i}(\tau;t)\Ac(t)+\Ac^{\top}(t)V_{i}(\tau;t)+\Bc_1^{\top}(t)	V_{i}(\tau;t)\Bc_1(t)+\Bc_2^{\top}(t)	V_{i}(\tau;t)\Bc_2(t)\\+\sum_{i=1}^{K}\Cc_{i}(\tau;t)v_{i}(t)\Cc_{-i}(\tau;t)=0,\quad 0\leq\tau\leq t\leq T, \\
 			\overline{\Lambda}(\tau;T)=0.
 		\end{cases}
 	\end{equation*}
 		Simple calculation yields 
	\begin{equation*}
	\begin{cases}
		(V_{1}-V_{2})'(\tau;t)+ (V_{1}-V_{2})(\tau;t)\Ac(t)+\Ac^{\top}(t)(V_{1}-V_{2})(\tau;t)+\Bc_1^{\top}(t)	(V_{1}-V_{2})(\tau;t)\Bc_1(t)\\+\Bc_2^{\top}(t)	(V_{1}-V_{2})(\tau;t)\Bc_2(t)+\sum_{i=1}^{K}\Cc_{i}(\tau;t)(v_{1}-v_{2})(t)\Cc_{-i}(\tau;t)=0,\quad 0\leq\tau\leq t\leq T, \\
		(V_{1}-V_{2})(\tau;T)=0.
	\end{cases}
\end{equation*}
 	Therefore, 
 	\begin{equation*}
 		|V_{1}(\tau;t)-V_{2}(\tau;t)|\leq K\int_{t}^{T}\left(|V_{1}(\tau;s)-V_{2}(\tau;s)|+|v_{1}(s)-v_{2}(s)|\right)ds,
 	\end{equation*}
 	and the Gronwall's inequality yields  
 	\begin{equation*}
 		\Vert V_{1}(\tau;\cdot)-V_{2}(\tau;\cdot)\Vert_{C([T-\delta,T];\R^{d\times d})}\leq \tilde{K}\delta\Vert v_{1}(\cdot)-v_{2}(\cdot)\Vert_{C([T-\delta,T];\R^{d\times d })}.
 	\end{equation*}
 	Hence, it holds that
 	\begin{equation*}
 		\Vert \tilde{V}_{1}(\cdot)-\tilde{V}_{2}(\cdot)\Vert_{C([T-\delta,T];\R^{d\times d })}\leq \tilde{K}\delta\Vert v_{1}(\cdot)-v_{2}(\cdot)\Vert_{C([T-\delta,T];\R^{d\times d})},
 	\end{equation*}
 	where $\tilde{V}_{i}(t):=V_{i}(t;t)$. Thus if $\delta$ is small, the map $v(\cdot)\rightarrow\tilde{V}(\cdot)$ is contractive on $C([T-\delta,T];\R^{d\times d})$.
 	Then a standard argument yields a unique fixed point of $v(\cdot)\rightarrow\tilde{V}(\cdot)$ on $C([0,T];\R^{d\times d})$.  After obtaining the diagonal functions $V(t;t)$, the well-posedness of \eqref{uq:ode} can be established by solving a linear ODE. Therefore, we obtain the uniqueness of \eqref{uq:ode}.
 	
 	\textit{Property \eqref{monotone:1}}. Let $\Lambda$ be the unique solution of \eqref{riccati1}. Note that $\Lambda$ also satisfies \eqref{Lyapunov:1}. As $\tau\rightarrow Q(\tau;\cdot)$, $\tau\rightarrow R(\tau;\cdot)$ and $\tau\rightarrow P(\tau)$ are non-decreasing, we deduce that for any $0\leq\tau_{1}\leq\tau_{2}\leq t\leq T$,	$\Lambda(\tau_{1};t)\leq\Lambda(\tau_{2};t)$,
 	which completes the proof.
 	 \end{proof}
   
We next turn to investigate the equation \eqref{riccati2}, which has a similar structure to  \eqref{riccati1}. Thus, we can essentially apply the same method as above with minor modifications to establish the existence and uniqueness of the solution. Denote
 \begin{equation*}
 	\begin{cases}
 		\hat{Q}(\tau;t):=Q(\tau;t)+\overline{Q}(\tau;t),\, \hat{R}(\tau;t):=R(\tau;t)+\overline{R}(\tau;t),\,	\hat{B}(t):=B(t)+\overline{B}(t),\\
 		\hat{C}(t):=C(t)+\overline{C}(t),\,
 		\hat{D}(t):=D(t)+\overline{D}(t),\, \hat{F}(t):=F(t)+\overline{F}(t),\, \hat{F}_0(t):=F_0(t)+\overline{F}_0(t),\\	\hat{P}(\tau):=P(\tau)+\overline{P}(\tau),\,
 		\hat{\Theta}(t):=W^{-1}(t;t)Z^{\top}(t;t).
 	\end{cases}
 \end{equation*}
A straightforward computation shows that \eqref{riccati2} is equivalent to the following:
{\small
\begin{equation}
	\begin{cases}
			\beta'(\tau;t)+\beta(\tau;t)(\hat{B}(t)-\hat{C}(t)\hat{\Theta}(t))+(\hat{B}(t)-\hat{C}(t)\hat{\Theta}(t))^{\top}\beta(\tau;t)+(\hat{D}_0(t)-\hat{F}_0(t)\hat{\Theta}(t))^{\top}\beta(\tau;t)(\hat{D}_0(t)-\hat{F}_0(t)\hat{\Theta}(t))\\+(\hat{D}(t)-\hat{F}(t)\hat{\Theta}(t))^{\top}\Lambda(\tau;t)(\hat{D}(t)-\hat{F}(t)\hat{\Theta}(t))+\hat{Q}(\tau;t)+\hat{\Theta}^{\top}(t)\hat{R}(\tau;t)\hat{\Theta}(t)=0,\\
			\beta(\tau;T):=\hat{P}(\tau).
	\end{cases}
\end{equation}}
For a partition $\Delta:0=t_{0}<t_{1}<\cdots<t_{N-1}<t_{N}=T$, let us introduce the finite sequences 
\begin{equation*}
	\hat{\beta}_{k}(\cdot),\,\ 0\leq k\leq N-1, \,\,\ \tilde{\beta}_{l}(\cdot),\,\, 0\leq l\leq N-2,
\end{equation*}
where $\hat{\beta}_{k}$ defined on $(t_{k},t_{k+1}]$ satisfies the following Riccati equation, whose well-posedness is guaranteed by following the same proof of Theorem 7.2 in \cite{book:yong1999},
\begin{equation*}
	\begin{cases}
		\hat{\beta}_{k}'(t)+\hat{\beta}_{k}(t)\hat{B}(t)+\hat{B}^{\top}(t)\hat{\beta}_{k}(t)+\hat{D}_0^{\top}(t)\hat{\beta}_{k}(t)\hat{D}_0(t)+\hat{D}^{\top}(t)\Lambda(t_{k};t)\hat{D}(t)+\hat{Q}(t_{k};t)\\ -(\hat{D}^{\top}(t)\Lambda(t_{k};t)\hat{F}(t)+\hat{D}_0^{\top}(t)\hat{\beta}_{k}(t)\hat{F}_0(t)+\hat{\beta}_{k}(t)\hat{C}(t))(\hat{R}(t_{k};t)+\hat{F}^{\top}(t)\Lambda(t_{k};t)\hat{F}(t)+\hat{F}_0^{\top}(t)\hat{\beta}_{k}(t)\hat{F}_0(t))^{-1}\\(\hat{F}^{\top}(t)\Lambda(t_{k};t)\hat{D}(t)+\hat{F}_0^{\top}(t)\hat{\beta}_{k}(t)\hat{D}_0(t)+\hat{C}^{\top}(t)\hat{\beta}_{k}(t))=0,\quad t\in(t_{k},t_{k+1}),\\
		\hat{\beta}_{k}(t_{k+1})=\tilde{\beta}_{k}(t_{k+1})
	\end{cases}
\end{equation*}
with 
$\tilde{\beta}_{N-1}(t_{N})=\hat{P}(t_{N-1})$,
and $\tilde{\beta}_{l}$ defined on $(t_{l+1},t_{N}]$ satisfies the Lyapunov equation
{\small
\begin{equation}\label{Lyapunov:approximation2}
	\begin{cases}
		\tilde{\beta}_{l}'(t)+\tilde{\beta}_{l}(t)(\hat{B}(t)-\hat{C}(t)\hat{\Theta}^{\Delta}(t))+(\hat{B}(t)-\hat{C}(t)\hat{\Theta}^{\Delta}(t))^{\top}\tilde{\beta}_{l}(t)+(\hat{D}_0(t)-\hat{F}_0(t)\hat{\Theta}^{\Delta}(t))^{\top}\tilde{\beta}_{l}(t)(\hat{D}_0(t)-\hat{F}_0(t)\hat{\Theta}^{\Delta}(t))\\+(\hat{D}(t)-\hat{F}(t)\hat{\Theta}^{\Delta}(t))^{\top}\Lambda(t_{l};t)(\hat{D}(t)-\hat{F}(t)\hat{\Theta}^{\Delta}(t))+\hat{Q}(t_{l};t)+(\hat{\Theta}^{\Delta})^{\top}(t)\hat{R}(t_{l};t)\hat{\Theta}^{\Delta}(t)=0,\, t\in(t_{l+1},t_{N}),\\
		\tilde{\beta}_{l}(t_{N})=\hat{P}(t_{l}),
	\end{cases}
\end{equation}  }	
where, for $0\leq l\leq N-1$ and $t\in[t_{l},t_{l+1}]$,
{\small
\begin{equation*}
	\hat{\Theta}^{\Delta}(t)=\left[\hat{F}^{\top}(t)\Lambda(t_{l};t)\hat{F}(t)+\hat{F}_0^{\top}(t)\hat{\beta}_{l}(t)\hat{F}_0(t)+\hat{R}(t_{l};t)\right]^{-1}\left[\hat{F}^{\top}(t)\Lambda(t_{l};t)\hat{D}(t)+\hat{F}_{0}^{\top}(t)\hat{\beta}_{l}(t)\hat{D}_0(t)+\hat{C}^{\top}(t)\hat{\beta}_{l}(t)\right].
\end{equation*}}
Denote 
\begin{equation*}
	\begin{cases}
		\hat{Q}^{\Delta}(t)=\sum_{k=0}^{N-1}\hat{Q}(t_{k};t)I_{(t_{k},t_{k+1}]}(t),\, t\in[0,T],\\
		\hat{R}^{\Delta}(t)=\sum_{k=0}^{N-1}\hat{R}(t_{k};t)I_{(t_{k},t_{k+1}]}(t),\, t\in[0,T],\\
		\Lambda^{\Delta}(t)=\sum_{k=0}^{N-1}\Lambda(t_{k};t)I_{(t_{k},t_{k+1}]}(t),\, t\in[0,T],\\
		\hat{\beta}^{\Delta}(t)=\sum_{k=0}^{N-1}\hat{\beta}_{k}(t)I_{(t_{k},t_{k+1}]}(t),\, t\in[0,T].
	\end{cases}
\end{equation*}
We can then equivalently derive that
\begin{equation*}
	\begin{cases}
		(\hat{\beta}^{\Delta})'(t)+\hat{\beta}^{\Delta}(t)(\hat{B}(t)-\hat{C}(t)\hat{\Theta}^{\Delta}(t))+(\hat{B}(t)-\hat{C}(t)\hat{\Theta}^{\Delta}(t))^{\top}\hat{\beta}^{\Delta}(t)+\hat{Q}^{\Delta}(t)+(\hat{\Theta}^{\Delta})^{\top}(t)\hat{R}^{\Delta}(t)\hat{\Theta}^{\Delta}(t)\\+(\hat{D}(t)-\hat{F}(t)\hat{\Theta}^{\Delta}(t))^{\top}\Lambda^{\Delta}(t)(\hat{D}(t)-\hat{F}(t)\hat{\Theta}^{\Delta}(t))+(\hat{D}_0(t)-\hat{F}_0(t)\hat{\Theta}^{\Delta}(t))^{\top}\hat{\beta}^{\Delta}(t)(\hat{D}_0(t)-\hat{F}_0(t)\hat{\Theta}^{\Delta}(t))=0,\\
		t\in\cup_{k=0}^{N-1}(t_{k},t_{k+1}),\\
		\hat{\beta}^{\Delta}(t_{N})=\hat{P}(t_{N-1}),\, \hat{\beta}^{\Delta}(t_{N-1})=\tilde{\beta}_{N-2}(t_{N-1}),\cdots,\hat{\beta}^{\Delta}(t_{1})=\tilde{\beta}_{0}(t_{1}).
	\end{cases}
\end{equation*}
Let us introduce 
\begin{equation*}
	\tilde{\beta}^{\Delta}(\tau;t)=\sum_{k=0}^{N-2}\tilde{\beta}_{k}(t)I_{(t_{k+1},t_{k+2}]}(\tau),\,\, 0\leq\tau\leq t\leq T,
\end{equation*}
with the extension on $[0,T]\times[0,T]$ that $\tilde{\beta}^{\Delta}(\tau;t)=\tilde{\beta}^{\Delta}(\tau;\tau)$, $t\in[0,\tau]$. Note that $t\rightarrow\tilde{\beta}^{\Delta}(\tau;t)$ is continuous on $[0,T]$, whereas $t\rightarrow\hat{\beta}^{\Delta}(t)$ might have jumps at $t=t_{k}$.

Similarly, one can show that 
\begin{equation}\label{convergence2}
	\lim\limits_{\Vert\Delta\Vert\rightarrow0}(|\hat{\beta}^{\Delta}(t)-\beta(t;t)|+|\tilde{\beta}^{\Delta}(\tau;t)-\beta(\tau;t)|)=0
\end{equation} 	
 holds uniformly in $(\tau,t)\in[0,T]\times[0,T]$	with $\beta$ satisfying \eqref{riccati2}.

Furthermore, following the same proofs of Lemma \ref{lemma:uniformbound} and Lemma \ref{lemma:tildabeta} with minor modifications, we can obtain the next result.

\begin{lemma}\label{lemma:uniformbound2}
	Let conditions \eqref{condition:pd} and \eqref{condition:monotonicity} hold.  Then  $(\hat{\beta}^{\Delta},\tilde{\beta}^{\Delta})$ is uniformly bounded in $\Delta$, and it holds that
	\begin{equation}\label{tildebeta:continuity2}
		\begin{split}
			|\tilde{\beta}_{l}(t)-\tilde{\beta}_{l-1}(t)|\leq K\bigg[|\hat{P}(t_{l})-\hat{P}(t_{l-1})|+\int_{t}^{T}(|\hat{Q}(t_{l};s)-\hat{Q}(t_{l-1};s)|\\+|\Lambda(t_{l};s)-\Lambda(t_{l-1};s)|+|\hat{R}(t_{l};s)-\hat{R}(t_{l-1};s)|)ds\bigg]
		\end{split}
	\end{equation}
\end{lemma}

The next result also holds, whose proof is similar to that of Proposition \ref{pro:A1}.
\begin{proposition}\label{pro:A2}
	Let conditions \eqref{condition:pd} and \eqref{condition:monotonicity} hold.  Then \eqref{riccati2} admits a unique solution	$\beta\in C(\Delta[0,T];\S^{d}_{+})$.
\end{proposition}
Although equation \eqref{linearode1} is nonlocal, it can be tackled by a standard contraction mapping argument.
 \begin{proposition}\label{pro:A3}
 If there exists a unique solution $(\Lambda,\beta)\in C(\Delta[0,T];\S^{d}_{+})\times C(\Delta[0,T];\S^{d}_{+})$ to \eqref{riccati1}-\eqref{riccati2},  then \eqref{linearode1} admits a unique solution $\gamma\in C(\Delta[0,T];\R^{d})$.
 \end{proposition}
 \begin{proof}
 	Note that after obtaining $\Lambda$ and $\beta$, \  $\gamma$ satisfies
 	\begin{equation*}
 		\begin{cases}
 			\gamma'(\tau;t)+\overline{\Ac}^{\top}(t)\gamma(\tau;t)+\overline{\Bc}^{\top}(\tau;t)\gamma(t;t)+\overline{\Cc}(\tau;t)=0,\\
 			\gamma(\tau;T)=p(\tau)+\overline{p}(\tau)
 		\end{cases}
 	\end{equation*}
 	for some continuous functions $\overline{\Ac}$, $\overline{\Bc}$ and $\overline{\Cc}$.
 	
 	We will apply the fixed point argument here. Given $v_{1}(\cdot)$ and $v_{2}(\cdot)\in C([0,T];\R^{d})$, let $\gamma_{i}(\tau;t)$ be the solution of 
 	\begin{equation*}
 			\begin{cases}
 			\gamma_{i}'(\tau;t)+\overline{\Ac}^{\top}(t)\gamma_{i}(\tau;t)+\overline{\Bc}^{\top}(\tau;t)v_{i}(t)+\overline{\Cc}(\tau;t)=0,\\
 			\gamma(\tau;T)=p(\tau)+\overline{p}(\tau).
 		\end{cases}
 	\end{equation*}
 	Basic calculations yield that  
 	\begin{equation*}
 			\begin{cases}
 			(\gamma_{1}(\tau;t)-\gamma_{2}(\tau;t))'+\overline{\Ac}^{\top}(t)(\gamma_{1}(\tau;t)-\gamma_{2}(\tau;t))+\overline{\Bc}^{\top}(\tau;t)(v_{1}(t)-v_{2}(t))=0,\\
 			\gamma_{1}(\tau;T)-\gamma_{2}(\tau;T)=0.
 		\end{cases}
 	\end{equation*}
 	As a result,
 	\begin{equation*}
 		|\gamma_{1}(\tau;t)-\gamma_{2}(\tau;t)|\leq K\int_{t}^{T}\left(|\gamma_{1}(\tau;s)-\gamma_{2}(\tau;s)|+|v_{1}(s)-v_{2}(s)|\right)ds.
 	\end{equation*}
 	Using the Gronwall's inequality,  we have

 	\begin{equation}
 		\Vert \gamma_{1}(\tau;\cdot)-\gamma_{2}(\tau;\cdot)\Vert_{C([T-\delta,T];\R^{d})}\leq \tilde{K}\delta\Vert v_{1}(\cdot)-v_{2}(\cdot)\Vert_{C([T-\delta,T];\R^{d})}.
 	\end{equation}
 	It then holds that 
 	\begin{equation}
 		\Vert \tilde{\gamma}_{1}(\cdot)-\tilde{\gamma}_{2}(\cdot)\Vert_{C([T-\delta,T];\R^{d})}\leq \tilde{K}\delta\Vert v_{1}(\cdot)-v_{2}(\cdot)\Vert_{C([T-\delta,T];\R^{d})},
 	\end{equation}
 	where $\tilde{\gamma}_{i}(t):=\gamma_{i}(t;t)$. Thus, if $\delta$ is small, the mapping $v(\cdot)\rightarrow\tilde{\gamma}(\cdot)$ is a contraction mapping on $C([T-\delta,T];\R^{d})$, which guarantees a unique fixed point of $v(\cdot)\rightarrow\tilde{\gamma}(\cdot)$ on $C([0,T];\R^{d})$.  After obtaining the diagonal functions $\gamma(t;t)$, the well-posedness of \eqref{linearode1} can be established by solving a linear ODE.
 \end{proof}
We are now ready to prove Proposition \ref{Pro:riccati}.
\begin{proof}[The proof of Propositions \ref{Pro:riccati}]
	Based on Propositions \ref{pro:A1}, \ref{pro:A2} and \ref{pro:A3}, Equations \eqref{riccati1}-\eqref{linearode1} admit a unique solution. Furthermore,
	it is easy to see that the solution to \eqref{linearode2} can be determined by taking the integral. Thus, the system \eqref{riccati1}-\eqref{linearode2} admits a unique solution. Moreover, we can obtain from Propositions \ref{pro:A1} and \ref{pro:A2} that $(\Lambda(\tau;t),\beta(\tau;t))\in\S^{d}_{+}\times\S^{d}_{+}$ for each $(\tau;t)\in\Delta[0,T]$. Thus,  $(U(t;t),W(t;t))$ lies in $\S^{m}_{>+}\times\S^{m}_{>+}$ for all $t\in[0,T]$, which completes the proof.
\end{proof}

\subsection{The proof of Lemma \ref{lemma:wellposedness:ricattiexample}}\label{appendixB}
For notational convenience, let us consider the transform $\tilde{\Lambda}(\tau;t):=\Lambda(\tau;t)+\frac{q}{2}\lambda(t-\tau)$. Then
\begin{equation}\label{riccatiB}
	\begin{cases}
		\tilde{\Lambda}'(\tau;t)-2k\tilde{\Lambda}(\tau;t)+2\lambda(t-\tau)\tilde{\Lambda}^{2}(t;t)-4\tilde{\Lambda}(\tau;t)\tilde{\Lambda}(t;t)+(\frac{\eta}{2}+kq)\lambda(t-\tau)-\frac{q}{2}\lambda'(t-\tau)=0,\\
		\tilde{\Lambda}(\tau;T)=\frac{(c+q)}{2}\lambda(T-\tau).
	\end{cases}
\end{equation}
We apply the fixed point argument here. Given $v_{i}(\cdot;\cdot)\in C(\Delta[0,T];\R)$ for $i=1,2$, let $\tilde{\Lambda}_{i}$ be the solution of 
\begin{equation}\label{fixedpointb}
	\begin{cases}
		\tilde{\Lambda}_{i}'(\tau;t)-2k\tilde{\Lambda}_{i}(\tau;t)+2\lambda(t-\tau)v^{2}_{i}(t;t)-4v_{i}(\tau;t)v_{i}(t;t)+(\frac{\eta}{2}+kq)\lambda(t-\tau)-\frac{q}{2}\lambda'(t-\tau)=0,\\ 
		\tilde{\Lambda}_{i}(\tau;T)=\frac{(c+q)}{2}\lambda(T-\tau).
	\end{cases}
\end{equation}
 Note that 
 {\small\begin{equation*}
 	\begin{split}
 	0&=\tilde{\Lambda}_{i}'(\tau;t)-2k\tilde{\Lambda}_{i}(\tau;t)+2\lambda(t-\tau)v_{i}^{2}(t;t)-4v_{i}(\tau;t)v_{i}(t;t)\\&+(\frac{\eta}{2}+kq)\lambda(t-\tau)-\frac{q}{2}\lambda'(t-\tau)\\&=\tilde{\Lambda}_{i}'(\tau;t)-2k\tilde{\Lambda}_{i}(\tau;t)-\frac{2}{\lambda(t-\tau)}v_{i}^{2}(\tau;t)+2\left(\frac{1}{\sqrt{\lambda(t-\tau)}}v_{i}(\tau;t)-\sqrt{\lambda(t-\tau)}v_{i}(t;t)\right)^{2}\\&
 	+(\frac{\eta}{2}+kq)\lambda(t-\tau)-\frac{q}{2}\lambda'(t-\tau),
 	\end{split}
 \end{equation*}}
 we have 
 {\small\begin{equation*}
 	\begin{split}
 		\left(e^{2k(T-t)}\tilde{\Lambda}_{i}(\tau;t)\right)'&=e^{2k(T-t)}\tilde{\Lambda}_{i}'(\tau;t)-2k\left(e^{2k(T-t)}\tilde{\Lambda}_{i}(\tau;t)\right)\\
 		&=\frac{2e^{2k(T-t)}}{\lambda(t-\tau)}v_{i}^{2}(\tau;t)-2e^{2k(T-t)}\left(\frac{1}{\sqrt{\lambda(t-\tau)}}v_{i}(\tau;t)-\sqrt{\lambda(t-\tau)}v_{i}(t;t)\right)^{2}\\&-e^{2k(T-t)}\left[(\frac{\eta}{2}+kq)\lambda(t-\tau)-\frac{q}{2}\lambda'(t-\tau)\right].
 	\end{split}
 \end{equation*}}
 One can conclude that 
{\small\begin{equation*}
 	\begin{split}
 		e^{2k(T-t)}\tilde{\Lambda}_{i}(\tau;t)=&\frac{(c+q)}{2}\lambda(T-\tau)+\int_{t}^{T}\bigg[2e^{2k(T-s)}\left(\frac{1}{\sqrt{\lambda(s-\tau)}}v_{i}(\tau;s)-\sqrt{\lambda(s-\tau)}v_{i}(s;s)\right)^{2}\\&+e^{2k(T-s)}\left[(\frac{\eta}{2}+kq)\lambda(s-\tau)-\frac{q}{2}\lambda'(s-\tau)\right]-\frac{2e^{2k(T-s)}}{\lambda(s-\tau)}v_{i}^{2}(\tau;s)\bigg]ds\\&
 		\leq \frac{(c+q)}{2}\lambda(T-\tau)+\int_{t}^{T}\bigg[2e^{2k(T-s)}\left(\frac{1}{\sqrt{\lambda(s-\tau)}}v_{i}(\tau;s)-\sqrt{\lambda(s-\tau)}v_{i}(s;s)\right)^{2}\\&+e^{2k(T-s)}\left[(\frac{\eta}{2}+kq)\lambda(s-\tau)-\frac{q}{2}\lambda'(s-\tau)\right]\bigg]ds.
 	\end{split}
 \end{equation*}}
 Hence, it follows that 
 {\small\begin{equation*}
 	\begin{split}
 		\tilde{\Lambda}_{i}(\tau;t)&\leq\frac{(c+q)}{2}	e^{-2k(T-t)}\lambda(T-\tau)+\int_{t}^{T}\bigg[2e^{2k(t-s)}\left(\frac{1}{\sqrt{\lambda(s-\tau)}}v_{i}(\tau;s)-\sqrt{\lambda(s-\tau)}v_{i}(s;s)\right)^{2}\\&+e^{2k(t-s)}\left[(\frac{\eta}{2}+kq)\lambda(s-\tau)-\frac{q}{2}\lambda'(s-\tau)\right]\bigg]ds
 		\\&\leq C_{1}+\int_{t}^{T}8C_{2}e^{2k(t-s)}\Vert v_{i}\Vert^{2}_{C(\Delta[0,T];\R)}ds+\int_{t}^{T}e^{2k(t-s)}\left[C_{3}+kC_{4}\right]ds\\
 		&\leq C_{5}(1+\frac{1}{k}\Vert v_{i}\Vert^{2}_{C(\Delta[0,T];\R)}),
 	\end{split}
 \end{equation*}}
 where 
 {\small\begin{equation*}
 	\begin{split}
 		&C_{1}:=\frac{c+q}{2}\Vert\lambda\Vert_{C([0,T];\R)},\quad
 		C_{2}:=\max(\Vert\frac{1}{\lambda}\Vert_{C([0,T];\R)},\Vert\lambda\Vert_{C([0,T];\R)}),\\
 		&C_{3}:=\frac{\eta}{2}\Vert\lambda\Vert_{C([0,T];\R)}+\frac{q}{2}\Vert\lambda'\Vert_{C([0,T];\R)},\quad
 		C_{4}:=q\Vert\lambda\Vert_{C([0,T];\R)},\\
 		&C_{5}:=\max(C_{1}+\frac{C_{3}}{2}+\frac{C_{4}}{2},4C_{2}).
 	\end{split}
 \end{equation*}}
  Similarly, we can obtain 
{\small\begin{equation*}
 	\begin{split}
 		\tilde{\Lambda}_{i}(\tau;t)&\geq-\int_{t}^{T}\frac{2e^{2k(t-s)}}{\lambda(s-\tau)}v_{i}^{2}(\tau;s)ds\geq -\frac{1}{k}\Vert\frac{1}{\lambda}\Vert_{C([0,T];\R)}\Vert v_{i}\Vert^{2}_{C(\Delta[0,T];\R)}.
 	\end{split}
 \end{equation*}}
 Finally, we have 
 \begin{equation*}
 	\Vert\tilde{\Lambda}_{i}\Vert_{C(\Delta[0,T];\R)}\leq C(1+\frac{1}{k}\Vert v_{i}\Vert^{2}_{C(\Delta[0,T];\R)}),
 \end{equation*}
 where $C$ is some constant depending on the discount function $\lambda$, parameters $c$, $q$ and $\eta$.
 If $k$ is large such that $k>4C^{2}$ and $\Vert v_{i}\Vert_{C(\Delta[0,T];\R)}\leq 2C$, we  obtain  
 \begin{equation}\label{uniformbound}
 		\Vert\tilde{\Lambda}_{i}\Vert_{C(\Delta[0,T];\R)}\leq C(1+\frac{1}{k}4C^{2})\leq 2C.
 \end{equation}
By \eqref{fixedpointb}, simple calculation yields 
{\small\begin{equation*}
	\begin{cases}
		(\tilde{\Lambda}_{1}-\tilde{\Lambda}_{2})'(\tau;t)-2k	(\tilde{\Lambda}_{1}-\tilde{\Lambda}_{2})(\tau;t)+2\lambda(t-\tau)(v_{1}(t;t)-v_{2}(t;t))(v_{1}(t;t)+v_{2}(t;t))\\-4(v_{1}(\tau;t)v_{1}(t;t)-v_{2}(\tau;t)v_{2}(t;t))=0,\\
	(\tilde{\Lambda}_{1}-\tilde{\Lambda}_{2})(\tau;T)=0.
	\end{cases}
\end{equation*}}
 Using the uniform bound in \eqref{uniformbound}, we  have 
 \begin{equation*}
 	\begin{split}
 		|(\tilde{\Lambda}_{1}-\tilde{\Lambda}_{2})(\tau;t)|&\leq 2k\int_{t}^{T}	|(\tilde{\Lambda}_{1}-\tilde{\Lambda}_{2})(\tau;s)|ds+K\int_{t}^{T}|v_{1}(s;s)-v_{2}(s;s)|ds\\&+K\int_{t}^{T}|v_{1}(\tau;s)-v_{2}(\tau;s)|ds.
 	\end{split}
 \end{equation*}
Then, the Gronwall's inequality yields 
 \begin{equation*}
 	\Vert(\tilde{\Lambda}_{1}-\tilde{\Lambda}_{2})\Vert_{C(\T[T-\delta,T];\R)}\leq K\delta\Vert(v_{1}-v_{2})\Vert_{C(\T[T-\delta,T];\R)},
 \end{equation*}
 where $\T[a,b]$ represents the trapezoidal region $\Delta[0,b]\cap([0,b]\times[a,b])$ and $K$ is a constant that  may differ from line to line. Thus, if $\delta$ is small, there exists a unique solution of \eqref{riccatiB} on the trapezoidal region $\T[T-\delta,T]$. For a  general $\Delta[0,T]$, we can divide it into $\T[T-\delta,T]$, $\T[T-2\delta,T-\delta]$, $\cdots$. One can conclude that there exists a unique solution of \eqref{riccatiB} on $\Delta[0,T]$, and thus \eqref{ricattiexample} admits a unique solution as desired.

{ \subsection{The proof of Proposition \ref{newsect:lemma}}\label{appendixC}

\paragraph{Step 1: Construction of the fixed point mapping} For any \( \phi \in \Xi^{(1+\alpha)}_{[0,T]} \), consider the mapping from \( \phi \) to \( \tilde{\Phi} := \Gamma(\phi) \), where \( \tilde{\Phi} \) is the solution of the following linear local parabolic PDE parameterized by \( \tau \in [0, T] \):
{\small\begin{equation}\label{appc:fixedpoint}
    \begin{split}
        &\partial_t \tilde{\Phi}(\tau; t, x) + b(x) \partial_x \tilde{\Phi}(\tau; t, x) + \frac{\sigma_0^2}{2} \partial_{xx} \tilde{\Phi}(\tau; t, x)  + (C + \overline{C}) \rho \left( \frac{1}{2} \partial_x \phi(t; t, x) \right) \partial_x \phi(\tau; t, x) \\
        &\quad + R(\tau; t) \rho^2 \left( \frac{1}{2} \partial_x \phi(t; t, x) \right) + \overline{R}(\tau; t) r \left( \rho \left( \frac{1}{2} \partial_x \phi(t; t, x) \right) \right)  + \overline{Q}(\tau; t) q(x) = 0, \quad 0 \leq \tau \leq t < T, \, x \in \mathbb{R}, \\
        &\tilde{\Phi}(\tau; T, x) = \overline{P}(\tau) p(x), \quad (\tau, x) \in [0, T] \times \mathbb{R}.
    \end{split}
\end{equation}}
To show the well-posedness of the mapping \(\Gamma\), we first estimate \( |\rho\left( \frac{1}{2} \partial_x \phi(\cdot; \cdot, \cdot) \right)|^{(\alpha)}_{[\tau\vee S,T]\times\R} \) for some \( S \in [0,T) \). For any \( t, t' \in [\tau\vee S,T] \) (without loss of generality, assume \( t < t' \)), $x\in\R$, we have
{\small\begin{equation*}
    \begin{split}
    &\left|\rho \left( \frac{1}{2} \partial_x \phi(t; t, x) \right) - \rho \left( \frac{1}{2} \partial_x \phi(t'; t', x) \right)\right| \\&\leq \left|\rho \left( \frac{1}{2} \partial_x \phi(t; t, x) \right) - \rho \left( \frac{1}{2} \partial_x \phi(t; t', x) \right)\right| + \left|\rho \left( \frac{1}{2} \partial_x \phi(t; t', x) \right) - \rho \left( \frac{1}{2} \partial_x \phi(t'; t', x) \right)\right| \\
    &\leq \frac{1}{2}|\rho'|^{(0)} \left( \left(\sup_{t \in [0,T]}\langle \partial_x \phi(t; \cdot, \cdot) \rangle_{s,[t\vee S,T] \times \mathbb{R}}^{\left( \frac{\alpha}{2} \right)}\right) |t - t'|^{\frac{\alpha}{2}} + \left(\sup_{t \in [0,T]} |\partial_{x} \partial_{\tau} \phi(t; \cdot, \cdot)|^{(0)}_{[t\vee S,T] \times \mathbb{R}}\right) |t - t'| \right),
\end{split}
\end{equation*}}
which implies that 
{\small\begin{equation}\label{appc:est1}
\begin{split}
    \left\langle \rho \left( \frac{1}{2} \partial_x \phi(\cdot; \cdot, \cdot) \right)\right\rangle_{s,[\tau\vee S,T] \times \mathbb{R}}^{\left( \frac{\alpha}{2} \right)}&\leq K(T,\alpha,\epsilon_0)|C+\overline{C}|\left( \sup_{t \in [0,T]}\langle \partial_x \phi(t; \cdot, \cdot) \rangle_{s,[t\vee S,T] \times \mathbb{R}}^{\left( \frac{\alpha}{2} \right)}  + \sup_{t \in [0,T]} |\partial_{x} \partial_{\tau} \phi(t; \cdot, \cdot)|^{(0)}_{[t\vee S,T] \times \mathbb{R}} \right)\\&\leq K(T,\alpha,\epsilon_0)|C+\overline{C}| \|\phi\|_{[S, T]}^{(1+\alpha)},
\end{split}
\end{equation}}
where we have used \eqref{newsect:rhobound}, and \( K(T,\alpha,\epsilon_0) \) is a constant depending on \( T \), \( \alpha \), and \( \epsilon_0 \).

Likewise, for $x,x'\in\R$, $0<|x-x'|\leq1$ and $t\in[\tau\vee S,T]$, we have
{\small 
\begin{equation*}
     \begin{split}
    &\left|\rho \left( \frac{1}{2} \partial_x \phi(t; t, x) \right) - \rho \left( \frac{1}{2} \partial_x \phi(t; t, x') \right)\right| \leq \frac{1}{2}|\rho'|^{(0)}  \sup_{t \in [0,T]}\langle \partial_x \phi(t; \cdot, \cdot) \rangle_{x,[t\vee S,T] \times \mathbb{R}}^{\left( \alpha \right)} |x - x'|^{\alpha},
\end{split}
\end{equation*}}
which implies that 
{\small\begin{equation}\label{appc:est2}
\begin{split}
   \left \langle \rho \left( \frac{1}{2} \partial_x \phi(\cdot; \cdot, \cdot) \right)\right\rangle_{x,[\tau\vee S,T] \times \mathbb{R}}^{\left(\alpha \right)}&\leq K(\epsilon_0)|C+\overline{C}|\left( \sup_{t \in [0,T]}\langle \partial_x \phi(t; \cdot, \cdot) \rangle_{x,[t\vee S,T] \times \mathbb{R}}^{\left( \alpha \right)} \right)\\&\leq K(\epsilon_0)|C+\overline{C}| \|\phi\|_{[S, T]}^{(1+\alpha)}.
\end{split}
\end{equation}}
Note that \(\rho\) has linear growth that $|\rho(x)| \leq |\rho'|^{(0)} |x| + |\rho(0)|$ and it is clear that \(\rho(0)\) is independent of \(C + \overline{C}\) from \eqref{newsect:rhoimplict}. Combining this with \eqref{appc:est1}-\eqref{appc:est2}, we obtain that
\begin{equation}\label{appc:rho}
   \left|\rho\left( \frac{1}{2} \partial_x \phi(\cdot; \cdot, \cdot) \right)\right|^{(\alpha)}_{[\tau\vee S,T]\times\R}\leq K(T,\alpha,\epsilon_0)|C+\overline{C}| \|\phi\|_{[S, T]}^{(1+\alpha)}+|\rho(0)|.
\end{equation}
The estimate of \( | r \left( \rho \left( \frac{1}{2} \partial_x \phi(\cdot; \cdot, \cdot) \right) \right)|^{(\alpha)}_{[\tau\vee S,T]\times\R} \) depends on whether condition (1) or (2) in Assumption \ref{newsect:asspcoef} is satisfied by \( r \).
\begin{itemize}
    \item In the case of \( r \) satisfying condition (1):  
By similar arguments as before, we have
\[
\left| r \left( \rho \left( \frac{1}{2} \partial_x \phi(\cdot; \cdot, \cdot) \right) \right)\right|^{(\alpha)}_{[\tau\vee S,T]\times\R} \leq K(T,\alpha,\epsilon_0,|r'|^{(0)}) |C+\overline{C}|  \|\phi\|_{[S, T]}^{(1+\alpha)} +|r'|^{(0)}|\rho(0)|+|r(0)|.
\]
\item In the case of \( r \) satisfying condition (2):
Here, the quadratic growth condition on \( r \) leads to the estimate  
\[
\left| r \left( \rho \left( \frac{1}{2} \partial_x \phi(\cdot; \cdot, \cdot) \right) \right)\right|^{(\alpha)}_{[\tau\vee S,T]\times\R} \leq K(T,\alpha,\epsilon_0,\tilde{K}) |C+\overline{C}|^2  \left(\|\phi\|_{[S, T]}^{(1+\alpha)}\right)^2 + 4\tilde{K}|\rho(0)|^2+\beta,
\]
 where \( \tilde{K} \) and \( \beta \) are the constants specified in condition (2) of Assumption \ref{newsect:asspcoef}.
\end{itemize}
Then, we can obtain
{\small\begin{align*}
    g(\tau;t,x)&:= (C + \overline{C}) \rho \left( \frac{1}{2} \partial_x \phi(t; t, x) \right) \partial_x \phi(\tau; t, x) 
         + R(\tau; t) \rho^2 \left( \frac{1}{2} \partial_x \phi(t; t, x) \right) \\&\quad+ \overline{R}(\tau; t) r \left( \rho \left( \frac{1}{2} \partial_x \phi(t; t, x) \right) \right)  + \overline{Q}(\tau; t) q(x)
\end{align*}}
satisfies 
{\small
\begin{equation*}
    \begin{split}
        &|g(\tau;\cdot,\cdot)|^{\alpha}_{[\tau\vee S,T]\times\R}\\&\leq |C+\overline{C}| \left|\rho\left( \frac{1}{2} \partial_x \phi(\cdot; \cdot, \cdot) \right)\right|^{(\alpha)}_{[\tau\vee S,T]\times\R}|\partial_x \phi(\tau; \cdot, \cdot)|^{(\alpha)}_{[\tau\vee S,T]\times\R} \\& +  |R(\tau;\cdot)|^{\alpha}_{[\tau\vee S,T]}\left(\left|\rho\left( \frac{1}{2} \partial_x \phi(\cdot; \cdot, \cdot) \right)\right|^{(\alpha)}_{[\tau\vee S,T]\times\R}\right)^2\\&+|\overline{R}(\tau;\cdot)|^{\alpha}_{[\tau\vee S,T]} \left| r \left( \rho \left( \frac{1}{2} \partial_x \phi(\cdot; \cdot, \cdot) \right) \right)\right|^{(\alpha)}_{[\tau\vee S,T]\times\R}+|\overline{Q}(\tau;\cdot)|^{(\alpha)}_{[\tau\vee S,T]}|q|^{(\alpha)}_{\R},
    \end{split}
\end{equation*}
} which implies that 
\begin{itemize}
    \item In the case of \( r \) satisfying condition (1), we have for any $S\in[0,T)$
\begin{equation*}
    \begin{split}
[g]^{(\alpha)}_{[S,T]}=\sup_{\tau\in[0,T]}|g(\tau;\cdot,\cdot)|^{\alpha}_{[\tau\vee S,T]\times\R}\leq K\left(|C+\overline{C}|^2\left(\|\phi\|_{[S, T]}^{(1+\alpha)}\right)^2+ |C+\overline{C}|\left(\|\phi\|_{[S, T]}^{(1+\alpha)}\right) +1 \right),
    \end{split}
\end{equation*}
where \( K \) is a constant depending on \( T \), \( \alpha \), \( \epsilon_0 \), \( |r'|^{(0)} \), \( \rho(0) \), \( r(0) \), and the coefficients \( R(\cdot;\cdot) \), \( \overline{R}(\cdot;\cdot) \), \( \overline{Q}(\cdot;\cdot) \), and \( q(\cdot) \).
    \item In the case of \( r \) satisfying condition (2), we have for any $S\in[0,T)$
    \begin{equation*}
    \begin{split}[g]^{(\alpha)}_{[S,T]}=\sup_{\tau\in[0,T]}|g(\tau;\cdot,\cdot)|^{\alpha}_{[\tau\vee S,T]\times\R}\leq K\left(|C+\overline{C}|^2\left(\|\phi\|_{[S, T]}^{(1+\alpha)}\right)^2+ |C+\overline{C}|\left(\|\phi\|_{[S, T]}^{(1+\alpha)}\right) +1 \right),
    \end{split}
\end{equation*}
 where \( K \) is a constant depending on \( T \), \( \alpha \), \( \epsilon_0 \), \( \tilde{K} \), \( \beta \), \( \rho(0) \), and the coefficients \( R(\cdot;\cdot) \), \( \overline{R}(\cdot;\cdot) \), \( \overline{Q}(\cdot;\cdot) \), and \( q(\cdot) \).
\end{itemize}
Therefore, in both cases, we arrive at
\begin{equation}\label{appc:est3}
      [g]^{(\alpha)}_{[S,T]} \leq K\left(|C+\overline{C}|^2\left(\|\phi\|_{[S, T]}^{(1+\alpha)}\right)^2+ |C+\overline{C}|\left(\|\phi\|_{[S, T]}^{(1+\alpha)}\right) +1 \right).
\end{equation}
for some constant $K$ independent of $C+\overline{C}$ and $S$.
Combining this with
\[
[\overline{P}(\cdot) p(\cdot)]^{(2+\alpha)}_{[0,T]} = |\overline{P}|_{[0,T]}^{(0)} |p|_{\mathbb{R}}^{(2+\alpha)} < \infty,
\]  
the classical theory of parabolic PDE guarantees the existence of a unique solution \( \tilde{\Phi} \in \Theta^{(2+\alpha)}_{[0,T]} \) to \eqref{appc:fixedpoint}; see, e.g., \cite{AF1964,ladyzenskaja_linear_nodate}. Moreover, this solution satisfies the estimate  
\[
[\tilde{\Phi}]^{(2+\alpha)}_{[0,T]} \leq K \left( [g]^{(\alpha)}_{[0,T]} + |\overline{P}|_{[0,T]}^{(0)} |p|_{\mathbb{R}}^{(2+\alpha)} \right).
\]
Furthermore, it is clear that $\tilde{\Phi}_{\tau}(\tau;t,x):=\partial_{\tau}\tilde{\Phi}(\tau;t,x)$ satisfies the following equation
{\small\begin{equation*}
    \begin{split}
        &\partial_t \tilde{\Phi}_{\tau}(\tau; t, x) + b(x) \partial_x \tilde{\Phi}_{\tau}(\tau; t, x) + \frac{\sigma_0^2}{2} \partial_{xx} \tilde{\Phi}_{\tau}(\tau; t, x)  + (C + \overline{C}) \rho \left( \frac{1}{2} \partial_x \phi(t; t, x) \right) \partial_x \partial_{\tau}\phi(\tau; t, x) \\
        &\quad + \partial_{\tau}R(\tau; t) \rho^2 \left( \frac{1}{2} \partial_x \phi(t; t, x) \right) + \partial_{\tau}\overline{R}(\tau; t) r \left( \rho \left( \frac{1}{2} \partial_x \phi(t; t, x) \right) \right)  + \partial_{\tau}\overline{Q}(\tau; t) q(x) = 0, \quad 0 \leq \tau \leq t < T, \, x \in \mathbb{R}, \\
        &\tilde{\Phi}_{\tau}(\tau; T, x) = \overline{P}'(\tau) p(x), \quad (\tau, x) \in [0, T] \times \mathbb{R}.
    \end{split}
\end{equation*}}
Similarly, we can show that \( g_{\tau}(\tau; t, x) := \partial_\tau g(\tau; t, x) \) satisfies the following estimate for any \( S \in [0,T) \):  
\begin{equation}\label{appc:est4}
    \begin{split}
        [g_{\tau}]^{(\alpha)}_{[S,T]} \leq K\left( |C+\overline{C}|^2 \left(\|\phi\|_{[S, T]}^{(1+\alpha)}\right)^2 + |C+\overline{C}| \left(\|\phi\|_{[S, T]}^{(1+\alpha)}\right) + 1 \right),
    \end{split}
\end{equation}
where \( K \) is a constant independent of \( C+\overline{C} \) and \( S \).
This, together with the estimate  
\[
[\overline{P}'(\cdot) p(\cdot)]^{(2+\alpha)}_{[0,T]} = |\overline{P}'|_{[0,T]}^{(0)} |p|_{\mathbb{R}}^{(2+\alpha)} < \infty,
\]  
implies that \( \tilde{\Phi} \in \Xi^{(2+\alpha)}_{[0,T]} \), and consequently, \( \tilde{\Phi} \in \Xi^{(1+\alpha)}_{[0,T]} \) as well. Moreover, by the inequality (5.7) of Proposition 5.1 in \cite{wang_backward_2019}, we obtain  
\begin{equation*}
\begin{split}
     |\tilde{\Phi}(\tau; \cdot, \cdot)|_{[\tau\vee S, T] \times \mathbb{R}}^{(1+\alpha)} + \left|\tilde{\Phi}_{\tau}(\tau; \cdot, \cdot)\right|_{[\tau\vee S,T] \times \mathbb{R}}^{(1+\alpha)}&\leq (|\overline{P}(\tau)|+|\overline{P}'(\tau)|)|p|^{(1+\alpha)}_{\R}+K(T-S)^{\frac{\alpha}{2}}\bigg(|g(\tau;\cdot,\cdot)|^{(\alpha)}_{[\tau\vee S,T]\times\R}\\&+|g_{\tau}(\tau;\cdot,\cdot)|^{(\alpha)}_{[\tau\vee S,T]\times\R}+(|\overline{P}(\tau)|+|\overline{P}'(\tau)|)|p|^{(2+\alpha)}_{\R}\bigg),
\end{split}
\end{equation*}
which, together with estimates \eqref{appc:est3}-\eqref{appc:est4}, leads to  \begin{equation}\label{appc:est5}
    \|\tilde{\Phi}\|_{[0, T]}^{(1+\alpha)}\leq K^*\left( |C+\overline{C}|^2 \left(\|\phi\|_{[0, T]}^{(1+\alpha)}\right)^2 + |C+\overline{C}| \left(\|\phi\|_{[0, T]}^{(1+\alpha)}\right) + 1 \right),
\end{equation}
where \( K^* \) is a constant  
independent of \( C+\overline{C} \). If \( |C+\overline{C}| \) is sufficiently small such that  
\[
|C+\overline{C}|\leq \frac{1}{4K^*} \quad \text{and} \quad \|\phi\|_{[0, T]}^{(1+\alpha)}\leq 2K^*.
\]
Then, from \eqref{appc:est5}, we deduce  
\[
\|\tilde{\Phi}\|_{[0, T]}^{(1+\alpha)}\leq \frac{7}{4}K^*\leq 2K^*.  
\]
Consequently, we have established that the mapping \( \Gamma \) is well-defined and maps the set  
\[
\mathcal{S}:=\left\{\phi\in \Xi^{(1+\alpha)}_{[0,T]}: \|\phi\|_{[0, T]}^{(1+\alpha)}\leq 2K^*\right\}
\]
into itself.

\paragraph{Step 2: Contraction mapping}
Let \( \phi_1, \phi_2 \in \mathcal{S} \) and define \( \Phi_i = \Gamma(\phi_i) \), where \( i = 1, 2 \). It follows that \( \Delta \Phi := \Phi_1 - \Phi_2 \) and \( \Delta \Phi_{\tau} := \partial_{\tau} \Delta \Phi \) satisfy the following equations:  
{\small
\begin{equation*}
    \begin{split}
        &\partial_t \Delta\Phi(\tau; t, x) + b(x) \partial_x \Delta\Phi(\tau; t, x) + \frac{\sigma_0^2}{2} \partial_{xx} \Delta\Phi(\tau; t, x)  + \Delta g(\tau;t,x) = 0, \quad 0 \leq \tau \leq t < T, \, x \in \mathbb{R}, \\
        &\Delta\Phi(\tau; T, x) =0, \quad (\tau, x) \in [0, T] \times \mathbb{R},\\
        &\partial_t \Delta\Phi_{\tau}(\tau; t, x) + b(x) \partial_x \Delta\Phi_{\tau}(\tau; t, x) + \frac{\sigma_0^2}{2} \partial_{xx} \Delta\Phi_{\tau}(\tau; t, x)  + \Delta g_{\tau}(\tau;t,x) = 0, \quad 0 \leq \tau \leq t < T, \, x \in \mathbb{R}, \\
        &\Delta\Phi_{\tau}(\tau; T, x) =0, \quad (\tau, x) \in [0, T] \times \mathbb{R},
    \end{split}
\end{equation*}}
where the terms \( \Delta g(\tau; t, x) \) and \( \Delta g_{\tau}(\tau; t, x) \) are defined by
\begin{equation*}
    \begin{split}
        &\Delta g(\tau;t,x):= (C + \overline{C}) \rho \left( \frac{1}{2} \partial_x \phi_1(t; t, x) \right) \partial_x \phi_1(\tau; t, x)  + R(\tau; t) \rho^2 \left( \frac{1}{2} \partial_x \phi_1(t; t, x) \right) \\&\qquad\qquad\quad+ \overline{R}(\tau; t) r \left( \rho \left( \frac{1}{2} \partial_x \phi_1(t; t, x) \right) \right)-(C + \overline{C}) \rho \left( \frac{1}{2} \partial_x \phi_2(t; t, x) \right) \partial_x \phi_2(\tau; t, x) \\
        &\qquad\qquad\quad - R(\tau; t) \rho^2 \left( \frac{1}{2} \partial_x \phi_2(t; t, x) \right) -\overline{R}(\tau; t) r \left( \rho \left( \frac{1}{2} \partial_x \phi_2(t; t, x) \right) \right), \\
        &\Delta g_{\tau}(\tau;t,x) :=\partial_{\tau}\Delta g(\tau;t,x).
    \end{split}
\end{equation*}
In what follows, we show that for any \( S \in [0,T) \), the estimates hold that  
\begin{align}
    &\sup_{\tau\in[0,T]}|\Delta g(\tau;\cdot,\cdot)|^{(\alpha)}_{[\tau\vee S,T]\times\mathbb{R}} \leq K \bigg(\sup_{\tau\in[0,T]} |\Delta \phi(\tau;\cdot,\cdot)|^{(1+\alpha)}_{[\tau\vee S,T]\times\mathbb{R}} + \sup_{\tau\in[0,T]} |\Delta \phi_{\tau}(\tau;\cdot,\cdot)|^{(1+\alpha)}_{[\tau\vee S,T]\times\mathbb{R}} \bigg), \label{appc:est6} \\
    &\sup_{\tau\in[0,T]}|\Delta g_{\tau}(\tau;\cdot,\cdot)|^{(\alpha)}_{[\tau\vee S,T]\times\mathbb{R}} \leq K \bigg(\sup_{\tau\in[0,T]} |\Delta \phi(\tau;\cdot,\cdot)|^{(1+\alpha)}_{[\tau\vee S,T]\times\mathbb{R}} + \sup_{\tau\in[0,T]} |\Delta \phi_{\tau}(\tau;\cdot,\cdot)|^{(1+\alpha)}_{[\tau\vee S,T]\times\mathbb{R}} \bigg), \label{appc:est7}
\end{align}
where \( K \) is a constant independent of \( S \). We only provide the proof of \eqref{appc:est6}, and the proof of \eqref{appc:est7} follows in a similar fashion and is omitted.

Note that 
{\small\begin{equation}\label{appc:g1}
    \begin{split}
        &\Delta g(\tau;t,x):= (C + \overline{C}) \left(\rho \left( \frac{1}{2} \partial_x \phi_1(t; t, x) \right)-\rho \left( \frac{1}{2} \partial_x \phi_2(t; t, x) \right)\right) \partial_x \phi_1(\tau; t, x)\\&\qquad\qquad\quad + (C + \overline{C}) \rho \left( \frac{1}{2} \partial_x\phi_2(t; t, x) \right) \partial_x \Delta\phi(\tau; t, x)\\&\qquad\qquad\quad+R(\tau; t) \left(\rho \left( \frac{1}{2} \partial_x \phi_1(t; t, x) \right)-\rho \left( \frac{1}{2} \partial_x \phi_2(t; t, x) \right)\right)\left(\rho \left( \frac{1}{2} \partial_x \phi_1(t; t, x) \right)+\rho \left( \frac{1}{2} \partial_x \phi_2(t; t, x) \right)\right) \\&\qquad\qquad\quad+ \overline{R}(\tau; t)\left( r \left( \rho \left( \frac{1}{2} \partial_x \phi_1(t; t, x) \right) \right)-r \left( \rho \left( \frac{1}{2} \partial_x \phi_2(t; t, x) \right) \right)\right),
    \end{split}
\end{equation}}
and
{\small\begin{equation*}
    \begin{split}
        \rho \left( \frac{1}{2} \partial_x \phi_1(t; t, x) \right)-\rho \left( \frac{1}{2} \partial_x \phi_2(t; t, x) \right)=\int_{0}^1 \rho' \left( \frac{\theta}{2} \partial_x \phi_1(t; t, x) +\frac{1-\theta}{2} \partial_x \phi_2(t; t, x)\right) \frac{1}{2}\partial_x \Delta\phi(\tau; t, x)d\theta.
    \end{split}
\end{equation*}}
Because \( r \) is three times continuously differentiable, it follows from \eqref{newsect:rho'} that \( \rho \) is twice continuously differentiable. Consequently, \( \rho' \) is locally Lipschitz continuous.  Therefore, given that \( \phi_1, \phi_2 \in \mathcal{S} \), a similar argument to the proof of \eqref{appc:rho} shows that  
\begin{equation*}
    \left| \rho' \left( \frac{\theta}{2} \partial_x \phi_1(\cdot; \cdot, \cdot) + \frac{1-\theta}{2} \partial_x \phi_2(\cdot; \cdot, \cdot) \right) \right|^{(\alpha)}_{[\tau\vee S,T]\times\R} \leq K,
\end{equation*}
where \( K \) is a constant independent of \( S \) and \( \theta \), but dependent on \( K^* \). Furthermore, we obtain 
{\small\begin{equation}\label{appc:g2}
\begin{split}
        &\left|\rho \left( \frac{1}{2} \partial_x \phi_1(t; t, x) \right)-\rho \left( \frac{1}{2} \partial_x \phi_2(t; t, x) \right) \right|^{(\alpha)}_{[\tau\vee S,T]\times\R} 
        \\&\leq \frac{1}{2}\int_{0}^1   \left| \rho' \left( \frac{\theta}{2} \partial_x \phi_1(\cdot; \cdot, \cdot) + \frac{1-\theta}{2} \partial_x \phi_2(\cdot; \cdot, \cdot) \right) \right|^{(\alpha)}_{[\tau\vee S,T]\times\R} \left|\partial_x\Delta\phi(\tau;\cdot, \cdot)\right|^{(\alpha)}_{[\tau\vee S,T]\times\R}d\theta
        \\&\leq K |\Delta \phi(\tau;\cdot,\cdot)|^{(1+\alpha)}_{[\tau\vee S,T]\times\mathbb{R}}.
\end{split}
\end{equation}}
Similarly, we can deduce that  
\begin{equation}\label{appc:g3}
    \begin{split}
        \left|r \left( \rho \left( \frac{1}{2} \partial_x \phi_1(t; t, x) \right) \right)-r \left( \rho \left( \frac{1}{2} \partial_x \phi_2(t; t, x) \right) \right) \right|^{(\alpha)}_{[\tau\vee S,T]\times\R} \leq K |\Delta \phi(\tau;\cdot,\cdot)|^{(1+\alpha)}_{[\tau\vee S,T]\times\mathbb{R}},
    \end{split}
\end{equation}
 by observing that the function \( h(x) := r \circ \rho(x) \) also has a locally Lipschitz continuous derivative.  Then,  combining \eqref{appc:g1}-\eqref{appc:g3} and utilizing the fact that \( \phi_1, \phi_2 \in \mathcal{S} \), we obtain  
\begin{equation*}
    \begin{split}
      \sup_{\tau\in[0,T]}  \left|\Delta g(\tau;\cdot,\cdot)\right|^{(\alpha)}_{[\tau\vee S,T]\times\R}\leq K\sup_{\tau\in[0,T]}|\Delta \phi(\tau;\cdot,\cdot)|^{(1+\alpha)}_{[\tau\vee S,T]\times\mathbb{R}},
    \end{split}
\end{equation*}
which establishes the desired estimate \eqref{appc:est6}.

Then, applying inequality (5.7) from Proposition 5.1 in \cite{wang_backward_2019} together with estimates \eqref{appc:est6}-\eqref{appc:est7}, we derive  
\begin{equation*}
\begin{split}
   \|\Delta\Phi\|_{[T-\delta, T]}^{(1+\alpha)} &=\sup_{\tau\in[0,T]}|\Delta\Phi(\tau; \cdot, \cdot)|_{[\tau\vee (T-\delta), T] \times \mathbb{R}}^{(1+\alpha)} + \sup_{\tau\in[0,T]} \left|\Delta\Phi_{\tau}(\tau; \cdot, \cdot)\right|_{[\tau\vee (T-\delta),T] \times \mathbb{R}}^{(1+\alpha)}\\&\leq K\delta^{\frac{\alpha}{2}}\bigg(\sup_{\tau\in[0,T]}|\Delta g(\tau;\cdot,\cdot)|^{(\alpha)}_{[\tau\vee (T-\delta),T]\times\R}+\sup_{\tau\in[0,T]}|\Delta g_{\tau}(\tau;\cdot,\cdot)|^{(\alpha)}_{[\tau\vee (T-\delta),T]\times\R}\bigg)\\
   &\leq  K\delta^{\frac{\alpha}{2}}\bigg(\sup_{\tau\in[0,T]}|\Delta \phi(\tau;\cdot,\cdot)|^{(1+\alpha)}_{[\tau\vee (T-\delta),T]\times\R}+\sup_{\tau\in[0,T]}|\Delta \phi_{\tau}(\tau;\cdot,\cdot)|^{(1+\alpha)}_{[\tau\vee (T-\delta),T]\times\R}\bigg)\\
   &\leq K\delta^{\frac{\alpha}{2}}\|\Delta\phi\|_{[T-\delta, T]}^{(1+\alpha)}
\end{split}
\end{equation*}
where  $K$ is a constant that  may differ from line to line. Thus, if $\delta$ is small, there exists a unique solution of \eqref{newsect:nonlocalpde} on the  region $\T[T-\delta,T]\times \R$.  Then, we may repeat the above procedure, to get a unique solution on  $\T[T-2\delta,T-\delta]\times \R$, and we eventually can obtain the existence and unique solution to the nonlocal PDE \eqref{newsect:nonlocalpde}. It is clear that the solution, denoted by \( \Phi \), belongs to the space \( \mathcal{S} \). Combining this with standard estimates for local linear parabolic PDEs, we deduce that \( \|\Phi\|_{[0, T]}^{(2+\alpha)} < \infty \). Consequently, it follows readily that the function \( V(\tau;t,\mu) := \Phi(\tau;t,\overline{\mu}) \) satisfies Assumption \ref{assumption:V}, moreover, the derived strategy \( \hat{\alpha} \) is admissible and is hence a closed-loop equilibrium strategy by Theorem \ref{corollary:verification}.}
\vskip 10pt

\ \\
\noindent
\textbf{Acknowledgements}:  Zongxia  Liang  is supported by the National Natural Science Foundation of China under grant no. 12271290. Xiang Yu and Keyu Zhang are supported by the Hong Kong RGC General Research Fund (GRF) under grant no. 15306523 and grant no. 15211524 and by the Research Centre for Quantitative Finance at the Hong Kong Polytechnic University under grant no. P0042708.

  {\small
	\bibliographystyle{siam}
	\bibliography{ref}
	}
	
\end{document}